\documentclass[12pt,reqno,final]{amsart}
\usepackage[margin=1in]{geometry}
\usepackage{amsmath,amssymb,amsthm,color,url}
\numberwithin{figure}{section}
\usepackage{graphicx}

\DeclareMathOperator{\M}{M}

\DeclareMathOperator{\re}{Re}

\DeclareMathOperator{\spn}{span}
\DeclareMathOperator{\codim}{codim}
\DeclareMathOperator{\rank}{rank}
\DeclareMathOperator{\Img}{Im}

\numberwithin{equation}{section}
\newtheorem{theorem}{Theorem}[section]

\newtheorem{lemma}[theorem]{Lemma}

\newtheorem{corollary}[theorem]{Corollary}

\newtheorem*{question*}{Question}

\theoremstyle{definition}
\newtheorem{definition}[theorem]{Definition}

\newtheorem*{definition*}{Definition}

\theoremstyle{remark}

\newtheorem*{remark*}{Remark}

\usepackage[normalem]{ulem}

\title{Another proof of the $U^4(\mathbf{F}_p^n)$-inverse theorem}
\author{Sarah Peluse}
\address{Department of Mathematics, Stanford University, 450 Jane Stanford Way, Building 380, Stanford, CA 94305, USA}
\email{speluse@stanford.edu}

\begin{document}

\begin{abstract}
  We give a new and shorter proof of a quantitative inverse theorem
  for the Gowers $U^4$-norm in the setting of high-dimensional vector
  spaces over finite fields.
\end{abstract}

\maketitle

\section{Introduction}\label{sec:intro}

In his proof of Szemer\'edi's theorem with reasonable quantitative
bounds, Gowers~\cite{Gowers1998,Gowers2001} introduced an important
family of norms: for any finite abelian group $G$, integer $s\geq 2$,
and $f:G\to\mathbf{C}$, the \textit{$U^s(G)$-norm} of $f$ is defined
by
  \begin{equation*}
    \|f\|_{U^s(G)}:=\big(\mathbf{E}_{x,h_1,\dots,h_s\in G}\Delta_{h_1}\cdots\Delta_{h_s}f(x)\big)^{1/2^s},
  \end{equation*}
  where
  $\mathbf{E}_{x,h_1,\dots,h_s\in
    G}:=|G|^{-(s+1)}\sum_{x,h_1,\dots,h_s\in G}$ denotes the average
  and $\Delta_{h}f(x):=f(x)\overline{f(x+h)}$ denotes the
  multiplicative discrete derivative. Gowers proved local inverse
  theorems for these norms on cyclic groups, showing that if
  $f:\mathbf{Z}/N\mathbf{Z}\to\mathbf{C}$ is bounded and
  $\|f\|_{U^s(\mathbf{Z}/N\mathbf{Z})}\geq\delta$, then there is a
  partition of $\mathbf{Z}/N\mathbf{Z}$ into long arithmetic
  progressions of almost equal length on which $f$ has large
  correlation on average with (possibly different) polynomial phases
  $e_N(P(x))$ of degree $\deg{P}\leq s-1$, where
  $e_N(z):=e^{2\pi i z/N}$. This result was sufficient to prove the
  first reasonable bounds in Szemer\'edi's theorem for progressions of
  length at least four when combined with the fact, also first shown
  by Gowers, that the $U^{s}(\mathbf{Z}/N\mathbf{Z})$-norm controls the count of
  $(s+1)$-term arithmetic progressions in subsets of $\mathbf{Z}/N\mathbf{Z}$.

  There are numerous applications in analytic number theory and
  additive combinatorics that require global inverse theorems, which
  say that a bounded function with large $U^s(G)$-norm must have large
  correlation over the entire group with a single structured
  object. It is easy to show that
  $\|f\|_{U^2(G)}=\|\widehat{f}\|_{\ell^4}$, and thus to deduce a
  global inverse theorem for the $U^2(G)$-norm on any finite abelian
  group: if $\|f\|_{U^2(G)}\geq\delta$ and $f$ is $1$-bounded, then
  there exists a character $\xi\in\widehat{G}$ such that
  $|\mathbf{E}_{x\in G}f(x)\xi(x)|\geq\delta^2$. Proving global
  inverse theorems for the $U^s(G)$-norms when $s\geq 3$ is
  significantly more difficult. The first global inverse theorems for
  the $U^s(\mathbf{Z}/N\mathbf{Z})$-norms were proven by Green and
  Tao~\cite{GreenTao2008II} in the case $s=3$ and in full generality
  by Green, Tao, and
  Ziegler~\cite{GreenTaoZiegler2011,GreenTaoZiegler2012}. In addition
  to cyclic groups, the other main setting of interest has been
  high-dimensional vector spaces over a fixed finite field. The first
  global inverse theorems for the $U^s(\mathbf{F}_p^n)$-norms were
  proven by Green and Tao~\cite{GreenTao2008II} and
  Samorodnitsky~\cite{Samorodnitsky2007} in the case $s=3$ and in full
  generality by Bergelson, Tao, and
  Ziegler~\cite{BergelsonTaoZiegler2010,TaoZiegler2010} in high
  characteristic ($p\geq s$) and by Tao and
  Ziegler~\cite{TaoZiegler2012} in low characteristic.

  The aforementioned proofs of global inverse theorems for the
  $U^s$-norms yield no quantitative bounds when $s\geq 4$. The first
  general quantitative versions of these theorems were proven by
  Manners~\cite{Manners2018} in the cyclic group setting and by Gowers
  and Mili\'cevi\'c~\cite{GowersMilicevic2017,GowersMilicevic2020} in
  the high-characteristic finite field vector space setting. The
  current best bounds in the cyclic group setting are due to
  Raghavan~\cite{Raghavan2025} for $s=3$, Leng~\cite{Leng2023} for
  $s=4$, and Leng, Sah, and Sawhney~\cite{LengSahSawhney2024I} for
  $s\geq 5$, while the current best bounds in the finite field vector
  space setting are due to Gowers, Green, Manners, and
  Tao~\cite{GowersGreenMannersTao2025,GowersGreenMannersTao2026} for
  $s=3$ and Mili\'cevi\'c~\cite{Milicevic2024} for $s=4$; the
  bounds of Gowers and Mili\'cevi\'c from~\cite{GowersMilicevic2020} still
  remain the strongest when $s\geq 5$.

  The proofs of inverse theorems for the $U^s$-norms when $s\geq 4$
  are all quite long and involved, especially those giving
  quantitative bounds. The purpose of this paper is to present a new
  and shorter proof of the inverse theorem for the
  $U^4(\mathbf{F}_p^n)$-norm, which also yields quantitative bounds
  (though not as strong as those obtained by Mili\'cevi\'c).
    
  \begin{theorem}\label{thm:main}
    Fix a prime $p\geq 5$. There exists a constant $C=C_p\geq 1$ such
    that the following holds. If $f:\mathbf{F}_p^n\to\mathbf{C}$ is a
    nonzero $1$-bounded function satisfying
    $\|f\|_{U^4(\mathbf{F}_p^n)}\geq\delta$, then there exists a
    polynomial $P:\mathbf{F}_p^n\to\mathbf{F}_p$ of degree at most $3$
    for which
    \begin{equation*}
      \left|\mathbf{E}_{x\in\mathbf{F}_p^n}f(x)e_p(P(x))\right|\geq \exp\left(-C\delta^{-C}\right).
    \end{equation*}
  \end{theorem}

  The proof of Theorem~\ref{thm:main} we present is almost
  self-contained, with the exception of our use of sufficiently strong
  bounds for the partition rank of trilinear forms in terms of their
  analytic rank (which have been known since work of Haramaty and
  Shpilka~\cite{HaramatyShpilka2010}), results belonging purely to the
  $U^2$- or $U^3$-theory such as the inverse theorem for the
  $U^3(\mathbf{F}_p^n)$-norm with polynomial
  bounds~\cite{GowersGreenMannersTao2026}, and basic facts such as the
  Gowers--Cauchy--Schwarz inequality. We give new and shorter proofs
  of some results already in the literature, most notably one of
  Kazhdan and
  Ziegler~\cite{KazhdanZiegler2017,KazhdanZiegler2019,KazhdanZiegler2020}
  on extending approximate quadratics from high-rank quadratic level
  sets. Further, in an attempt to make this paper more accessible,
  some results that are very minor variants of existing ones in the
  literature have been given full proofs. If these proofs of existing
  or almost existing ingredients were omitted, the proof of
  Theorem~\ref{thm:main} would require only fifteen pages.

  \subsection*{Asymptotic notation} In addition to the standard
  asymptotic notation $O$ and $\Omega$, we will also use
  Vinogradov's notation $\ll$ and $\gg$ throughout the paper. This
  means that, for any two positive quantities $X$ and $Y$, we will
  write $X=O(Y)$, $Y=\Omega(X)$, $X\ll Y$, and $Y\gg X$ to mean that
  $X\leq CY$ for some absolute constant $C>0$. When any of $O$,
  $\Omega$, $\ll$, or $\gg$ appears with a subscript, this means that
  the implied constant $C$ may depend on the parameters in the
  subscript. We will often write $O(1)$ or $\Omega(1)$ to represent a
  quantity that is bounded above or below, respectively, by an
  absolute constant.
  
  \section*{Acknowledgments}
  The author is partially supported by NSF grant DMS-2516641 and a
  Sloan Research Fellowship, and thanks Dan Altman, Ben Green, James
  Leng, and Tamar Ziegler for helpful conversations.

  \section{Preliminaries and an outline of the
    proof}\label{sec:prelims}

  Like previous proofs of the $U^s(G)$-inverse theorems when
  $s\geq 3$, the proof of Theorem~\ref{thm:main} requires classifying
  certain ``approximately structured'' functions. To describe these
  functions, we will need a few definitions.

  Let $G$ and $G'$ be abelian groups, $\phi:G\to G'$, and $h\in G$. We
  define the \textit{additive discrete derivative}
  $\partial_h\phi:G\to G'$ by $\partial_h\phi(x):=\phi(x)-\phi(x+h)$
  and, given $h_1,\dots,h_s\in G$, the \textit{$s$-fold additive
    discrete derivative} $\partial_{h_1,\dots,h_s}\phi:G\to G'$ by
  $\partial_{h_1,\dots,h_s}\phi:=\partial_{h_1}\cdots\partial_{h_s}\phi$. We
  will, analogously, denote the \textit{$s$-fold multiplicative
    discrete derivative} operator $\Delta_{h_1}\cdots\Delta_{h_s}$ by
  $\Delta_{h_1,\dots,h_s}$. Observe that $\partial_{h_1,\dots,h_s}$
  and $\Delta_{h_1,\dots,h_s}$ do not depend on the ordering of
  $h_1,\dots,h_s$, and also that if $\Phi(x)=\xi(\phi(x))$ for some
  character $\xi\in\widehat{G'}$, then
  $\Delta_{h_1,\dots,h_s}\Phi(x)=\xi(\partial_{h_1,\dots,h_s}\phi(x))$.
  
  It is a simple exercise to show that if
  $\phi:\mathbf{F}_p^n\to\mathbf{F}_p$ satisfies
  $\partial_{h_1,\dots,h_s}\phi(x)=0$ for all
  $x,h_1,\dots,h_s\in\mathbf{F}_p^n$, then $\phi$ is a polynomial of
  degree at most $s-1$. The bulk of this paper will be concerned with
  the study of ``approximate polynomials'', which are functions
  $\phi$ satisfying more relaxed conditions:
  $\partial_{h_1,\dots,h_s}\phi(x)=0$ may only hold a small fraction of the time and
  the domain of $\phi$ may be an arbitrary subset having no subgroup structure
  at all. To that end, for any $x,h_1,\dots,h_s\in G$, we define the
  associated \textit{(first-order) $s$-dimensional cube} to be
  \begin{equation*}
    \square(x;h_1,\dots,h_s):=\left\{x+\omega\cdot(h_1,\dots,h_s):\omega\in\{0,1\}^s\right\}\subset G;
  \end{equation*}
  $2$-dimensional cubes are usually called \textit{additive
    quadruples}. An $s$-dimensional cube is \textit{nondegenerate} if
  $|\square(x;h_1,\dots,h_s)|=2^s$ and is \textit{degenerate}
  otherwise. Observe that, when $G$ is finite, the number of
  $s$-dimensional cubes in any $A\subset G$ is
  $\|1_{A}\|_{U^s(G)}^{2^s}|G|^{s+1}$. A map $\phi:A\to G'$ is said to
  \textit{respect} an $s$-dimensional cube
  $\square(x;h_1,\dots,h_s)\subset A$ if
  $\partial_{h_1,\dots,h_s}\phi(x)=0$, i.e., if
  \begin{equation*}
      \sum_{\omega\in\{0,1\}^s}(-1)^{|\omega|}\phi(x+\omega\cdot(h_1,\dots,h_s))=0,
    \end{equation*}
    where $|\omega|$ denotes the number of $1$'s in $\omega$.

    Now, we can define approximate polynomials of bounded degree.
    \begin{definition}
      Let $G$ and $G'$ be abelian groups with $G$ finite, $A\subset G$, and
      $\phi:A\to G'$. We say that $\phi$ is a
      \textit{$\delta$-approximate polynomial of degree at most $s-1$
        on $A$} if $\partial_{h_1,\dots,h_s}\phi(x)=0$ for at least
      $\delta\|1_A\|^{2^s}_{U^s(G)}|G|^{s+1}$
      $(s+1)$-tuples $(x,h_1,\dots,h_s)\in G^{s+1}$ for which
      $\square(x;h_1,\dots,h_s)\subset A$, i.e., if the proportion of $s$-dimensional
      cubes in $A$ respected by $\phi$ is at least $\delta$.
    \end{definition}
    We will only encounter approximate polynomials of degree at most $1$ or
    $2$ in the proof of Theorem~\ref{thm:main}. For brevity, we will
    write ``$\phi$ is $\delta$-approximately linear'' to mean that
    $\phi$ is a $\delta$-approximate polynomial of degree at most $1$
    and, analogously, ``$\phi$ is $\delta$-approximately quadratic'' to
    mean that $\phi$ is a $\delta$-approximate polynomial of degree at
    most $2$.
    
    The first step in the proof of Theorem~\ref{thm:main} is to show
    that it follows from an inverse theorem for approximately
    quadratic functions, which says that such functions must agree
    with a genuine polynomial map of degree at most $2$ a positive
    proportion of the time.
    \begin{theorem}\label{thm:approxpoly}
      Fix a prime $p\geq 5$. If $\phi:\mathbf{F}_p^n\to\mathbf{F}_p^n$
      is $\delta$-approximately quadratic, then there exists a
      polynomial map $Q:\mathbf{F}_p^n\to\mathbf{F}_p^n$ of degree at
      most $2$ such that $\phi(x)=Q(x)$ for
      $\gg_p\exp(-O_p(\delta^{-O_p(1)}))p^n$ elements
      $x\in\mathbf{F}_p^n$.
    \end{theorem}
    This is analogous to the fact that polynomial bounds in the
    inverse theorem for approximately linear functions, now proven by
    Gowers, Green, Manners, and
    Tao~\cite{GowersGreenMannersTao2025,GowersGreenMannersTao2026},
    imply polynomial bounds in the $U^3(\mathbf{F}_p)$-inverse
    theorem:
    \begin{theorem}\label{thm:approxlinear}
      If $\phi:\mathbf{F}_p^n\to\mathbf{F}_p^n$ is
      $\delta$-approximately linear, then there exists an
      affine-linear $L:\mathbf{F}_p^n\to\mathbf{F}_p^n$ such that
      $\phi(x)=L(x)$ for $\gg_p\delta^{O_p(1)}p^n$ elements
      $x\in\mathbf{F}_p^n$.
    \end{theorem}
    \begin{theorem}\label{thm:U3inv}
      If $f:\mathbf{F}_p^n\to\mathbf{C}$ is $1$-bounded and
      $\|f\|_{U^3(\mathbf{F}_p^n)}\geq\delta$, then there exists a
      polynomial $Q:\mathbf{F}_p^n\to\mathbf{F}_p$ of degree at most
      $2$ such that
      \begin{equation*}
        \left|\mathbf{E}_{x\in \mathbf{F}_p^n}f(x)e_p(Q(x))\right|\gg_p\delta^{O_p(1)}.
      \end{equation*}
    \end{theorem}
     The reverse implication also holds (see,
     e.g.,~\cite{GreenTao2010}), and one can analogously show that
     Theorem~\ref{thm:main} implies Theorem~\ref{thm:approxpoly}
     (though we will not do so here).
    
    Instead of deriving their $U^4(\mathbf{F}_p^n)$-inverse theorem
    from an inverse theorem for approximately quadratic functions,
    Gowers and Mili\'cevi\'c~\cite{GowersMilicevic2017} derive it from
    an inverse theorem for approximately bilinear functions. Proving
    that Theorem~\ref{thm:approxpoly} implies Theorem~\ref{thm:main}
    takes about the same amount of work as the analogous deduction in
    the work of Gowers and Mili\'cevi\'c and does not contain any real new ideas (we adapt an argument of Manners~\cite{Manners2018} from
    his work in cyclic groups). Our proof of
    Theorem~\ref{thm:approxpoly} and their proof of an inverse theorem
    for approximately bilinear functions have the same broad
    structure, pioneered by Gowers in~\cite{Gowers2001}: prove that
    the functions respect almost all ``second-order'' configurations
    in a dense subset, use this to show that it suffices to classify
    functions respecting almost all first-order configurations in some
    algebraically structured set, and then solve the classification
    problem. It turns out that one can execute these steps with less
    effort for approximately quadratic functions than for
    approximately bilinear functions. In particular, $3$-dimensional
    cubes and $U^3$-dual functions are easier to work with than the
    ``$4$-arrangements'' and ``mixed convolutions'' from~\cite{GowersMilicevic2017}.
    
    We will now describe our proof of Theorem~\ref{thm:approxpoly},
    which occupies the bulk of the paper. For any
    $x,h_1,\dots,h_s\in\mathbf{F}_p^n$ and
    $\mathbf{k}_1,\dots,\mathbf{k}_s\in(\mathbf{F}_p^n)^{\{0,1\}^s}$, define the associated \textit{second-order $s$-dimensional
      cube} $\square^2(x;h_1,\dots,h_s;\mathbf{k}_1,\dots,\mathbf{k}_s)$ to be the subset
    \begin{equation*}
      \left\{x+\omega\cdot(h_1,\dots,h_s)+\omega'\cdot(k^{\omega}_1,\dots,k^{\omega}_s):\omega,\omega'\in\{0,1\}^s\text{ with }\omega'\neq\mathbf{0}\right\}
    \end{equation*}
    of $\mathbf{F}_p^n$, where $\mathbf{0}\in\{0,1\}^s$ denotes the vector of all $0$'s. Analogously to the definition for first-order
    $s$-dimensional cubes, we say that $\phi$ \textit{respects} $\square^2(x;h_1,\dots,h_s;\mathbf{k}_1,\dots,\mathbf{k}_s)$ if
    \begin{equation*}
      \sum_{\omega\in\{0,1\}^s}\sum_{\mathbf{0}\neq\omega'\in\{0,1\}^s}(-1)^{|\omega|+|\omega'|}\phi\left(x+\omega\cdot(h_1,\dots,h_s)+\omega'\cdot(k^{\omega}_1,\dots,k^{\omega}_s)\right)=0.
    \end{equation*}

    Recall, now, the Gowers--Cauchy--Schwarz inequality (see, e.g.,~\cite[Lemma B.2]{GreenTao2010II}):
    \begin{theorem}\label{thm:GCSnorm}
      Let $s\geq 2$ be an integer and, for each $\omega\in\{0,1\}^{s}$,
      $f_\omega:\mathbf{F}_p^n\to\mathbf{C}$. Then,
      \begin{equation*}
        \Big|\mathbf{E}_{x,h_1,\dots,h_s\in\mathbf{F}_p^n}\prod_{\omega\in\{0,1\}^s}\mathcal{C}^{|\omega|}f_\omega(x+\omega\cdot(h_1,\dots,h_s))\Big|\leq\prod_{\omega\in\{0,1\}^s}\|f_\omega\|_{U^s(\mathbf{F}_p^n)},
      \end{equation*}
      where $\mathcal{C}:\mathbf{C}\to\mathbf{C}$ denotes the complex conjugation operator.
    \end{theorem}
    Combining an application of the Cauchy--Schwarz inequality with
    the Gowers--Cauchy--Schwarz inequality shows that if $\phi$ is
    $\delta$-approximately quadratic, then $\phi$ respects at least a
    $\delta^{16}$-proportion of second-order $3$-dimensional cubes in
    $\mathbf{F}_p^n$. Adapting a dependent random choice argument of
    Gowers~\cite[Sections 9, 12, and 14]{Gowers2001} then produces a
    subset $A\subset\mathbf{F}_p^n$ containing many second-order
    $3$-dimensional cubes, almost all of which are respected
    by $\phi|_A$.
    
    Define, given $f_\omega:\mathbf{F}_p^n\to\mathbf{C}$ for each $\mathbf{0}\neq \omega\in\{0,1\}^3$, the ($U^3$-)\textit{dual function}
    \begin{equation}\label{eq:dualdef}
      \mathcal{D}[(f_\omega)_{\omega}](x):=\mathbf{E}_{h_1,h_2,h_3\in G}\prod_{\mathbf{0}\neq\omega\in\{0,1\}^3}\mathcal{C}^{|\omega|}f_\omega(x+\omega\cdot(h_1,h_2,h_3)).
    \end{equation}
    We will simply write $\mathcal{D}[f]$ for
    $\mathcal{D}[(f_\omega)_\omega]$ when $f_\omega=f$ for all
    $\mathbf{0}\neq\omega\in\{0,1\}^3$. The next step of the argument
    is to prove a ``Bogolyubov-type'' lemma for $\mathcal{D}[1_A]$. One
    variant of Bogolyubov's lemma says that $U^2$-dual functions,
    which are just three-fold convolutions, can be very well
    approximated by functions that are constant on cosets of a
    subspace of small codimension. We show, analogously, that there
    exists a subspace $V\leq \mathbf{F}_p^n$ and quadratic polynomials
    $Q_1,\dots,Q_s:\mathbf{F}_p^n\to\mathbf{F}_p$ with
    $\codim{V},s\ll\delta^{-O(1)}$ such that $\mathcal{D}[1_A]$ can be
    very well approximated by a function that is constant on
    \textit{quadratic level sets}
    $(a+V)\cap Q_1^{-1}(b_1)\cap\dots\cap Q_s^{-1}(b_s)$. A similar
    quadratic Bogolyubov-type result has been established
    independently by Castro-Silva, Golovnev, Gur, and Wolf and will be
    published in forthcoming work~\cite{Wolf2026}.

    Now, we define a function $\psi:\mathbf{F}_p^n\to\mathbf{F}_p^n$
    by choosing its values randomly as follows: if
    $\mathcal{D}[1_A](x)$ is large, so there are many choices of
    $h_1,h_2,h_3\in \mathbf{F}_p^n$ for which
    $x+\omega\cdot(h_1,h_2,h_3)\in A$ for all
    $\mathbf{0}\neq \omega\in\{0,1\}^3$, then we pick such a triple
    $(h_1,h_2,h_3)$ uniformly at random and set $\psi(x)$ to equal
    \begin{equation}\label{eq:psiprelim}
      \sum_{\mathbf{0}\neq\omega\in\{0,1\}^3}(-1)^{|\omega|}\phi|_A(x+\omega\cdot(h_1,h_2,h_3))
    \end{equation}
    if there are many other such triples giving the same value
    of~\eqref{eq:psiprelim}; otherwise, we pick $\psi(x)$ to be a
    uniformly random element of $\mathbf{F}_p^n$. A typical $\psi$ is
    closely related to $\phi$ (in a precise sense implying that if
    $\psi(x)$ often agrees with a polynomial map of degree at most
    $2$, then so does~\eqref{eq:psiprelim}) and, by our
    Bogolyubov-type lemma combined with a bit more work, respects
    almost all $3$-dimensional cubes in a high-rank quadratic level
    set of small codimension.

    It thus remains to classify $(1-\varepsilon)$-approximate
    quadratics on high-rank quadratic level sets when $\varepsilon$ is
    small. Kazhdan and Ziegler have already shown that such functions
    agree almost everywhere with the restriction of a polynomial map
    of degree at most $2$ (by combining~\cite{KazhdanZiegler2019} with
    either~\cite{KazhdanZiegler2017} or the more
    general~\cite{KazhdanZiegler2020}). We give a new and shorter
    proof of this fact using ideas from group cohomology, and also
    record the quantitative dependence on rank.

    The upshot is that, under the assumptions of
    Theorem~\ref{thm:main}, there exists a polynomial map
    $Q:\mathbf{F}_p^n\to\mathbf{F}_p^n$ of degree at most $2$ such that
    \begin{equation*}
      \sum_{\mathbf{0}\neq\omega\in\{0,1\}^3}(-1)^{|\omega|}\phi(x+\omega\cdot(h_1,h_2,h_3))=Q(x)
    \end{equation*}
    for many $x,h_1,h_2,h_3\in\mathbf{F}_p^n$. Then, since
    $\partial_{h_1,h_2,h_3}Q(x)=0$, setting $\phi':=\phi+Q$ yields that
    \begin{equation*}
      \sum_{\mathbf{0}\neq\omega\in\{0,1\}^3}(-1)^{|\omega|}\phi'(x+\omega\cdot(h_1,h_2,h_3))=0
    \end{equation*}
    for many $x,h_1,h_2,h_3\in\mathbf{F}_p^n$. By a couple of
    applications of the Cauchy--Schwarz inequality, this implies that
    $\phi'$ is approximately linear. Theorem~\ref{thm:approxpoly} now
    follows from Theorem~\ref{thm:approxlinear}.

    \section{Reduction to classification of approximate quadratics}\label{sec:approxpolys}
    The goal of this section is to prove that
    Theorem~\ref{thm:approxpoly} implies Theorem~\ref{thm:main}. Since
    we fix a prime $p\geq 5$ in Theorems~\ref{thm:main}
    and~\ref{thm:approxpoly}, we will suppress the dependence of
    implied constants (which we will denote by $C$ or $C_i$) on
    $p$ and set $G:=\mathbf{F}_p^n$ for the remainder of the paper.

    Following Manners~\cite[Section 5]{Manners2018} (with slightly
    different notation), for $f:G\to\mathbf{C}$ we define the function
    $S_{f}:G^3\to\mathbf{C}$ by
    \begin{equation*}
      S_f(h_1,h_2,h_3):=\mathbf{E}_{x\in G}\Delta_{h_1,h_2,h_3}f(x).
    \end{equation*}
    For any $\xi\in G$, we will write $\widehat{S_f}(h_1,h_2,\xi)$ to
    denote the Fourier transform of $S_f$ in the third component,
    i.e.,
    \begin{equation*}
      \widehat{S_f}(h_1,h_2,\xi):=\mathbf{E}_{x,h_3\in G}\Delta_{h_1,h_2,h_3}f(x)e_p(-\xi\cdot h_3)
    \end{equation*}
    for all $h_1,h_2\in G$. Observe that $\widehat{S_f}(h_1,h_2,\xi)=|\widehat{\Delta_{h_1,h_2}f}(-\xi)|^2$, and so
    \begin{equation}\label{eq:parse}
      \sum_{\xi\in G}\widehat{S_f}(h_1,h_2,\xi)=\|\Delta_{h_1,h_2}f\|_{L^2}^2
    \end{equation}
    for all $h_1,h_2\in G$ by Parseval's identity and
    \begin{equation}\label{eq:U4identity}
      \mathbf{E}_{h_1,h_2\in G}\sum_{\xi\in G}\widehat{S_f}(h_1,h_2,\xi)^2=\|f\|^{16}_{U^4(G)},
    \end{equation}
    by the identities
    $\|\Delta_{h_1,h_2}f\|_{U^2(G)}=\|\widehat{\Delta_{h_1,h_2}f}\|_{\ell^4}$
    and
    $\|f\|^{16}_{U^4(G)}=\mathbf{E}_{h_1,h_2\in
      G}\|\Delta_{h_1,h_2}f\|_{U^2(G)}^4$.

    A simple application of the pigeonhole principle shows that if $f$
    has large $U^4$-norm, then there are many $h_1,h_2\in G$ for which
    $\|\widehat{S}_f(h_1,h_2,\cdot)\|_{L^\infty}$ is large.
    \begin{lemma}\label{lem:ftophi}
      Let $f:G\to\mathbf{C}$ be $1$-bounded. If
      $\|f\|_{U^4(G)}\geq\delta$, then there exists $A\subset G^2$
      with $|A|\geq\frac{\delta^{16}}{2}|G|^2$ and $\phi:A\to G$ such
      that $\widehat{S_f}(a,a',\phi(a,a'))\geq\frac{\delta^{16}}{2}$
      whenever $(a,a')\in A$.
    \end{lemma}
    \begin{proof}
      By~\eqref{eq:U4identity}, we have
    \begin{equation*}
      \mathbf{E}_{h_1,h_2\in G}\sum_{\xi\in G}\widehat{S_f}(h_1,h_2,\xi)^2\geq\delta^{16}.
    \end{equation*}
    Thus, there exists a subset $A\subset G^2$ of density at least
    $\frac{\delta^{16}}{2}$ for which
    \begin{equation*}
      \sum_{\xi\in G}\widehat{S_f}(a,a',\xi)^2\geq\frac{\delta^{16}}{2}
    \end{equation*}
    for all $(a,a')\in A$. For each such pair $(a,a')$,
    \begin{equation*}
      \frac{\delta^{16}}{2}\leq \max_{\eta\in G}\widehat{S_f}(a,a',\eta)\cdot\sum_{\xi\in G}\widehat{S_f}(a,a',\xi)\leq\max_{\eta\in G}\widehat{S_f}(a,a',\eta),
    \end{equation*}
    by~\eqref{eq:parse} since $f$ is $1$-bounded. Setting $\phi(a,a'):=\eta$ for $\eta$
    attaining the maximum on the right-hand side above yields the
    conclusion of the lemma.
  \end{proof}

  Now, we diverge from Manners's proof a bit, using a random sampling
  argument to extend $\phi$ to a function on all of $G^2$ that respects
  very few more $3$-dimensional cubes than $\phi$ itself.

  \begin{lemma}\label{lem:randomfn}
    Let $A\subset G^2$ and $\phi:A\to G$. Consider the function
    $\psi:G^2\to G$ where $\psi|_{A}=\phi$ and, for all
    $x\in G^2\setminus A$, the value of $\psi(x)$ is chosen
    independently and uniformly at random from $G$. Then, with probability at
    least $1-\frac{30}{p^{n/2}}$, $\psi$ respects at most $p^{15n/2}$
    $3$-dimensional cubes $\square(x;h_1,h_2,h_3)\not\subset A$.
  \end{lemma}
  \begin{proof}
    If $\square(x;h_1,h_2,h_3)\not\subset A$ is nondegenerate, then
    $\mathbf{E}_{\psi}1_{\partial_{h_1,h_2,h_3}\psi(x)=0}=p^{-n}$
    since the elements $x+\omega\cdot(h_1,h_2,h_3)$ for
    $\omega\in\{0,1\}^3$ are distinct. The number of degenerate
    $3$-dimensional cubes in $G^2$ is certainly at most
    $\binom{8}{2}|G|^{6}= 28p^{6n}$. Thus, the expected number of $3$-dimensional cubes respected by $\psi$
    that are not wholly contained in $A$ is at most
    $p^{7n}+28p^{6n}$. The conclusion of the lemma now follows from
    Markov's inequality.
  \end{proof}

  Combining the previous two lemmas yields $\psi: G^2\to G$ respecting basically the same set of $3$-dimensional cubes as $\phi$ for
  which, setting
  $\Psi_{h_3}(h_1,h_2):=e_p(-\psi(h_1,h_2)\cdot h_3)$, the average
  \begin{equation}\label{eq:psireplace}
    \mathbf{E}_{h_3\in G}\mathbf{E}_{x,h_1,h_2\in G}\Delta_{h_3}f(x)\overline{\Delta_{h_3}f(x+h_1)}\overline{\Delta_{h_3}f(x+h_2)}\Delta_{h_3}f(x+h_1+h_2)\Psi_{h_3}(h_1,h_2)
  \end{equation}
  is large. The same sequence of applications of the Cauchy--Schwarz
  inequality used by Manners in~\cite[Section~5]{Manners2018} will
  then show that $\psi$ is $\Omega(\delta^{O(1)})$-approximately quadratic. To carry
  out Manners's argument, we will require a second version of the
  Gowers--Cauchy--Schwarz inequality (see, e.g.,~\cite[Lemma
  B.3]{GreenTao2010II}), as well as the notion of a generalized
  convolution.
  \begin{theorem}\label{thm:GCSbox}
    Let $s\geq 2$ be an integer, $f:G\to\mathbf{C}$ and, for each
    $i=1,\dots,s$, $f_i:G^{s}\to\mathbf{C}$ be a $1$-bounded function where $f_i(x_1,\dots,x_s)$ does not depend on the
    $i^{th}$ variable. Then,
    \begin{equation*}
      \Big|\mathbf{E}_{h_1,\dots,h_s\in G}f(h_1+\dots+h_s)\prod_{i=1}^sf_i(h_1,\dots,h_s)\Big|\leq\|f\|_{U^s(G)}.
    \end{equation*}
  \end{theorem}
  \begin{definition}
    A \textit{generalized convolution} is any function $F:G\to\mathbf{C}$ of the form
    \begin{equation*}
      F(x)=\mathbf{E}_{\substack{y_1,y_2,y_3\in G \\ y_1+y_2+y_3=x}}f_1(y_2,y_3)f_2(y_1,y_3)f_3(y_1,y_2)
    \end{equation*}
    for some $f_1,f_2,f_3:G^2\to\mathbf{C}$.
  \end{definition}

  Applying the Cauchy--Schwarz inequality four times to the inner
  average of~\eqref{eq:psireplace} allows one to replace each instance
  of $\Delta_{h_3}f$ by a generalized convolution of $1$-bounded
  functions, and then an application of the Gowers--Cauchy--Schwarz inequality bounds
  the resulting expression by the $U^3(G^2)$-norm of
  $\Psi_{h_3}$. This is the content of the next two lemmas.

  \begin{lemma}\label{lem:genconvolv2}
    Let $\Psi:G^2\to\mathbf{C}$ and
    $f_{00},f_{10},f_{01},f_{11}:G\to\mathbf{C}$ all be
    $1$-bounded. Then, for any $\omega\in\{0,1\}^2$, there exists a
    generalized convolution $F_\omega$ of $1$-bounded functions such
    that
    \begin{align*}
      \Big|\mathbf{E}_{x,h_1,h_2\in G}\Psi(h_1,h_2)&\prod_{\omega'\in\{0,1\}^2}f_{\omega'}(x+\omega'\cdot(h_1,h_2))\Big|^2\\
      &\leq \mathbf{E}_{x,h_1,h_2\in G}\Psi(h_1,h_2)F_{\omega}(x+\omega\cdot(h_1,h_2))\prod_{\omega\neq \omega'\in\{0,1\}^2}f_{\omega'}(x+\omega'\cdot(h_1,h_2)).
    \end{align*}
  \end{lemma}
  \begin{proof}
    By the change of variables $x\mapsto x-\omega\cdot(h_1,h_2)$, the left-hand side above equals
    \begin{equation*}
      \Big|\mathbf{E}_{x,h_1,h_2\in G}\Psi(h_1,h_2)f_{\omega}(x)\prod_{\omega\neq \omega'\in\{0,1\}^2}f_{\omega'}(x+(\omega'-\omega)\cdot(h_1,h_2))\Big|^2.
    \end{equation*}
    An application of the Cauchy--Schwarz inequality shows that this is bounded above by
    \begin{equation*}
      \mathbf{E}_{x,h_1,h_2,h_1',h_2'\in G}\Psi(h_1,h_2)\overline{\Psi(h_1',h_2')}\prod_{\omega\neq \omega'\in\{0,1\}^2}f_{\omega'}(x+(\omega'-\omega)\cdot(h_1,h_2))\overline{f_{\omega'}(x+(\omega'-\omega)\cdot(h_1',h_2'))},
    \end{equation*}
    which, setting
    \begin{equation*}
      F_\omega(x):=\mathbf{E}_{h_1',h_2'\in G}\overline{\Psi(h_1',h_2')}\prod_{\omega\neq\omega'\in\{0,1\}^2}\overline{f_{\omega'}(x+(\omega'-\omega)\cdot(h_1',h_2'))}
    \end{equation*}
    and then making the change of variables $x\mapsto x+\omega\cdot(h_1,h_2)$, can be written as 
    \begin{equation*}
      \mathbf{E}_{x,h_1,h_2\in G}\Psi(h_1,h_2)F_\omega(x+\omega\cdot(h_1,h_2))\prod_{\omega\neq \omega'\in\{0,1\}^2}f_{\omega'}(x+\omega'\cdot(h_1,h_2)).
    \end{equation*}
    Observe that
    $F_\omega(x)=\mathbf{E}_{h_1,h_2\in
      G}\overline{\Psi(h_1,h_2)}g_1(x+\epsilon_1h_1)g_2(x+\epsilon_2h_2)g_3(x+\epsilon_1h_1+\epsilon_2h_2)$
    for some $\epsilon_1,\epsilon_2\in\{-1,1\}$ and $1$-bounded
    $g_1,g_2,g_3:G\to\mathbf{C}$, and so writing $y_1=-\epsilon_1h_1$,
    $y_2=-\epsilon_2h_2$, and $y_3=x+\epsilon_1h_1+\epsilon_2h_2$, we
    have
    \begin{equation*}
      F_\omega(x)=\mathbf{E}_{\substack{y_1,y_2,y_3\in G \\ y_1+y_2+y_3=x}}f_1(y_2,y_3)f_2(y_1,y_3)f_3(y_1,y_2)
    \end{equation*}
    with $f_1(y_2,y_3)=g_1(y_2+y_3)g_3(y_3)$,
    $f_2(y_1,y_3)=g_2(y_1+y_3)$, and
    $f_3(y_1,y_2)=\Psi(-\epsilon_1y_1,-\epsilon_2y_2)$. Thus
    $F_\omega$ is a generalized convolution of $1$-bounded functions.
  \end{proof}
 
  \begin{lemma}\label{lem:genconvolv1}
    Let $F_{00},F_{10},F_{01},F_{11}:G\to\mathbf{C}$ be generalized convolutions of $1$-bounded functions and $\Psi:G^2\to\mathbf{C}$. Then,
    \begin{equation}\label{eq:genconvolveU3}
      \left|\mathbf{E}_{x,h_1,h_2\in G}F_{00}(x)F_{10}(x+h_1)F_{01}(x+h_2)F_{11}(x+h_1+h_2)\Psi(h_1,h_2)\right|\leq\|\Psi\|_{U^3(G^2)}.
    \end{equation}
  \end{lemma}
  \begin{proof}
    Using the definition of generalized convolution and relabeling $(h_1,h_2)$ as $(h,h')$, the average on
    the left-hand side of~\eqref{eq:genconvolveU3} can be written as
    \begin{equation*}
      \mathbf{E}_{x,h,h'\in G}\mathbf{E}_{\substack{y^\omega_1,y^\omega_2,y^\omega_3\in G \\ y^\omega_1+y^\omega_2+y^\omega_3=x+\omega\cdot(h,h')\\\omega\in\{0,1\}^2}}\Psi(h,h')\prod_{\omega\in\{0,1\}^2}f_1^\omega(y_2^\omega,y_3^\omega)f_2^\omega(y_1^\omega,y_3^\omega)f_3^\omega(y_1^\omega,y_2^\omega)
    \end{equation*}
    for some $1$-bounded functions
    $f_1^{\omega},f_2^{\omega},f_3^{\omega}:G^2\to\mathbf{C}$ for each
    $\omega\in\{0,1\}^2$. Inserting extra averaging, the above equals
    \begin{align*}
      \mathbf{E}_{\substack{h_1,h_2,h_3\in G \\ h_1',h_2',h_3'\in G}}\mathbf{E}_{\substack{x,x_1^\omega,x_2^\omega,x_3^\omega,y^\omega_1,y^\omega_2,y^\omega_3\in G \\ x_1^\omega+x_2^\omega+x_3^\omega=x \\ y^\omega_1+y^\omega_2+y^\omega_3=x+\omega\cdot(h_1+h_2+h_3,h_1'+h_2'+h_3')\\\omega\in\{0,1\}^2}}\Big[&\Psi(h_1+h_2+h_3,h_1'+h_2'+h_3') \\
                                                                                                                                                                                                                                                                                                                                                                                                                                   &\prod_{\omega\in\{0,1\}^2}f_1^\omega(y_2^\omega,y_3^\omega)f_2^\omega(y_1^\omega,y_3^\omega)f_3^\omega(y_1^\omega,y_2^\omega)\Big],
    \end{align*}
    which, after making the change of variables
    $z_i^\omega= y_i^\omega-x_i^\omega-\omega\cdot(h_i,h_i')$ for
    each $i=1,2,3$ and $\omega\in\{0,1\}^2$ and swapping the order of
    summation, equals
    \begin{align*}
      \mathbf{E}_{\substack{x,x_1^\omega,x_2^\omega,x_3^\omega,z^\omega_1,z^\omega_2,z^\omega_3\in G \\ x_1^\omega+x_2^\omega+x_3^\omega=x \\ z^\omega_1+z^\omega_2+z^\omega_3=0 \\ \omega\in\{0,1\}^2}}\mathbf{E}_{\substack{h_1,h_2,h_3\in G \\ h'_1,h'_2,h'_3\in G}}\Psi(h_1+h_2+h_3,h'_1+h'_2+h'_3)\prod_{1\leq i<j\leq 3} g^{(i,j)}_{\mathbf{x},\mathbf{z}}(h_i,h_j,h_i',h_j')
    \end{align*}
    where
    \begin{equation*}
      g^{(i,j)}_{\mathbf{x},\mathbf{z}}(h_i,h_j,h'_i,h'_j):=\prod_{\omega\in\{0,1\}^2}f_k^\omega(z_i^\omega+x_i^\omega+\omega\cdot(h_i,h'_i),z_j^\omega+x_j^\omega+\omega\cdot(h_j,h'_j))
    \end{equation*}
    with $k$ being the remaining number in $[3]\setminus\{i,j\}$ for each $1\leq i<j\leq 3$. The conclusion of the lemma now follows from the
    Gowers--Cauchy--Schwarz inequality in the form of
    Theorem~\ref{thm:GCSbox}.
  \end{proof}

  Now, we can combine the previous three lemmas with
  Theorem~\ref{thm:approxpoly} to replace the $\phi$ produced by
  Lemma~\ref{lem:ftophi} with a polynomial map of degree at most $2$.
  \begin{lemma}\label{lem:SfQ}
    Fix a prime $p\geq 5$. If $f:G\to\mathbf{C}$ is a nonzero $1$-bounded function
    and $\|f\|_{U^4(G)}\geq\delta$, then there exists a polynomial map
    $Q:G^2\to G$ of degree at most $2$ such that
    \begin{equation}\label{eq:conclu}
      \mathbf{E}_{h_1,h_2\in G}\widehat{S_f}(h_1,h_2,Q(h_1,h_2))\gg\exp(-O(\delta^{-O(1)})).
    \end{equation}
  \end{lemma}
  \begin{proof}    
    By Lemma~\ref{lem:ftophi}, we can find $A\subset G^2$ with
    $|A|\geq\frac{\delta^{16}}{2}|G|^2$ and $\phi:A\to G$ such that
    $\widehat{S_f}(a,a',\phi(a,a'))\geq\frac{\delta^{16}}{2}$ whenever
    $(a,a')\in A$. If $p^n>900$, then Lemma~\ref{lem:randomfn} says
    that there exists $\psi:G^2\to G$ with $\psi|_{A}=\phi$ such that
    $\psi$ respects at most $p^{15n/2}$ $3$-dimensional cubes that are
    not wholly contained in $A$. Since $\widehat{S_f}$ is always
    nonnegative, we therefore have
    \begin{equation*}
      \mathbf{E}_{h_1,h_2\in G}\widehat{S_{f}}(h_1,h_2,\psi(h_1,h_2))\gg\delta^{O(1)}.
    \end{equation*}
    Inserting the definition of $\widehat{S_f}$, the left-hand side above equals
    \begin{equation}\label{eq:avgabove}
      \mathbf{E}_{x,h_1,h_2,h_3\in G}\Delta_{h_3}f(x)\overline{\Delta_{h_3}f(x+h_1)}\overline{\Delta_{h_3}f(x+h_2)}\Delta_{h_3}f(x+h_1+h_2)\Psi_{h_3}(h_1,h_2),
    \end{equation}
    where, as in our exposition,
    $\Psi_{h_3}(h_1,h_2):=e_p(-\psi(h_1,h_2)\cdot h_3)$.

    Using Lemma~\ref{lem:genconvolv2} four times to replace each
    instance of $\Delta_{h_3}f$ in~\eqref{eq:avgabove} by a generalized
    convolution of $1$-bounded functions, Lemma~\ref{lem:genconvolv1},
    and then H\"older's inequality, we get
    \begin{equation*}
      \mathbf{E}_{h_3\in G}\|\Psi_{h_3}\|_{U^3(G^2)}^{8}\gg\delta^{O(1)}.
    \end{equation*}
    Expanding the definition of the $U^3$-norm and using orthogonality
    of characters, the above says that $\psi$ is
    $\Omega(\delta^{O(1)})$-approximately quadratic. Viewing $G$ as
    $G\times\{0\}\subset G^2$, Theorem~\ref{thm:approxpoly} then
    produces a polynomial map $Q:G^2\to G$ of degree at most $2$ such
    that $\psi(h_1,h_2)=Q(h_1,h_2)$ for at least
    $\frac{\exp(-C\delta^{-C})}{C}|G|^2$ pairs $(h_1,h_2)\in G^2$.

    Set $\theta(\delta):=\frac{\exp(-C\delta^{-C})}{C}$. We will show
    that, in fact, $\psi(a,a')=Q(a,a')$ for at least
    $\frac{\theta(\delta)}{2}|G|^2$ elements $(a,a')\in A$. Indeed, if
    $\psi(x,y)=Q(x,y)$ for more than $\frac{\theta(\delta)}{2}|G|^2$
    elements of $B:=G^2\setminus A$, then by the
    Gowers--Cauchy--Schwarz inequality in the form of
    Theorem~\ref{thm:GCSnorm} (with
    $f_{\mathbf{0}}(x,y)=1_{B}(x,y)1_{\psi(x,y)=Q(x,y)}$ and
    $f_\omega=1$ for all $\mathbf{0}\neq\omega\in\{0,1\}^3$), $\psi$
    must respect more than
    $\frac{\theta(\delta)^8}{2^8}|G|^8=\frac{\theta(\delta)^8}{2^8}p^{8n}$
    $3$-dimensional cubes contained in $B$. However, $\psi$ respects
    at most $p^{15n/2}$ $3$-dimensional cubes not contained in $A$,
    and so we have a contradiction as long as
    $p^n\geq \frac{2^{16}}{\theta(\delta)^{16}}$ (which, since
    $\theta(\delta)<1$, is certainly greater than $900$).

    Thus, $\phi(a,a')=Q(a,a')$ for at least
    $\frac{\theta(\delta)}{2}|G|^2$ pairs $(a,a')\in A$ when $p^{n}\geq \frac{2^{16}}{\theta(\delta)^{16}}$, so
    since $\widehat{S_f}(a,a',\phi(a,a'))\geq\frac{\delta^{16}}{2}$
    for all $(a,a')\in A$,~\eqref{eq:conclu} follows by the
    nonnegativity of $\widehat{S_f}$.

    When $p^n<\frac{2^{16}}{\theta(\delta)^{16}}$,~\eqref{eq:conclu} can be obtained almost
    immediately. Indeed,
    \begin{equation*}
      \delta^{16}\leq \mathbf{E}_{h_1,h_2\in G}\sum_{\xi\in G}\widehat{S_f}(h_1,h_2,\xi)^2\leq \sum_{\xi\in G}\mathbf{E}_{h_1,h_2\in G}\widehat{S_f}(h_1,h_2,\xi)\leq p^{n}\max_{\xi\in G}\mathbf{E}_{h_1,h_2\in G}\widehat{S_f}(h_1,h_2,\xi),
    \end{equation*}
    and so there exists $\xi\in G$ for which
    $\mathbf{E}_{h_1,h_2\in G}\widehat{S_f}(h_1,h_2,\xi)\geq
    \delta^{16}p^{-n}$. Thus, if $p^n< \frac{2^{16}}{\theta(\delta)^{16}}$, then by taking
    $Q\equiv \xi$, we get
    $\mathbf{E}_{h_1,h_2\in G}\widehat{S_f}(h_1,h_2,Q(h_1,h_2))\geq
    \frac{\delta^{16}\theta(\delta)^{16}}{2^{16}}$, which is acceptable.
  \end{proof}

  Combining Lemma~\ref{lem:SfQ} with two applications of the
  Cauchy--Schwarz inequality yields
  \begin{equation*}
    \mathbf{E}_{x,h_1,h_2,h_3\in G}\Delta_{h_1,h_2,h_3}f(x) e_p(T(h_1,h_2,h_3))\gg_\delta 1.
  \end{equation*}
  for some trilinear form $T:G^3\to \mathbf{F}_p$, and we are now in
  the same position as Gowers and Mili\'cevi\'c at the beginning of
  Section~10 of~\cite{GowersMilicevic2017} (except that we have not
  proven Theorem~\ref{thm:approxpoly} yet). To complete the proof of
  Theorem~\ref{thm:main}, one would ideally like to ``integrate'' $T$
  when it is nonzero, writing it as the $3$-fold additive discrete
  derivative of a cubic polynomial, so that the phase
  $e_p(T(h_1,h_2,h_3))$ can be absorbed into
  $\Delta_{h_1,h_2,h_3}f(x)$. This is not possible, however, unless
  $T$ is symmetric. The $T$ produced by the argument so far is not
  necessarily symmetric, but it does have large correlation with
  $\Delta_{h_1,h_2,h_3}f(x)$, which is symmetric in $h_1,h_2,$ and
  $h_3$. This is enough information to ``symmetrize'' $T$ up to a low
  (analytic) rank error.

  One can then handle the low rank error by invoking bounds for the
  partition rank of trilinear forms in terms of their analytic
  rank. Polynomial bounds, which are sufficient for our purposes, have
  been known since work of Haramaty and
  Shpilka~\cite{HaramatyShpilka2010}, and Gowers and Mili\'cevi\'c
  gave their own proof along similar lines in~\cite[Subsection
  10.2]{GowersMilicevic2017}. Linear bounds, which we state below, have
  now been proven by Adiprasito, Kazhdan, and
  Ziegler~\cite{AdiprasitoKazhdanZiegler2021} and Cohen and
  Moshkovitz~\cite{CohenMoshkovitz2022}; Lampert~\cite{Lampert2025}
  also recently gave a short and elementary proof.

  \begin{theorem}\label{thm:rank}
    Let $T:G^3\to\mathbf{F}_p$ be a trilinear form satisfying
    \begin{equation*}
      \mathbf{E}_{x,y,z\in G}e_p(T(x,y,z))\geq\gamma.
    \end{equation*}
    Then, there exist linear forms $L_i,L_j',L_k'':G\to\mathbf{F}_p$
    and bilinear forms $B_i,B_j',B_k'':G^2\to\mathbf{F}_p$, for
    $i=1,\dots,d_1$, $j=1,\dots,d_2$, and $k=1,\dots,d_3$ with
    $d_1,d_2,d_3\ll\log(2/\gamma)$, such that
    \begin{equation*}
      T(x,y,z)=\sum_{i=1}^{d_1}L_i(x)B_i(y,z)+\sum_{j=1}^{d_2}L_j'(y)B_j'(x,z)+\sum_{k=1}^{d_3}L_k''(z)B_k''(x,y).
    \end{equation*}
  \end{theorem}

  Combining the symmetrization argument of Gowers and Mili\'cevi\'c
  from~\cite[Section~10]{GowersMilicevic2017} with
  Theorem~\ref{thm:rank} yields the following useful lemma.
  
  \begin{lemma}\label{lem:symmetrization}
    Let $p\geq 5$, $f:G\to\mathbf{C}$ be $1$-bounded, and
    $T:G^3\to\mathbf{F}_p$ be a trilinear form with
    \begin{equation}\label{eq:largetri}
      \left|\mathbf{E}_{x,h_1,h_2,h_3\in G}\Delta_{h_1,h_2,h_3}f(x)e_p(T(h_1,h_2,h_3))\right|\geq\gamma.
    \end{equation}
    Then, there exists a polynomial $P:G\to\mathbf{F}_p$ of degree at most $3$ such that
    \begin{equation*}
      \left|\mathbf{E}_{x\in\mathbf{F}_p^n}f(x)e_p(P(x))\right|\gg \gamma^{O(1)}.
    \end{equation*}
  \end{lemma}
  
  For the sake of completeness, we will prove
  Lemma~\ref{lem:symmetrization} here. Our argument is basically the
  same as the one in~\cite{GowersMilicevic2017}, but we avoid invoking
  the subadditivity of analytic rank and are able to be more efficient
  by using a cocycle identity. As in~\cite{GowersMilicevic2017}, we
  require a symmetrization result for bilinear forms from work of
  Green and Tao on the $U^3$-norm~\cite[Section~6, Step
  2]{GreenTao2008II}. We record this lemma, whose proof is just two
  applications of the Cauchy--Schwarz inequality, below, and then
  prove Lemma~\ref{lem:symmetrization}.
  \begin{lemma}\label{lem:bilsym}
    Let $p\geq 3$ and $B:G^2\to\mathbf{F}_p$ be a bilinear form. If,
    for $f:G\to\mathbf{C}$ $1$-bounded,
    \begin{equation*}
      \left|\mathbf{E}_{x,h_1,h_2\in G}\Delta_{h_1,h_2}f(x)e_p(B(h_1,h_2))\right|\geq\gamma,
    \end{equation*}
    then $\mathbf{E}_{x,y\in G}e_p(B(x,y)-B(y,x))\geq\gamma^4$.
  \end{lemma}
  
  \begin{proof}[Proof of Lemma~\ref{lem:symmetrization}]
    First, there exists a subset $H\subset G$ of density at least
    $\frac{\gamma}{2}$ such that
    \begin{equation*}
      \left|\mathbf{E}_{x,h_1,h_2\in G}\Delta_{h_1,h_2}[\Delta_{h_3}f](x)e_p(T(h_1,h_2,h_3))\right|\geq\frac{\gamma}{2}
    \end{equation*}
    for all $h_3\in H$. By Lemma~\ref{lem:bilsym}, this implies that
    \begin{equation*}
      \mathbf{E}_{h_1,h_2,h_3\in G}e_p(T(h_1,h_2,h_3)-T(h_2,h_1,h_3))\gg\gamma^5.
    \end{equation*}
    Set $T'(h_1,h_2,h_3):=T(h_1,h_2,h_3)-T(h_2,h_1,h_3)$, apply
    Theorem~\ref{thm:rank}, and then let $V$ be the intersection of
    the kernels of all the linear forms $L_i,L_j',L_k''$ produced by
    the theorem. Then, $T'$ vanishes on $V$ and
    $\codim V\ll\log(2/\gamma)$. Repeating the same argument, but
    fixing $h_1$ and swapping $h_2$ and $h_3$ instead, produces a
    subspace $W\leq G$ with $\codim W\ll\log(2/\gamma)$ such that
    $T''(h_1,h_2,h_3):=T(h_1,h_2,h_3)-T(h_1,h_3,h_2)$ vanishes on $W$.

    Set $Z:=V\cap W$. Then, $T'(h_1,h_2,h_3)=T''(h_1,h_2,h_3)=0$
    for all $h_1,h_2,h_3\in Z$, so since the transpositions $(1\ 2)$
    and $(2\ 3)$ generate $S_3$, this means that $T|_{Z}$ is
    symmetric. Thus, if we set $P(x):=-\frac{T(x,x,x)}{6}$, then
    $\partial_{h_1,h_2,h_3}P(x)=T(h_1,h_2,h_3)$ for all
    $h_1,h_2,h_3\in Z$ and $x\in G$.
    
    Splitting up the left-hand side of~\eqref{eq:largetri} into
    averages over cosets of $Z$, we have
    \begin{equation*}
      \mathbf{E}_{a+Z,b+Z,c+Z\in G/Z}\left|\mathbf{E}_{x\in G}\mathbf{E}_{z_1,z_2,z_3\in
      Z}\Delta_{a+z_1,b+z_2,c+z_3}f(x)e_p(T(a+z_1,b+z_2,c+z_3))\right|\geq\gamma.
    \end{equation*}
    By the pigeonhole principle, there exists a triple of cosets
    $(a+Z,b+Z,c+Z)$ for which
    \begin{equation}\label{eq:abcpigeon}
      \left|\mathbf{E}_{x\in G}\mathbf{E}_{z_1,z_2,z_3\in
          Z}\Delta_{a+z_1,b+z_2,c+z_3}f(x)e_p(T(a+z_1,b+z_2,c+z_3))\right|\geq\gamma.
    \end{equation}
    Now, for any $z_1,z_2,z_3\in Z$ and $x\in G$,
    \begin{equation*}
      T(a+z_1,b+z_2,c+z_3)=\partial_{z_1,z_2,z_3}P(x)+F^{(1)}_{a,b,c}(z_2,z_3)+F^{(2)}_{a,b,c}(z_1,z_3)+F^{(3)}_{a,b,c}(z_1,z_2),
    \end{equation*}
    where
    $F^{(1)}_{a,b,c}(z_2,z_3)=T(a,b,c)+T(a,b,z_3)+T(a,z_2,c)+T(a,z_2,z_3)$,
    $F^{(2)}_{a,b,c}(z_1,z_3)=T(z_1,b,c)+T(z_1,b,z_3)$, and
    $F^{(3)}_{a,b,c}(z_1,z_2)=T(z_1,z_2,c)$. On the other hand, by the
    cocycle identity
    $\partial_{a+z}\phi(x)=\partial_z\phi(x)+\partial_{a}\phi(x+z)$
    (which is valid for any $\phi:G\to G'$ and $x,a,z\in G$), we have
    \begin{equation*}
      \partial_{a+z_1,b+z_1,c+z_3}P(x)=\partial_{z_1,z_2,z_3}P(x)+Q_{a,b,c}^{(1)}(z_2,z_3)+Q_{a,b,c}^{(2)}(z_1,z_3)+Q_{a,b,c}^{(3)}(z_1,z_2),
    \end{equation*}
    where
    $Q^{(1)}_{a,b,c}(z_2,z_3)=\partial_{a,b,c}P(x+z_1+z_2+z_3)+\partial_{a,b,z_3}P(x+z_1+z_2)+\partial_{a,z_2,c}P(x+z_1+z_3)+\partial_{a,z_2,z_3}P(x+z_1)$,
    $Q^{(2)}_{a,b,c}(z_1,z_3):=\partial_{z_1,b,c}P(x+z_2+z_3)+\partial_{z_1,b,z_3}P(x+z_2)$,
    and $Q^{(3)}_{a,b,c}(z_1,z_2):=\partial_{z_1,z_2,c}P(x+z_3)$;
    here, we have used that $\partial_{h_1,h_2,h_3}P(x)$ does not
    depend on $x$ since $\deg{P}\leq 3$.

    Thus, setting $\tilde{f}(x):=f(x)e_p(P(x))$, the
    left-hand side of~\eqref{eq:abcpigeon} can be written as
    \begin{equation*}
      p^{3\codim{Z}}\big|\mathbf{E}_{x,h_1,h_2,h_3\in G}\Delta_{a+h_1,b+h_2,c+h_3}\tilde{f}(x)1_Z(h_1)1_Z(h_2)1_Z(h_3)g_1(h_2,h_3)g_2(h_1,h_3)g_3(h_1,h_2)\big|
    \end{equation*}
    for some $1$-bounded functions
    $g_1,g_2,g_3:G^2\to\mathbf{C}$. Expanding
    $\Delta_{a+h_1,b+h_2,c+h_3}\tilde{f}(x)$ and applying the
    Gowers--Cauchy--Schwarz inequality in the form of
    Theorem~\ref{thm:GCSbox} yields
    \begin{equation*}
      \|\tilde{f}\|_{U^3(G)}\gg\gamma^{O(1)}p^{-3\codim{Z}}\gg \gamma^{O(1)},
    \end{equation*}
    since $\codim{Z}\ll\log(2/\gamma)$, and then the
    conclusion of the lemma follows from Theorem~\ref{thm:U3inv}.
  \end{proof}
  
  Now, we can finish deriving Theorem~\ref{thm:main} from Theorem~\ref{thm:approxpoly}
  \begin{proof}[Proof of Theorem~\ref{thm:main}]
    By Lemma~\ref{lem:SfQ}, there exists a polynomial map $Q:G^2\to G$
    of degree at most $2$ for which
    \begin{equation}\label{eq:almostfinal}
      \mathbf{E}_{h_1,h_2\in G}\widehat{S_f}(h_1,h_2,Q(h_1,h_2))\gg \exp(-O(\delta^{-O(1)})).
    \end{equation}
    We can write $Q(x,y)=R_1(x)+R_2(y)+B(x,y)$, where $R_1,R_2:G\to G$ and $B:G^2\to G$ is bilinear. Expanding
    the definition of $\widehat{S_f}$, the left-hand side
    of~\eqref{eq:almostfinal} equals
    \begin{equation*}
      \mathbf{E}_{x,h_1,h_2,h_3\in G}\Delta_{h_1,h_2,h_3}f(x)g_1(h_1,h_3)g_2(h_2,h_3)e_p(T(h_1,h_2,h_3)),
    \end{equation*}
    where $T:G^3\to\mathbf{F}_p$ is the trilinear map defined by
    $T(h_1,h_2,h_3)=-B(h_1,h_2)\cdot h_3$,
    $g_1(h_1,h_3):=e_p(-R_1(h_1)\cdot h_3)$, and
    $g_2(h_2,h_3):=e_p(-R_2(h_2)\cdot h_3)$. Applying the
    Cauchy--Schwarz inequality twice, first to double the variable
    $h_2$ and then again to double the variable $h_1$, yields
    \begin{equation*}
      \mathbf{E}_{x,h_1',h_2',h_3\in G}\Delta_{h_1',h_2',h_3}f(x)e_p(T(h_1',h_2',h_3))\gg \exp(-O(\delta^{-O(1)}))
    \end{equation*}
    after a change of variables. The theorem now follows from
    Lemma~\ref{lem:symmetrization}.
  \end{proof}

  \section{Obtaining a $99\%$ approximate second-order quadratic on a dense set}\label{sec:DRC}

  We now begin the proof of Theorem~\ref{thm:approxpoly} by showing
  that if $\phi:G\to G$ is an approximate quadratic, then there exists
  a subset of $G$ containing many second-order $3$-dimensional cubes,
  almost all of which are respected by $\phi$. To do this, we adapt an
  argument introduced by Gowers in his proof of Szemer\'edi's
  theorem~\cite[Sections 9, 12, and 14]{Gowers2001}. Gowers's argument
  is an example of the dependent random choice method, and constructs
  a set randomly by selecting elements according to a distribution
  favoring instances of a fixed configuration that are respected by
  the function of interest. In both~\cite{Gowers2001} and the work of
  Gowers and
  Mili\'cevi\'c~\cite{GowersMilicevic2017,GowersMilicevic2020}, this
  technique is applied to certain multidimensional configurations,
  and, as we will see, it is not hard to apply it to second-order
  $3$-dimensional cubes as well.

  Before we proceed with the argument, we will need a couple of
  preliminaries from linear algebra. Denote the space of $n\times n$
  matrices with entries in $\mathbf{F}_p$ by $\M_n(\mathbf{F}_p)$.
  
  \begin{lemma}\label{lem:linalg}
    Let $p\geq 5$ and $k\in\mathbf{N}$. The proportion of
    $(a_1,\dots,a_k)\in G^k$ for which
  \begin{equation}\label{eq:forms}
    a_ia_j^T+a_ja_i^T\qquad 1\leq i\leq j\leq k
  \end{equation}
  are linearly dependent in $\M_n(\mathbf{F}_p)$ is at most
  $p^{-n/16}$, provided that $n\geq 32k(k+1)$.
  \end{lemma}
  \begin{proof}
    For any $M,M'\in \M_n(\mathbf{F}_p)$, we will (temporarily and
    locally for this proof) write $M\cdot M'\in\mathbf{F}_p$ to mean
    the dot product of $M$ and $M'$ viewed as vectors in
    $\mathbf{F}_p^{n^2}$. For each choice of
    $\xi_{i,j}\in\mathbf{F}_p$ for $1\leq i\leq j\leq k$ not all
    zero, we want to count $a_1,\dots,a_k\in G$ such that
    \begin{equation}\label{eq:xisum}
      \sum_{1\leq i\leq j\leq k}\xi_{i,j}(a_ia_j^T+a_ja_i^T)
    \end{equation}
    is the zero matrix. By the symmetry of~\eqref{eq:xisum}, it
    suffices to consider the case when $\xi_{1,1}\neq 0$ and the case
    when $\xi_{1,1}=0$ and $\xi_{1,2}\neq 0$.

    First suppose that $\xi_{1,1}\neq 0$. Then, by orthogonality of
    characters, the proportion of $k$-tuples $(a_1,\dots,a_k)\in G^k$
    for which~\eqref{eq:xisum} equals the zero matrix is
    \begin{equation}\label{eq:xi1count}
      \mathbf{E}_{M\in \M_n(\mathbf{F}_p)}\mathbf{E}_{a_1,\dots,a_k\in G}e_p(M\cdot a_1a_1^T+\phi_M(a_1,\dots,a_k)+\phi_M'(a_2,\dots,a_k)),
    \end{equation}
    for some $\phi_M:G^k\to\mathbf{F}_p$ and
    $\phi'_M:G^{k-1}\to\mathbf{F}_p$ with $\phi_M(a_1,\dots,a_k)$
    depending linearly on $a_1$ for all $M\in \M_n(\mathbf{F}_p)$. By
    applying the Cauchy--Schwarz inequality twice, each time doubling
    the $a_1$ variable, and then H\"older's inequality, we get that the fourth power
    of~\eqref{eq:xi1count} is at most
    \begin{equation}\label{eq:lmavg}
      \mathbf{E}_{M \in \M_n(\mathbf{F}_p)}\mathbf{E}_{m,\ell\in G}e_p(M\cdot(m\ell^T+\ell m^T)),
    \end{equation}
    which counts the proportion of pairs $(m,\ell)\in G^2$ for which
    $m\ell^T+\ell m^T$ equals the zero matrix. The $i$-th entry on the
    diagonal of $m\ell^T+\ell m^T$ is $2m_i\ell_i$, and so since
    $p>2$, at least one of $m_i$ or $\ell_i$ must be zero for each
    $i=1,\dots,n$ if $m\ell^T+\ell m^T$ is the zero matrix. There are
    then at most $(2p-1)^n$ possibilities for $(m,\ell)$, which, by
    the assumption that $p\geq 5$, is at most
    $p^{(1+\log{2}/\log{5})n}\leq p^{3n/2}$. Thus,~\eqref{eq:lmavg} is
    at most $p^{-n/2}$, and so~\eqref{eq:xi1count} is at most
    $p^{-n/8}$.

    The argument in the case that $\xi_{1,1}=0$ and $\xi_{1,2}\neq 0$
    is similar. The proportion of $k$-tuples $(a_1,\dots,a_k)\in G^k$
    for which~\eqref{eq:xisum} equals the zero matrix is
    \begin{equation}\label{eq:xi2count}
      \mathbf{E}_{M\in \M_n(\mathbf{F}_p)}\mathbf{E}_{a_1,\dots,a_k\in G}e_p(M\cdot (a_1a_2^T+a_2a_1^T)+\psi_M(a_1,a_3,\dots,a_k)+\psi'_M(a_2,\dots,a_k))
    \end{equation}
    for some $\psi_M,\psi'_M:G^{k-1}\to\mathbf{F}_p$ with
    $\psi_M(a_1,a_3,\dots,a_k)$ depending linearly on $a_1$ for all
    $M\in\M_n(\mathbf{F}_p)$. Two applications of the Cauchy--Schwarz
    inequality, first to double the variable $a_1$ and then to double
    the variable $a_2$, followed by H\"older's inequality yields
    that~\eqref{eq:xi2count} is also bounded above by the fourth root
    of~\eqref{eq:lmavg}, and therefore is at most $p^{-n/8}$.

    Now, by summing over all $p^{k+\binom{k}{2}}-1$ possible choices
    of the $\xi_{i,j}\in\mathbf{F}_p$ for which not all $\xi_{i,j}$
    are zero, we obtain the conclusion of the lemma when
    $n\geq 32k(k+1)$.
  \end{proof}

  The following basic lemma, whose proof is immediate by inspection,
  concerns systems of linear equations that naturally arise when
  dealing with $U^3$-dual functions.
  
  \begin{lemma}\label{lem:fixedeqns}
    Let $d\in\mathbf{N}$ and $\xi_\omega\in\mathbf{F}_p^d$ for all $\omega\in\{0,1\}^3$. If
    \begin{equation*}
      \begin{cases}
        \xi_{100}+\xi_{110}+\xi_{101}+\xi_{111}=0 \\
        \xi_{010}+\xi_{110}+\xi_{011}+\xi_{111}=0 \\
        \xi_{001}+\xi_{101}+\xi_{011}+\xi_{111}=0 \\
        \xi_{110}+\xi_{111}=0 \\
        \xi_{101}+\xi_{111}=0 \\
        \xi_{011}+\xi_{111}=0
      \end{cases},
    \end{equation*}
    then $\xi_\omega=(-1)^{|\omega|+1}\xi_{111}$ for $\mathbf{0}\neq\omega\in\{0,1\}^3$. We also have $\xi_{000}=-\xi_{111}$ if, in addition,
    \begin{equation*}
      \xi_{000}+\xi_{100}+\xi_{010}+\xi_{001}+\xi_{110}+\xi_{101}+\xi_{011}+\xi_{111}=0.
    \end{equation*}
  \end{lemma}

  Now, we can state and prove the main result of this section.
  
  \begin{lemma}\label{lem:DRC}
    Let $p\geq 5$. Suppose that $\phi:G\to G$ is
    $\delta$-approximately quadratic. For any
    $\gamma>0$, if $n\geq 2^{16}\lceil\log_p(2/\delta^{16}\gamma)\rceil$ then
    there exists $A\subset G$ of density
    $\Omega(\delta^{32}\gamma^{2})$ containing at least
    $2^{-55}\delta^{56\cdot 16}\gamma^{55}p^{28n}$ second-order
    $3$-dimensional cubes such that $\phi|_{A}$ respects at least a
    $(1-\gamma)$-proportion of them.
  \end{lemma}
  \begin{proof}
    First, we will show that $\phi$ respects many second-order
    $3$-dimensional cubes. By orthogonality of characters, we have
    \begin{equation*}
      \mathbf{E}_{\xi\in G}\mathbf{E}_{x,h_1,h_2,h_3\in G}e_p\left(\xi\cdot\partial_{h_1,h_2,h_3}\phi(x)\right)\geq\delta.
    \end{equation*}
    Applying the Cauchy--Schwarz inequality to the above to double $h_1$, $h_2$, and $h_3$ yields
    \begin{equation*}
      \mathbf{E}_{\xi\in G}\mathbf{E}_{x,h_1,h_2,h_3\in G}\overline{\mathcal{D}[f_\xi](x)}\prod_{\mathbf{0}\neq \omega\in\{0,1\}^3}\mathcal{C}^{|\omega|}f_\xi(x+\omega\cdot(h_1,h_2,h_3))\geq\delta^2,
    \end{equation*}
    where $f_\xi(x):=e_p(\xi\cdot\phi(x))$. Thus, by the
    Gowers--Cauchy--Schwarz inequality in the form of
    Theorem~\ref{thm:GCSnorm} followed by H\"older's inequality,
    $\mathbf{E}_{\xi\in
      G}\|\mathcal{D}[f_\xi]\|^8_{U^3(G)}\geq\delta^{16}$. Expanding
    the definitions of $\mathcal{D}[f_\xi]$ and the $U^3$-norm and
    using orthogonality of characters again, we obtain that $\phi$
    respects at least a $\delta^{16}$-proportion of second order
    $3$-dimensional cubes in $G$.

    Let $k\in\mathbf{N}$ (to be chosen later depending on $\delta$ and
    $\gamma$). Pick $M_1,\dots,M_k\in \M_n(\mathbf{F}_p)$ and
    $v_1,\dots,v_k\in G$ independently and uniformly at random, and
    consider the set
    \begin{equation*}
      A:=\left\{x\in G:x^TM_ix+v_i\cdot\phi(x)=0\right\}.
    \end{equation*}
    For any $56$-tuple
    \begin{equation*}
    (x_{\omega,\omega'})_{\substack{\omega,\omega'\in\{0,1\}^3 \\ \omega'\neq \mathbf{0}}}\in G^{\{0,1\}^3\times (\{0,1\}^3\setminus\{\mathbf{0}\})},  
    \end{equation*}
    the probability that each $x_{\omega,\omega'}$ lies in $A$ is
    \begin{align*}
      \mathbf{E}_{M_1,\dots,M_k\in \M_n(\mathbf{F}_p)}\mathbf{E}_{v_1,\dots,v_k\in G}&\prod_{\substack{\omega,\omega'\in\{0,1\}^3 \\ \omega'\neq \mathbf{0}}}\prod_{j=1}^k1_{\{0\}}(x_{\omega,\omega'}^TM_jx_{\omega,\omega'}+v_j\cdot\phi(x_{\omega,\omega'})) \\
      &=\Bigg(\mathbf{E}_{M\in \M_n(\mathbf{F}_p)}\mathbf{E}_{v\in G}\prod_{\substack{\omega,\omega'\in\{0,1\}^3 \\ \omega'\neq \mathbf{0}}}1_{\{0\}}(x_{\omega,\omega'}^TMx_{\omega,\omega'}+v\cdot\phi(x_{\omega,\omega'}))\Bigg)^k,
    \end{align*}
    which, by orthogonality of characters, equals
    \begin{equation*}
      \Bigg(\mathbf{E}_{\substack{\xi_{\omega,\omega'}\in\mathbf{F}_p\\ \omega,\omega'\in\{0,1\}^3 \\ \omega'\neq \mathbf{0}}}\mathbf{E}_{M\in \M_n(\mathbf{F}_p)}\mathbf{E}_{v\in G}e_p\Bigg(\sum_{\substack{\omega,\omega'\in\{0,1\}^3 \\ \omega'\neq \mathbf{0}}}\xi_{\omega,\omega'}[x_{\omega,\omega'}^TMx_{\omega,\omega'}]+v\cdot\sum_{\substack{\omega,\omega'\in\{0,1\}^3 \\ \omega'\neq \mathbf{0}}}\xi_{\omega,\omega'}\phi(x_{\omega,\omega'})\Bigg)\Bigg)^k.
    \end{equation*}
    
    Observe that the inner average
    \begin{equation}\label{eq:inneraverage}
      \mathbf{E}_{M\in \M_n(\mathbf{F}_p)}\mathbf{E}_{v\in G}e_p\Bigg(\sum_{\substack{\omega,\omega'\in\{0,1\}^3 \\ \omega'\neq \mathbf{0}}}\xi_{\omega,\omega'}[x_{\omega,\omega'}^TMx_{\omega,\omega'}]+v\cdot\sum_{\substack{\omega,\omega'\in\{0,1\}^3 \\ \omega'\neq \mathbf{0}}}\xi_{\omega,\omega'}\phi(x_{\omega,\omega'})\Bigg)
    \end{equation}
    is nonzero if and only if
    \begin{equation*}
      \sum_{\substack{\omega,\omega'\in\{0,1\}^3 \\ \omega'\neq \mathbf{0}}}\xi_{\omega,\omega'}[x_{\omega,\omega'}x_{\omega,\omega'}^T]=0\qquad \text{and}\qquad \sum_{\substack{\omega,\omega'\in\{0,1\}^3 \\ \omega'\neq \mathbf{0}}}\xi_{\omega,\omega'}\phi(x_{\omega,\omega'})=0,
    \end{equation*}
    in which case it equals $1$. If
    \begin{equation}\label{eq:3cube}
      x_{\omega,\omega'}=x+\omega\cdot(h_1,h_2,h_3)+\omega'\cdot(k_1^\omega,k_2^\omega,k_3^\omega),
    \end{equation}
    then obviously $\sum_{\substack{\omega,\omega'\in\{0,1\}^3 \\ \omega'\neq \mathbf{0}}}\xi_{\omega,\omega'}[x_{\omega,\omega'}x_{\omega,\omega'}^T]=0$ if and only if
    \begin{equation}\label{eq:quadeq}
      \sum_{\substack{\omega,\omega'\in\{0,1\}^3 \\ \omega'\neq \mathbf{0}}}\xi_{\omega,\omega'}(x+\omega\cdot(h_1,h_2,h_3)+\omega'\cdot(k_1^\omega,k_2^\omega,k_3^\omega))(x+\omega\cdot(h_1,h_2,h_3)+\omega'\cdot(k_1^\omega,k_2^\omega,k_3^\omega))^T=0.
    \end{equation}
    If, in addition,
    \begin{equation}\label{eq:28}
      \{x,h_1,h_2,h_3\}\cup\left\{k_1^\omega:\omega\in\{0,1\}^3\right\}\cup\left\{k_2^\omega:\omega\in\{0,1\}^3\right\}\cup\left\{k_3^\omega:\omega\in\{0,1\}^3\right\}
    \end{equation}
    consists of $28$ distinct elements $y_1,\dots,y_{28}$ of $G$ and 
    \begin{equation}\label{eq:LI}
      y_iy_j^T+y_jy_i^T\qquad 1\leq i\leq j\leq 28
    \end{equation}
    are linearly independent in $\M_n(\mathbf{F}_p)$, then, by
    expanding and collecting like terms,~\eqref{eq:quadeq} holds if
    and only if
    \begin{equation}\label{eq:system2}
      \begin{cases}
        \xi_{\omega,100}+\xi_{\omega,110}+\xi_{\omega,101}+\xi_{\omega,111}=0 \\
        \xi_{\omega,010}+\xi_{\omega,110}+\xi_{\omega,011}+\xi_{\omega,111}=0 \\
        \xi_{\omega,001}+\xi_{\omega,101}+\xi_{\omega,011}+\xi_{\omega,111}=0 \\
        \xi_{\omega,110}+\xi_{\omega,111}=0 \\
        \xi_{\omega,101}+\xi_{\omega,111}=0 \\
        \xi_{\omega,011}+\xi_{\omega,111}=0
      \end{cases},
    \end{equation}
    for each $\omega\in\{0,1\}^3$ (by looking at the terms involving only the $k_j^\omega$'s) and
    \begin{equation}\label{eq:system3}
    \begin{cases}
      &\sum_{\omega\in\{0,1\}^3}\sum_{\mathbf{0}\neq\omega'\in\{0,1\}^3}\xi_{\omega,\omega'}=0\\
      &\sum_{\mathbf{0}\neq\omega'\in\{0,1\}^3}\xi_{100,\omega'}+\sum_{\mathbf{0}\neq\omega'\in\{0,1\}^3}\xi_{110,\omega'}+\sum_{\mathbf{0}\neq\omega'\in\{0,1\}^3}\xi_{101,\omega'}+\sum_{\mathbf{0}\neq\omega'\in\{0,1\}^3}\xi_{111,\omega'}=0 \\
      &\sum_{\mathbf{0}\neq\omega'\in\{0,1\}^3}\xi_{010,\omega'}+\sum_{\mathbf{0}\neq\omega'\in\{0,1\}^3}\xi_{110,\omega'}+\sum_{\mathbf{0}\neq\omega'\in\{0,1\}^3}\xi_{011,\omega'}+\sum_{\mathbf{0}\neq\omega'\in\{0,1\}^3}\xi_{111,\omega'}=0 \\
      &\sum_{\mathbf{0}\neq\omega'\in\{0,1\}^3}\xi_{001,\omega'}+\sum_{\mathbf{0}\neq\omega'\in\{0,1\}^3}\xi_{101,\omega'}+\sum_{\mathbf{0}\neq\omega'\in\{0,1\}^3}\xi_{011,\omega'}+\sum_{\mathbf{0}\neq\omega'\in\{0,1\}^3}\xi_{111,\omega'}=0 \\
      &\sum_{\mathbf{0}\neq\omega'\in\{0,1\}^3}\xi_{110,\omega'}+\sum_{\mathbf{0}\neq\omega'\in\{0,1\}^3}\xi_{111,\omega'}=0 \\
      &\sum_{\mathbf{0}\neq\omega'\in\{0,1\}^3}\xi_{101,\omega'}+\sum_{\mathbf{0}\neq\omega'\in\{0,1\}^3}\xi_{111,\omega'}=0 \\
      &\sum_{\mathbf{0}\neq\omega'\in\{0,1\}^3}\xi_{011,\omega'}+\sum_{\mathbf{0}\neq\omega'\in\{0,1\}^3}\xi_{111,\omega'}=0
    \end{cases}
    \end{equation}
    (by looking at the terms involving only the $h_i$'s). Applying
    Lemma~\ref{lem:fixedeqns} to~\eqref{eq:system2} for each
    $\omega\in\{0,1\}^3$, we get that
    $\xi_{\omega,\omega'}=(-1)^{|\omega'|+1}\xi_{\omega,111}$ for each
    $\mathbf{0}\neq \omega'\in\{0,1\}^3$, so then by plugging this in
    to~\eqref{eq:system3} and then applying Lemma~\ref{lem:fixedeqns}
    again, it follows that
    $(\xi_{\omega,\omega'})_{\substack{\omega,\omega'\in\{0,1\}^3 \\
        \omega'\neq\mathbf{0}}}$ must be a scalar multiple of
    $((-1)^{|\omega|+|\omega'|})_{\substack{\omega,\omega'\in\{0,1\}^3\\
        \omega'\neq\mathbf{0}}}$.

    The upshot is that if~\eqref{eq:3cube} holds,~\eqref{eq:28}
    consists of $28$ distinct elements, and the matrices~\eqref{eq:LI}
    are linearly independent, then~\eqref{eq:inneraverage} equals $1$
    if and only if $(\xi_{\omega,\omega'})_{\substack{\omega,\omega'\in\{0,1\}^3 \\
        \omega'\neq\mathbf{0}}}$ is a scalar multiple of
    $((-1)^{|\omega|+|\omega'|})_{\substack{\omega,\omega'\in\{0,1\}^3\\
        \omega'\neq\mathbf{0}}}$ and
    $\sum_{\substack{\omega,\omega'\in\{0,1\}^3 \\ \omega'\neq
        \mathbf{0}}}\xi_{\omega,\omega'}\phi(x_{\omega,\omega'})=0$,
    and otherwise~\eqref{eq:inneraverage} equals zero. Thus,
    if~\eqref{eq:28} consists of $28$ distinct elements, the
    matrices~\eqref{eq:LI} are linearly independent, and
    $\square^2(x;h_1,h_2,h_3;\mathbf{k}_1,\mathbf{k}_2,\mathbf{k}_3)$
    is respected by $\phi$, then the probability that
    $\square^2(x;h_1,h_2,h_3;\mathbf{k}_1,\mathbf{k}_2,\mathbf{k}_3)\subset
    A$ is $(p^{-56}\cdot p)^k=p^{-55k}$, and if~\eqref{eq:28} consists
    of $28$ distinct elements,~\eqref{eq:LI} are linearly independent,
    and
    $\square^2(x;h_1,h_2,h_3;\mathbf{k}_1,\mathbf{k}_2,\mathbf{k}_3)$
    is not respected by $\phi$, then the probability that
    $\square^2(x;h_1,h_2,h_3;\mathbf{k}_1,\mathbf{k}_2,\mathbf{k}_3)\subset
    A$ is $(p^{-56}\cdot 1)^k=p^{-56k}$.

    The proportion of
    \begin{equation*}
      (x,h_1,h_2,h_3,\mathbf{k}_1,\mathbf{k}_2,\mathbf{k}_3)\in G^4\times\big(G^{\{0,1\}^3}\big)^3
    \end{equation*}
    for which~\eqref{eq:28} does not consist of $28$ distinct elements
    is at most $\binom{27}{2}p^{-n}=351 p^{-n}$ and, by
    Lemma~\ref{lem:linalg}, the proportion for which the
    matrices~\eqref{eq:LI} are linearly dependent is at most
    $p^{-n/16}$, provided that $n\geq 32\cdot 28\cdot 29$. In this
    case, then, the expected number of second-order $3$-dimensional
    cubes in $A$ respected by $\phi$ is at least
    $\delta^{16} p^{-55k}\cdot p^{28n}$ and the expected number of
    second-order $3$-dimensional cubes in $A$ not respected by $\phi$
    is at most $p^{-56k}\cdot p^{28n}+p^{(28-1/16)n}+351 p^{27n}$.

    Taking $k=\lceil\log_p(2/\delta^{16}\gamma)\rceil$, it follows
    that there exists a choice of
    $M_1,\dots,M_k\in \M_n(\mathbf{F}_p)$ and
    $v_1,\dots,v_k\in G$ such that $A$ contains at least
    $2^{-55}\delta^{56\cdot 16}\gamma^{55}p^{28n}$ second-order
    $3$-dimensional cubes and the proportion that are not respected is
    at most $\gamma$, provided that $n\geq 2^{16}k$, say. Since the number
    of second-order $3$-dimensional cubes in a subset of $G$ of
    density $\beta$ is at most $\beta^{28}p^{28n}$, it further follows
    that this $A$ must have density
    $\gg \delta^{16}\gamma^{55/28}\gg(\delta^{16}\gamma)^2$.
  \end{proof}
  
  \section{A quadratic Bogolyubov-type lemma}\label{sec:bogo}

  Next, we will show that $\mathcal{D}[1_A]$ is very close to constant
  on certain quadratic level sets. The proof is quite straightforward:
  all we need to do is apply a quadratic arithmetic regularity lemma
  to $1_A$, plug the decomposition into $\mathcal{D}[\cdot]$, and then
  use character sum bounds and a bit of linear algebra to handle the
  resulting expressions.

  We begin with some preliminaries on quadratic functions. Let $V$ be
  a $d$-dimensional vector space over $\mathbf{F}_p$. We say that a
  function $Q:V\to\mathbf{F}_p$ is \textit{quadratic} if
  $\partial_{h_1,h_2,h_3}Q(x)=0$ for all $x,h_1,h_2,h_3\in V$. If $W$ is a
  subspace of $V$, then $Q|_{W}:W\to\mathbf{F}_p$ is obviously also
  quadratic. By picking a basis $v_1,\dots,v_d$ for $V$, it is not
  hard to show that $Q$ is quadratic if and only if there exists
  $M\in\M_d(\mathbf{F}_p)$ symmetric (assuming that $p$ is odd, which
  is the case for us), $z\in\mathbf{F}_p^d$, and $c\in \mathbf{F}_p$
  such that
  $Q(x_1v_1+\dots+x_dv_d)=\mathbf{x}^TM\mathbf{x}+z\cdot\mathbf{x}+c$
  for all $\mathbf{x}=(x_1,\dots,x_d)\in\mathbf{F}_p^d$. We then say
  that the \textit{rank} of $Q$ is the rank of the matrix $M$;
  obviously, the rank does not depend on our choice of basis for
  $V$.

  Note that when $V\leq G$ is a subspace and $Q:G\to\mathbf{F}_p$ is a quadratic map
  such that $Q|_V$ has rank $R$, then for any coset $a+V\in G/V$,
  there exists a subspace $V_a'\leq V$ of codimension at most $R+2$
  such that $Q$ is constant on all cosets of $V'_a$ contained in
  $a+V$. Indeed, if $Q|_V$ has rank $R$, then there exist linear
  maps $L,L_1,\dots,L_R:V\to\mathbf{F}_p$ and constants $c\in\mathbf{F}_p$ and $\lambda_1,\dots,\lambda_R\in\mathbf{F}_p^\times$ such that $Q|_V(v)=\lambda_1L_1(v)^2+\dots+\lambda_RL_R(v)^2+L(v)+c$, and we can take
  \begin{equation*}
    V_a':=\ker{L}\cap\ker{L_1}\cap\dots\cap\ker{L_R}\cap\{v\in V:Q(v+a)-Q(v)-Q(a)+Q(0)=0\}.
  \end{equation*}

  We will also recall the following standard bounds for Gauss sums:
  \begin{lemma}\label{lem:GaussSum}
    For any symmetric $M\in\M_d(\mathbf{F}_p)$ with $\rank{M}=R$ and $v,v'\in\mathbf{F}_p^d$, we have
    \begin{equation*}
      \big|\mathbf{E}_{x,y\in \mathbf{F}_p^d}e_p(x^TMy+v\cdot x+v'\cdot y)\big|\leq p^{-R}
    \end{equation*}
    and
    \begin{equation*}
      \big|\mathbf{E}_{x\in \mathbf{F}_p^d}e_p(x^TMx+v\cdot x)\big|\leq p^{-R/2}.
    \end{equation*}
  \end{lemma}

  Lemma~\ref{lem:GaussSum} lets us easily handle character sums that
  naturally arise when analyzing $\mathcal{D}[1_A]$:
  
  \begin{lemma}\label{lem:Qomega}
    Let $p\geq 3$, $V\leq G$ be a subspace of dimension $d$,
    $L_1,L_2,L_3:V\to\mathbf{F}_p$ be linear, and, for each
    $\mathbf{0}\neq\omega\in\{0,1\}^3$, $Q_\omega:V\to\mathbf{F}_p$ be
    a quadratic function. Set
    \begin{equation}\label{eq:QomegaR}
      R=\max_{i=1,2,3}\rank\sum_{\substack{\omega\in\{0,1\}^3 \\ \omega_i=1}}Q_\omega\qquad\text{and}\qquad R'=\max_{1\leq i<j\leq 3}\rank\sum_{\substack{\omega\in\{0,1\}^3 \\ \omega_i=\omega_j=1}}Q_\omega.
    \end{equation}
    Then,
    \begin{equation}
      \label{eq:Qomegaavg}
      \mathbf{E}_{h_1,h_2,h_3\in V}e_p\Big(\sum_{\mathbf{0}\neq \omega\in\{0,1\}^3}Q_\omega(\omega\cdot(h_1,h_2,h_3))+L_1(h_1)+L_2(h_2)+L_3(h_3)\Big)
    \end{equation}
    has modulus at most $p^{-R/2}$ if $R\neq 0$ and modulus at most $p^{-R'}$ if $R=0$.
  \end{lemma}
  \begin{proof}
    By picking a basis $v_1,\dots,v_d$ for $V$, we may write
    $Q_\omega(v)=\mathbf{x}^TM_{\omega}\mathbf{x}+z_\omega\cdot \mathbf{x}+c_\omega$ for
    $v=x_1v_1+\dots+x_dv_d$ with each $M_\omega\in\M_d(\mathbf{F}_p)$ symmetric,
    $z_\omega\in\mathbf{F}_p^d$, and $c_\omega\in\mathbf{F}_p$. Then, there exist linear
    maps $L_1',L_2',L_3':\mathbf{F}_p^d\to\mathbf{F}_p$ such
    that~\eqref{eq:Qomegaavg} and
    \begin{equation}
      \label{eq:matrixavg}
      \mathbf{E}_{h_1,h_2,h_3\in\mathbf{F}_p^d}e_p\Big(\sum_{\mathbf{0}\neq\omega\in\{0,1\}^3}(\omega\cdot(h_1,h_2,h_3))^TM_\omega(\omega\cdot(h_1,h_2,h_3))+L_1'(h_1)+L_2'(h_2)+L_3'(h_3)\Big)
    \end{equation}
    have the same modulus. We may write the phase in the average~\eqref{eq:matrixavg} as
    \begin{align*}
      e_p\Bigg(&\sum_{\substack{\omega\in\{0,1\}^3 \\ \omega_1=1}}h_1^TM_\omega h_1+\sum_{\substack{\omega\in\{0,1\}^3 \\ \omega_2=1}}h_2^TM_\omega h_2+\sum_{\substack{\omega\in\{0,1\}^3 \\ \omega_3=1}}h_3^TM_\omega h_3 \\
      &+2\sum_{\substack{\omega\in\{0,1\}^3 \\ \omega_1=\omega_2=1}}h_1^TM_\omega h_2+2\sum_{\substack{\omega\in\{0,1\}^3 \\ \omega_1=\omega_3=1}}h_1^TM_\omega h_3+2\sum_{\substack{\omega\in\{0,1\}^3 \\ \omega_2=\omega_3=1}}h_2^TM_\omega h_3+\sum_{i=1}^3L_i'(h_i)\Bigg).
    \end{align*}
    Applying the triangle inequality and
    Lemma~\ref{lem:GaussSum}, it follows that~\eqref{eq:matrixavg} is
    at most $p^{-R/2}$. If $R=0$, in which case we must have
    \begin{equation}\label{eq:firstMomega}
      \sum_{\substack{\omega\in\{0,1\}^3 \\ \omega_1=1}}M_\omega=\sum_{\substack{\omega\in\{0,1\}^3 \\ \omega_2=1}}M_\omega=\sum_{\substack{\omega\in\{0,1\}^3 \\ \omega_3=1}}M_\omega=0,
    \end{equation}
    it follows from the triangle inequality and
    Lemma~\ref{lem:GaussSum} again that~\eqref{eq:matrixavg} is at
    most $p^{-R'}$.
  \end{proof}
  
  \subsection{The regularity lemma}\label{subsec:reg}

  The exact statement of the quadratic arithmetic regularity lemma we
  need is not already in the literature, due to the recency of
  polynomial bounds in the $U^3$-inverse theorem, but the proof of it
  essentially is (in several places). For example, using
  Theorem~\ref{thm:U3inv} in the proof of~\cite[Proposition
  3.7]{Green2007} and keeping track of the quantitative bounds yields
  the desired regularity lemma. For the sake of completeness, and
  because the notation introduced will be useful later,
  we will provide a full proof here.
  
  Recall that a \textit{factor} of $G$ is an algebra of subsets of
  $G$. An element $B\in\mathcal{B}$ is said to be an \textit{atom} if
  the only element of $\mathcal{B}$ strictly contained in $B$ is
  $\emptyset$. Note that the atoms of $\mathcal{B}$ form a partition
  of $G$ and that every element of $\mathcal{B}$ is a union of atoms
  of $\mathcal{B}$. We will denote the unique atom of $\mathcal{B}$
  containing $x\in G$ by $\mathcal{B}(x)$.

  A function $f:G\to \mathbf{C}$ is said to be
  $\mathcal{B}$\textit{-measurable} if $f^{-1}(z)\in\mathcal{B}$ for
  all $z\in \mathbf{C}$. We define the
  \textit{conditional expectation}
  $\mathbf{E}(f|\mathcal{B}):G\to\mathbf{C}$ of $f$ with respect to
  $\mathcal{B}$ by
  \begin{equation*}
    \mathbf{E}(f|\mathcal{B})(x):=\mathbf{E}_{y\in\mathcal{B}(x)}f(y),
  \end{equation*}
  so that $\mathbf{E}(f|\mathcal{B})$ is constant on each atom of
  $\mathcal{B}$. It is not hard to show the \textit{Pythagorean identity}
  \begin{equation}\label{eq:pythag}
    \|f\|_{L^2}^2=\|f-\mathbf{E}(f|\mathcal{B})\|_{L^2}^2+\|\mathbf{E}(f|\mathcal{B})\|_{L^2}^2.
  \end{equation}
  Another factor $\mathcal{B}'$ of $G$ \textit{refines} $\mathcal{B}$
  if $\mathcal{B}\subset\mathcal{B}'$. Note that $\mathcal{B}'$
  refines $\mathcal{B}$ if and only if each atom of $\mathcal{B}$ is a
  union of atoms of $\mathcal{B}'$. We also have the inequality
  \begin{equation}\label{eq:refine}
    \|\mathbf{E}(f|\mathcal{B})\|_{L^2}\leq\|\mathbf{E}(f|\mathcal{B}')\|_{L^2},
  \end{equation}
  which follows using the Pythagorean identity from the fact that
  $\mathbf{E}(\mathbf{E}(f|\mathcal{B}')|\mathcal{B})=\mathbf{E}(f|\mathcal{B})$
  whenever $\mathcal{B}'$ refines $\mathcal{B}$. Finally, for any two factors $\mathcal{B}_1$ and
  $\mathcal{B}_2$ of $G$, their \textit{join}
  $\mathcal{B}_1\vee\mathcal{B}_2$ is the minimal common refinement of
  $\mathcal{B}_1$ and $\mathcal{B}_2$; equivalently,
  $\mathcal{B}_1\vee\mathcal{B}_2$ is the factor whose atoms are the
  intersections $B_1\cap B_2$ of atoms $B_1$ of $\mathcal{B}_1$ and $B_2$ of
  $\mathcal{B}_2$.

  The only factors we will consider are those generated by level sets
  of polynomials of degree at most $2$. For any
  $P_1,\dots,P_m:G\to\mathbf{F}_p$, define
  $\mathcal{B}_{(P_1,\dots,P_m)}$ to be the factor whose atoms are the sets
  $P_1^{-1}(a_1)\cap\dots\cap
  P_m^{-1}(a_m)$ as $a_1,\dots,a_m$ range over $\mathbf{F}_p$. A
  factor of the form $\mathcal{B}_{\mathcal{P}}$ with
  $\mathcal{P}=(L_1,\dots,L_r,Q_1,\dots,Q_s)$ for
  $L_1,\dots,L_r:G\to\mathbf{F}_p$ homogeneous linear polynomials and
  $Q_1,\dots,Q_s:G\to\mathbf{F}_p$ quadratic polynomials is called a
  \textit{quadratic factor of complexity} $(r,s)$.

  Now, we can state the version of the quadratic arithmetic regularity lemma we will need.
  \begin{lemma}\label{lem:reglem}
    Let $f:G\to[0,1]$ and $\eta>0$. There exists a quadratic factor
    $\mathcal{B}$ of $G$ of complexity at most
    $(O(\eta^{-O(1)}),O(\eta^{-O(1)}))$ such that
    \begin{equation*}
      \|f-\mathbf{E}(f|\mathcal{B})\|_{U^3(G)}\leq\eta.
    \end{equation*}
  \end{lemma}

  As with~\cite[Proposition 3.7]{Green2007}, Lemma~\ref{lem:reglem}
  will be proven using a (now) standard energy-increment argument. The main
  input will be the following corollary of Theorem~\ref{thm:U3inv}.

\begin{corollary}\label{cor:energy}
  Let $f:G\to\mathbf{C}$ be $1$-bounded with
  $\|f\|_{U^3(G)}\geq\delta$. There exists a quadratic factor
  $\mathcal{B}$ of complexity at most $(1,1)$ such that
  \begin{equation*}
    \|\mathbf{E}(f|\mathcal{B})\|_{L^2}\gg\delta^{O(1)}.
  \end{equation*}
\end{corollary}
\begin{proof}
  By Theorem~\ref{thm:U3inv}, there exists a polynomial
  $P:G\to\mathbf{F}_p$ of the form $P(x)=x^TMx+v\cdot x$ with
  $M\in\M_n(\mathbf{F}_p)$ symmetric such that
  \begin{equation*}
    \left|\mathbf{E}_{x\in G}f(x)e_p(P(x))\right|\gg\delta^{O(1)}.
  \end{equation*}
  Set $Q(x):=x^TMx$ and $L(x):=v\cdot x$, and define the quadratic
  factor $\mathcal{B}:=\mathcal{B}_{(L,Q)}$, which has complexity at
  most $(1,1)$. Then, since $e_p(P(y))$ is constant on atoms of $\mathcal{B}$,
  \begin{equation*}
   \left|\mathbf{E}_{x\in G}f(x)e_p(P(x))\right|=\left|\mathbf{E}_{x\in G}\mathbf{E}_{y\in\mathcal{B}(x)}f(y)e_p(P(y))\right|\leq\mathbf{E}_{x\in G}\left|\mathbf{E}_{y\in\mathcal{B}(x)}f(y)\right|=\mathbf{E}_{x\in G}\left|\mathbf{E}(f|\mathcal{B})(x)\right|
  \end{equation*}
  The conclusion now follows from the Cauchy--Schwarz inequality.
\end{proof}

Next, we prove the energy-increment lemma that will be iterated to obtain Lemma~\ref{lem:reglem}.
\begin{lemma}\label{lem:iterate}
  Let $f:G\to[0,1]$ and $\mathcal{B}$ be a quadratic factor of $G$ of complexity at most
  $(r,s)$. If
  \begin{equation*}
    \|f-\mathbf{E}(f|\mathcal{B})\|_{U^3(G)}>\eta,
  \end{equation*}
  then there exists a quadratic factor $\mathcal{B}'$ of complexity at most
  $(r+1,s+1)$ refining $\mathcal{B}$ with
  \begin{equation*}
    \|\mathbf{E}(f|\mathcal{B}')\|_{L^2}^2\geq\|\mathbf{E}(f|\mathcal{B})\|_{L^2}^2+\Omega(\eta^{O(1)}).
  \end{equation*}
\end{lemma}
\begin{proof}
  Apply Corollary~\ref{cor:energy} to the function
  $f-\mathbf{E}(f|\mathcal{B})$, which is $1$-bounded. This gives a
  quadratic factor $\mathcal{B}_{L,Q}$ of complexity at most $(1,1)$
  such that
  \begin{equation*}
    \|\mathbf{E}(f-\mathbf{E}(f|\mathcal{B})|\mathcal{B}_{L,Q})\|_{L^2}\gg\eta^{O(1)}.
  \end{equation*}
  Set $\mathcal{B}':=\mathcal{B}\vee\mathcal{B}_{L,Q}$, so that
  $\mathcal{B}'$ is a quadratic factor of complexity at most
  $(r+1,s+1)$. Then,
  by~\eqref{eq:refine} and~\eqref{eq:pythag},
  \begin{equation*}
    \eta^{O(1)}\ll\|\mathbf{E}(f-\mathbf{E}(f|\mathcal{B})|\mathcal{B}')\|_{L^2}^2=\|\mathbf{E}(f|\mathcal{B}')-\mathbf{E}(f|\mathcal{B})\|_{L^2}^2=\|\mathbf{E}(f|\mathcal{B}')\|_{L^2}^2-\|\mathbf{E}(f|\mathcal{B})\|_{L^2}^2
  \end{equation*}
  since $\mathbf{E}(f|\mathcal{B})$ is already $\mathcal{B}'$-measurable.
\end{proof}

Now, we can prove Lemma~\ref{lem:reglem}.

\begin{proof}[Proof of Lemma~\ref{lem:reglem}]
  Beginning with the trivial (quadratic) factor
  $\mathcal{B}_0:=\{\emptyset,G\}$, repeatedly apply
  Lemma~\ref{lem:iterate} to refine $\mathcal{B}_0$ until obtaining a
  quadratic factor $\mathcal{B}$ for which
  $\|f-\mathbf{E}(f|\mathcal{B})\|_{U^3(G)}\leq\eta$. Since $f$ is $1$-bounded,
  $\|\mathbf{E}(f|\mathcal{B}')\|_{L^2}\leq 1$ for any factor
  $\mathcal{B}'$ of $G$, and thus this iteration must terminate in
  $\ll\eta^{-O(1)}$ steps, yielding that $\mathcal{B}$ has
  complexity at most $(O(\eta^{-O(1)}),O(\eta^{-O(1)}))$.
\end{proof}

\subsection{Analysis of dual functions}
Observe from the definition~\eqref{eq:dualdef} that
$\mathcal{D}[(f_\omega)_\omega]$ is multilinear in the $f_\omega$ and that,
if each of the functions $f_\omega$ is $1$-bounded, then
$\mathcal{D}[(f_\omega)_\omega]$ is $1$-bounded as well. In fact, more
is true: $\mathcal{D}[(f_\omega)_\omega]$ can be controlled in terms
of the $U^3$-norms of the $f_\omega$.

\begin{lemma}\label{lem:dualU3}
  Let $f_\omega:G\to\mathbf{C}$ be $1$-bounded functions for each $\mathbf{0}\neq\omega\in\{0,1\}^3$. Then,
  \begin{equation*}
    \|\mathcal{D}[(f_\omega)_\omega]\|_{L^\infty}\leq\min_{\mathbf{0}\neq\omega\in\{0,1\}^3}\|f_\omega\|_{U^3(G)}.
  \end{equation*}
\end{lemma}
\begin{proof}
  By the symmetry of the average~\eqref{eq:dualdef} defining
  $\mathcal{D}[(f_\omega)_{\omega}]$ in $h_1$, $h_2$, and $h_3$, it
  suffices to prove that $\|\mathcal{D}[(f_\omega)_\omega]\|_{L^\infty}$ is
  bounded above by each of $\|f_{100}\|_{U^3(G)}$, $\|f_{110}\|_{U^3(G)}$,
  and $\|f_{111}\|_{U^3(G)}$. For any $x\in G$, we have by applying the
  Cauchy--Schwarz inequality to double the $h_1$ variable in~\eqref{eq:dualdef} that
  $|\mathcal{D}[(f_\omega)_\omega](x)|^2$ is bounded above by
  \begin{equation*}
    \mathbf{E}_{h_1,h_2,h_3,k\in G}\Delta_{k}\overline{f_{100}}(x+h_1)\Delta_{k}f_{110}(x+h_1+h_2)\Delta_{k}f_{101}(x+h_1+h_3)\Delta_k\overline{f_{111}}(x+h_1+h_2+h_3).
  \end{equation*}
  Thus, by the triangle inequality and then the
  Gowers--Cauchy--Schwarz inequality in the form of
  Theorem~\ref{thm:GCSnorm},
  \begin{equation*}
    |\mathcal{D}[(f_\omega)_\omega](x)|^2 \leq \mathbf{E}_{k\in G}\|\Delta_kf_{100}\|_{U^2(G)}\|\Delta_kf_{110}\|_{U^2(G)}\|\Delta_kf_{101}\|_{U^2(G)}\|\Delta_kf_{111}\|_{U^2(G)}.
  \end{equation*}
  The conclusion of the lemma now follows from the $1$-boundedness of
  the $f_\omega$'s, H\"older's inequality, and the identity
  $\|f\|_{U^3(G)}^8=\mathbf{E}_{k\in G}\|\Delta_kf\|_{U^2(G)}^4$.
\end{proof}

Now, by combining Lemmas~\ref{lem:reglem} and~\ref{lem:dualU3}, we
obtain that $\mathcal{D}[1_A]$ is very well approximated by
$\mathcal{D}[\mathbf{E}(1_A|\mathcal{B})]$ for some quadratic factor
$\mathcal{B}$.

\begin{lemma}\label{lem:dualapprox}
  Let $A\subset G$ and $\eta>0$. There exists a quadratic factor
  $\mathcal{B}$ of $G$ of complexity at most
  $(O(\eta^{-O(1)}),O(\eta^{-O(1)}))$ such that
  \begin{equation*}
    \|\mathcal{D}[1_A]-\mathcal{D}[\mathbf{E}(1_A|\mathcal{B})]\|_{L^\infty}\leq 127\eta.
  \end{equation*}
\end{lemma}
\begin{proof}
  Applying Lemma~\ref{lem:reglem} to $1_A$, we obtain a quadratic
  factor $\mathcal{B}$ of $G$ of complexity at most
  $(O(\eta^{-O(1)}),O(\eta^{-O(1)}))$ such that
  $\|1_A-\mathbf{E}(1_A|\mathcal{B})\|_{U^3(G)}\leq\eta$. Writing
  $1_A=\mathbf{E}(1_A|\mathcal{B})+(1_A-\mathbf{E}(1_A|\mathcal{B}))$,
  then using the multilinearity of $\mathcal{D}$ and that both $\mathbf{E}(1_A|\mathcal{B})$
  and $(1_A-\mathbf{E}(1_A|\mathcal{B}))$ are $1$-bounded, it follows from Lemma~\ref{lem:dualU3} that
  \begin{equation*}
    \|\mathcal{D}[1_A]-\mathcal{D}[\mathbf{E}(1_A|\mathcal{B})]\|_{L^\infty}\leq 127\|1_A-\mathbf{E}(1_A|\mathcal{B})\|_{U^3(G)}\leq 127\eta.
  \end{equation*}
\end{proof}

If $\mathcal{B}$ is a quadratic factor of $G$ and $f$ is
$\mathcal{B}$-measurable, then it is not hard to show that
$\mathcal{D}[f]$ is very close to constant on the atoms of some
refinement of $\mathcal{B}$ of complexity not much larger than that of
$\mathcal{B}$ itself. The idea of the proof is to apply orthogonality
of characters to express $\mathcal{D}[f]$ as a linear combination of
character sums of the form appearing in Lemma~\ref{lem:Qomega}, and to
then split the linear combination into two parts depending on the size
of the corresponding $R$ appearing
in~\eqref{eq:QomegaR}. Lemma~\ref{lem:Qomega} says that the
``high-rank'' part of the linear combination is small, and the
quadratic functions appearing in the ``low rank'' part of the linear
combination live in a subspace of $\M_n(\mathbf{F}_p)$ of dimension
bounded in terms of the complexity of $\mathcal{B}$. This leads easily
to a suitable refinement with respect to which the low-rank part of
the linear combination is measurable. We carry out this argument to
prove the following lemma.

\begin{lemma}\label{lem:finalapprox}
  Let $\eta>0$ and $\mathcal{B}$ be a quadratic factor of $G$ of
  complexity at most $(r,s)$. There is a quadratic factor
  $\mathcal{B}'$ of $G$ of complexity at most
  $(15(r+s)^2+2s\log_p(2/\eta),s)$ refining $\mathcal{B}$ such that
  for any $\mathcal{B}$-measurable $f:G\to[0,1]$, there is a
  $\mathcal{B}'$-measurable $F:G\to[0,1]$ with
  \begin{equation*}
    \|\mathcal{D}[f]-F\|_{L^\infty}\leq \eta.
  \end{equation*}
\end{lemma}
\begin{proof}
  Write $\mathcal{B}=\mathcal{B}_{\mathcal{P}}$ for
  $\mathcal{P}=(L_1,\dots,L_r,Q_1,\dots,Q_s)$ with
  $L_1,\dots,L_r:G\to\mathbf{F}_p$ homogeneous linear polynomials and
  $Q_1,\dots,Q_s:G\to\mathbf{F}_p$ quadratic polynomials. We have
  $Q_i(x)=x^TM_ix+z_i\cdot x+c_i$ for some symmetric
  $M_i\in\M_n(\mathbf{F}_p)$, $z_i\in G$, and $c_i\in\mathbf{F}_p$ for
  every $i\in[s]$. Set $M(\xi_1,\dots,\xi_s):=\xi_1M_1+\dots+\xi_sM_s$
  for each $s$-tuple $(\xi_1,\dots,\xi_s)\in\mathbf{F}_p^s$, and
  consider the collection of matrices
  \begin{equation*}
    \mathcal{M}:=\{M(\xi_1,\dots,\xi_s):\xi_1,\dots,\xi_s\in\mathbf{F}_p\text{ with }\rank{M(\xi_1,\dots,\xi_s)}\leq 14(r+s)+2\log_p(2/\eta)\}.
  \end{equation*}
  The span of $\mathcal{M}$ over $\mathbf{F}_p$ is a subspace of
  dimension at most $s$ in $\M_n(\mathbf{F}_p)$ (as it is contained in
  the span of $M_1,\dots,M_s$), and has a basis of matrices in
  $\mathcal{M}$. This means that there exist matrices
  $M_1',\dots,M_{s'}'\in\M_n(\mathbf{F}_p)$ with $s'\leq s$, each
  satisfying $R_j:=\rank{M_j'}\leq 14(r+s)+2\log_p(2/\eta)$, such that
  every $M\in\mathcal{M}$ can be written as a linear combination of
  $M_1',\dots,M_{s'}'$. For each $j\in[s']$, we have
  \begin{equation*}
    M_j'=\sum_{i=1}^{R_j}\lambda_{i,j}v_{i,j}v_{i,j}^T
  \end{equation*}
  for some $v_{i,j}\in G$ and
  $\lambda_{i,j}\in\mathbf{F}_p^\times$. We then define homogeneous
  linear polynomials $L_{i,j},L_k':G\to\mathbf{F}_p$ by
  $L_{i,j}(x)=v_{i,j}\cdot x$ and $L'_k(x)=z_k\cdot x$, and set
  $\mathcal{B}'=\mathcal{B}_{\mathcal{P}'}$ for
  \begin{equation*}
    \mathcal{P}':=(L_1,\dots,L_r,L_{1,1},\dots,L_{R_1,1},\dots,L_{1,s'},\dots,L_{R_{s'},s'},L'_1,\dots,L'_s,Q_1,\dots,Q_s).
  \end{equation*}
  Note that $\mathcal{B}'$ refines $\mathcal{B}$ and has
  complexity at most $(r+s(14(r+s)+2\log_p(2/\eta)+1),s)$, which is
  certainly at most $(15(r+s)^2+2s\log_p(2/\eta),s)$, say.

  Now, we will show that $\mathcal{D}[f]$ can be approximated by a
  $\mathcal{B}'$-measurable function. Set
  $\mathcal{Q}=(Q_1,\dots,Q_s)$ and $V=\ker{L_1}\cap\dots\cap\ker{L_r}$. The atoms of
  $\mathcal{B}$ are the sets
  $\mathcal{A}(a,\mathbf{b}):=(a+V)\cap \mathcal{Q}^{-1}(\mathbf{b})$
  as $a+V$ ranges over $G/V$ and $\mathbf{b}$ ranges over
  $\mathbf{F}_p^s$. We can thus write $f$ as
  \begin{equation*}
    f=\sum_{a+V\in G/V}\sum_{\mathbf{b}\in\mathbf{F}_p^s}c(a,\mathbf{b})1_{\mathcal{A}(a,\mathbf{b})}
  \end{equation*}
  for $c(a,\mathbf{b})\in[0,1]$, so that
  \begin{equation*}
    \mathcal{D}[f]=\sum_{\substack{a_\omega+V\in G/V\\\mathbf{0}\neq\omega\in\{0,1\}^3 }}\sum_{\substack{ \mathbf{b}^{(\omega)}\in\mathbf{F}_p^s\\ \mathbf{0}\neq\omega\in\{0,1\}^3}}\prod_{\mathbf{0}\neq\omega\in\{0,1\}^3}c(a_\omega,\mathbf{b}^{(\omega)})\mathcal{D}[(1_{\mathcal{A}(a_\omega,\mathbf{b}^{(\omega)})})_\omega].
  \end{equation*}
  By orthogonality of characters, for every septuple of atoms 
  $\mathcal{D}[(1_{\mathcal{A}(a_\omega,\mathbf{b}^{(\omega)})})_{\omega}](x)$
  equals
  \begin{equation*}
    \mathbf{E}_{\mathbf{h}\in G^3}\mathbf{E}_{\substack{\xi^{(\omega)}_i,\eta^{(\omega)}_j\in\mathbf{F}_p\\\mathbf{0}\neq\omega\in\{0,1\}^3 \\ i\in[s],j\in[r]}}e_p\Big(\sum_{\mathbf{0}\neq\omega\in\{0,1\}^3}\Big[\sum_{i=1}^s\xi^{(\omega)}_i(Q_i(x+\omega\cdot\mathbf{h})-b_{i}^{(\omega)})+\sum_{j=1}^r\eta_{j}^{(\omega)}L_j(x+\omega\cdot\mathbf{h}-a_\omega)\Big]\Big).
  \end{equation*}
  Since each $Q_i$ is quadratic, $Q_i(x+y)=Q_i(y)+Q_i'(y,x)$ for some
  function $Q_i':G^2\to\mathbf{F}_p$ with $Q_i'(y,x)$ affine-linear in
  the $y$ variable. Thus, for each fixed choice of
  $\xi_i^{(\omega)}\in\mathbf{F}_p$ for $i\in[s]$ and
  $\mathbf{0}\neq\omega\in\{0,1\}^3$, setting $R:=14(r+s)+2\log_p(2/\eta)$ and
  \begin{equation*}
    Q_\omega:=\sum_{i=1}^s\xi_i^{(\omega)}[Q_i+Q_i'(\cdot,x)-b^{(\omega)}_{i}]
  \end{equation*}
  for each $\mathbf{0}\neq\omega\in\{0,1\}^3$, we see that by Lemma~\ref{lem:Qomega},
  \begin{equation}\label{eq:highrankbound}
    \Big|\mathbf{E}_{\mathbf{h}\in G^3}e_p\Big(\sum_{\mathbf{0}\neq\omega\in\{0,1\}^3}\Big[\sum_{i=1}^s\xi^{(\omega)}_i(Q_i(x+\omega\cdot\mathbf{h})-b_{i}^{(\omega)})+\sum_{j=1}^r\eta_{j}^{(\omega)}L_j(x+\omega\cdot\mathbf{h}-a_\omega)\Big]\Big)\Big|\leq p^{-R/2}
  \end{equation}
  unless
  \begin{equation*}
    \rank\sum_{\substack{\omega\in\{0,1\}^3\\\omega_1=1}}Q_\omega,\ \rank\sum_{\substack{\omega\in\{0,1\}^3\\\omega_2=1}}Q_\omega,\ \rank\sum_{\substack{\omega\in\{0,1\}^3\\\omega_3=1}}Q_\omega< R.
  \end{equation*}
  So, setting
  $M_\omega(\xi^{(\omega)}_1,\dots,\xi^{(\omega)}_s):=\xi_{1}^{(\omega)}M_1+\dots+\xi_{s}^{(\omega)}M_s$
  for each $\omega\in\{0,1\}^3$,~\eqref{eq:highrankbound} holds unless
  \begin{equation}\label{eq:smallrank}
    \sum_{\substack{\omega\in\{0,1\}^3 \\ \omega_1=1}}M_\omega(\xi_1^{(\omega)},\dots,\xi_s^{(\omega)}),\ \sum_{\substack{\omega\in\{0,1\}^3 \\ \omega_2=1}}M_\omega(\xi_1^{(\omega)},\dots,\xi_s^{(\omega)}),\ \sum_{\substack{\omega\in\{0,1\}^3 \\ \omega_3=1}}M_\omega(\xi_1^{(\omega)},\dots,\xi_s^{(\omega)})\in\mathcal{M}.
  \end{equation}

  Observe that, by plugging in the expressions $Q_i(x)=x^TM_ix+z_i\cdot x+c_i$ for each $i\in[s]$,
  \begin{equation*}
    \sum_{\mathbf{0}\neq\omega\in\{0,1\}^3}\sum_{i=1}^s\xi^{(\omega)}_iQ_i(x+\omega\cdot\mathbf{h})
  \end{equation*}
  equals
  \begin{align*}
        \sum_{\mathbf{0}\neq\omega\in\{0,1\}^3}\sum_{i=1}^s&\xi^{(\omega)}_iQ_i(x)+\mathcal{L}(x,\mathbf{h})+2\sum_{\substack{\omega\in\{0,1\}^3 \\ \omega_1=1}}h_1^TM_\omega(\xi_1^{(\omega)},\dots,\xi_s^{(\omega)})x\\
    &+2\sum_{\substack{\omega\in\{0,1\}^3 \\ \omega_2=1}}h_2^TM_\omega(\xi_1^{(\omega)},\dots,\xi_s^{(\omega)})x+2\sum_{\substack{\omega\in\{0,1\}^3 \\ \omega_3=1}}h_3^TM_\omega(\xi_1^{(\omega)},\dots,\xi_s^{(\omega)})x,
  \end{align*}
  where $\mathcal{L}(x,\mathbf{h})$ depends only on $\mathbf{h}$ and
  the values of $z_1\cdot x,\dots,z_s\cdot x$ (and, of course, the
  $\xi_i^{(\omega)}$'s). Thus, when~\eqref{eq:smallrank} holds,
  \begin{equation*}
    \mathbf{E}_{\mathbf{h}\in G^3}e_p\Big(\sum_{\mathbf{0}\neq\omega\in\{0,1\}^3}\Big[\sum_{i=1}^s\xi^{(\omega)}_i(Q_i(x+\omega\cdot\mathbf{h})-b_{i}^{(\omega)})+\sum_{j=1}^r\eta_{j}^{(\omega)}L_j(x+\omega\cdot\mathbf{h}-a_\omega)\Big]\Big)
  \end{equation*}
  is $\mathcal{B}'$-measurable, as it is a linear combination of
  $\mathcal{B}'$-measurable functions.

  Set, for each septuple of atoms of $\mathcal{B}$, $F_{(\mathcal{A}(a_\omega,\mathbf{b}^{(\omega)}))_\omega}(x)$ to equal
  \begin{equation*}
    \frac{1}{p^{7r+7s}}\sideset{}{'}\sum_{\substack{\xi^{(\omega)}_i,\eta^{(\omega)}_j\in\mathbf{F}_p\\\mathbf{0}\neq\omega\in\{0,1\}^3 \\ i\in[s],j\in[r]}} \mathbf{E}_{\mathbf{h}\in G^3}e_p\Big(\sum_{\mathbf{0}\neq\omega\in\{0,1\}^3}\Big[\sum_{i=1}^s\xi^{(\omega)}_i(Q_i(x+\omega\cdot\mathbf{h})-b_{i}^{(\omega)})+\sum_{j=1}^r\eta_{j}^{(\omega)}L_j(x+\omega\cdot\mathbf{h}-a_\omega)\Big]\Big),
  \end{equation*}
  where the sum is restricted to $\xi^{(\omega)}_i$'s
  satisfying~\eqref{eq:smallrank}, and
  $H_{(\mathcal{A}(a_\omega,\mathbf{b}^{(\omega)}))_\omega}:=\mathcal{D}[(1_{\mathcal{A}(a_\omega,\mathbf{b}^{(\omega)})})_\omega]-F_{(\mathcal{A}(a_\omega,\mathbf{b}^{(\omega)}))_\omega}$. Then,
  $F_{(\mathcal{A}(a_\omega,\mathbf{b}^{(\omega)}))_\omega}$ is
  $\mathcal{B}'$-measurable and
  $\|H_{(\mathcal{A}(a_\omega,\mathbf{b}^{(\omega)}))_\omega}\|_{L^\infty}\leq
  p^{-R/2}$ for each septuple of atoms. Since we have

  \begin{equation*}
    \Bigg\|\sum_{\substack{a_\omega+V\in G/V\\\mathbf{0}\neq\omega\in\{0,1\}^3 }}\sum_{\substack{ \mathbf{b}^{(\omega)}\in\mathbf{F}_p^s\\ \mathbf{0}\neq\omega\in\{0,1\}^3}}\Big(\prod_{\mathbf{0}\neq\omega\in\{0,1\}^3}c(a_\omega,\mathbf{b}^{(\omega)})\Big)H_{(\mathcal{A}(a_\omega,\mathbf{b}^{(\omega)}))_\omega}\Bigg\|_{L^\infty}\leq p^{7r+7s}p^{-R/2}\leq\eta
  \end{equation*}
  by the triangle inequality and the choice of $R$, defining $F_0:G\to\mathbf{C}$ by
  \begin{equation*}
    F_0(x):=\sum_{\substack{a_\omega+V\in G/V\\\mathbf{0}\neq\omega\in\{0,1\}^3 }}\sum_{\substack{ \mathbf{b}^{(\omega)}\in\mathbf{F}_p^s\\ \mathbf{0}\neq\omega\in\{0,1\}^3}}\Big(\prod_{\mathbf{0}\neq\omega\in\{0,1\}^3}c(a_\omega,\mathbf{b}^{(\omega)})\Big)F_{(\mathcal{A}(a_\omega,\mathbf{b}^{(\omega)}))_\omega},
  \end{equation*}
  we get that $\|\mathcal{D}[f]-F_0\|_{L^\infty}\leq \eta$ and that
  $F_0$ is $\mathcal{B}'$-measurable. Finally, setting
  $F:=g\circ\re{F_0}$ for $g:\mathbf{R}\to\mathbf{R}$ with $g(x)$
  equaling $0$ on $(-\infty,0]$, $1$ on $[1,\infty)$, and $x$ on
  $[0,1]$, we have $\|\mathcal{D}[f]-F\|_{L^\infty}\leq\eta$ as well
  since $\mathcal{D}[f]$ takes values in $[0,1]$, completing the
  proof.
\end{proof}

Combining Lemmas~\ref{lem:dualapprox} and~\ref{lem:finalapprox} now yields our Bogolyubov-type lemma.
\begin{lemma}\label{lem:bogo}
  Let $A\subset G$ and $\eta>0$. There exists a quadratic factor
  $\mathcal{B}$ of complexity at most
  $(O(\eta^{-O(1)}),O(\eta^{-O(1)}))$ and a
  $\mathcal{B}$-measurable function $F:G\to[0,1]$ such that
  \begin{equation*}
    \|\mathcal{D}[1_A]-F\|_{L^\infty}\leq\eta.
  \end{equation*}
\end{lemma}

\subsection{Lemmas on quadratic level sets}
We end this section with two basic lemmas, which we will use later to
help us locate a high-rank quadratic level set on which a closely
related function to $\phi$ is a $(1-\varepsilon)$-approximate
quadratic. The first lemma estimates the size of a high-rank quadratic
level set, and the second estimates the number of $3$-dimensional
cubes whose points lie in a fixed octuple of high-rank quadratic level
sets, and both proofs consist simply of applying orthogonality of
characters and then invoking character sum bounds.
\begin{lemma}\label{lem:atomsize}
  Let $V\leq G$ be a subspace of $G$, $Q_1,\dots,Q_s:G\to\mathbf{F}_p$
  be quadratic functions such that
  $\lambda_1Q_1|_V+\dots+\lambda_s Q_s|_V$ has rank at least $R$
  whenever $\lambda_1,\dots,\lambda_s\in\mathbf{F}_p$ are not all
  zero, $a\in G$, and $\mathbf{b}\in\mathbf{F}^s_p$. Then, setting
  $\mathcal{Q}=(Q_1,\dots,Q_s)$, we have
  \begin{equation*}
    \left||(a+V)\cap\mathcal{Q}^{-1}(\mathbf{b})|-p^{-s}|V|\right|\leq p^{-R/2}|V|.
  \end{equation*}
\end{lemma}
\begin{proof}
  By orthogonality of characters, $|(a+V)\cap\mathcal{Q}^{-1}(\mathbf{b})|$ can be written as
  \begin{equation*}
     \sum_{x\in V}\mathbf{E}_{\xi_1,\dots,\xi_s\in\mathbf{F}_p}e_p\Big(\sum_{i=1}^s\xi_i[Q_i(x+a)-b_i]\Big)= \mathbf{E}_{\xi_1,\dots,\xi_s\in\mathbf{F}_p}\sum_{x\in V}e_p\Big(\sum_{i=1}^s\xi_i[Q_i(x)+Q'_i(x,a)-b_i]\Big)
  \end{equation*}
  for some $Q'_i:G^2\to\mathbf{F}_p$ with $Q'_i(x,a)$ affine-linear
  in the $x$ variable. By Lemma~\ref{lem:GaussSum} and the rank
  assumption on the $Q_i$, we have that
  \begin{equation*}
    \Big|\sum_{x\in V}e_p\Big(\sum_{i=1}^s\xi_i[Q_i(x)+Q'_i(x,a)-b_i]\Big)\Big|\leq |V|p^{-R/2}
  \end{equation*}
  unless $\xi_1=\dots=\xi_s=0$, in which case the sum on the left-hand
  side above equals $|V|$.
\end{proof}
\begin{lemma}\label{lem:count}
  Let $V\leq G$ be a subspace of $G$ of codimension $r$,
  $Q_1,\dots,Q_s:G\to\mathbf{F}_p$ be quadratic functions such that
  $\lambda_1Q_1|_V+\dots+\lambda_sQ_s|_V$ has rank at least $R$ whenever
  $\lambda_1,\dots,\lambda_s\in\mathbf{F}_p$ are not all zero, $a_\omega\in G$ for
  each $\omega\in\{0,1\}^3$, and $\mathbf{b}^{(\omega)}\in\mathbf{F}_p^s$
  for each $\omega\in\{0,1\}^3$. Set $\mathcal{Q}=(Q_1,\dots,Q_s)$ and
  $f_\omega(x):=1_{a_\omega+V}(x)\cdot
  1_{\mathcal{Q}(x)=\mathbf{b}^{(\omega)}}$. If the $a_\omega$'s satisfy
    \begin{equation}\label{eq:acondition}
      \begin{cases}
        &a_{010}-a_{001}-a_{110}+a_{101}\in V \\
        &a_{100}-a_{001}-a_{110}+a_{011}\in V \\
        &a_{100}+a_{010}-a_{001}-2a_{110}+a_{111}\in V \\
        &a_{000}-a_{100}-a_{010}+a_{110}\in V \\
    \end{cases}
    \end{equation}
    and the $\mathbf{b}^{(\omega)}$'s satisfy
    \begin{equation}\label{eq:bcondition}
      \sum_{\omega\in\{0,1\}^3}(-1)^{|\omega|}\mathbf{b}^{(\omega)}=0,
    \end{equation}
    then
    \begin{equation}\label{eq:count}
      \Big|\mathbf{E}_{x,h_1,h_2,h_3\in G}\prod_{\omega\in\{0,1\}^3}f_\omega(x+\omega\cdot(h_1,h_2,h_3))-p^{-4r-7s}\Big|\leq 2p^{-4r-R/2}.
    \end{equation}
  \end{lemma}
  \begin{proof}
    Setting $a_0:=a_{100}+a_{010}-a_{110}$, $a_1:=a_{110}-a_{010}$,
    $a_2:=a_{110}-a_{100}$, $a_3:=a_{110}-a_{100}-a_{010}+a_{001}$,
    we have by~\eqref{eq:acondition} that the average appearing on the left-hand side of~\eqref{eq:count} equals
    \begin{align*}
      p^{-4r}\mathbf{E}_{x,h_1,h_2,h_3\in V}\prod_{\omega\in\{0,1\}^3}1_{\mathcal{Q}(x+a_0+\omega\cdot(h_1+a_1,h_2+a_2,h_3+a_3))=\mathbf{b}^{(\omega)}}.
    \end{align*}
    By orthogonality of characters, the right-hand side above can be written as $p^{-4r}$ times
    \begin{equation}\label{eq:charsumrep}
      \mathbf{E}_{\substack{\xi_i^{(\omega)}\in\mathbf{F}_p \\ i\in[s],\omega\in\{0,1\}^3}}\mathbf{E}_{x,h_1,h_2,h_3\in V}e_p\Big(\sum_{\omega\in\{0,1\}^3}\sum_{i=1}^s\xi_i^{(\omega)}[Q_i(x+a_0+\omega\cdot(h_1+a_1,h_2+a_2,h_3+a_3))-b_i^{(\omega)}]\Big).
    \end{equation}
    Since each $Q_i$ is quadratic, $Q_i(x+y)=Q_i(y)+Q'_i(y,x)$ for some function $Q'_i:G^2\to\mathbf{F}_p$ with $Q'_i(y,x)$ affine-linear in the $y$ variable. Thus, for each fixed choice of $\xi_i^{(\omega)}\in\mathbf{F}_p$ for $i\in[s]$ and $\omega\in\{0,1\}^3$, setting
    \begin{equation*}
      Q_\omega:=\sum_{i=1}^s\xi_i^{(\omega)}[Q_i+Q'_i(\cdot,x+a_0+\omega\cdot(a_1,a_2,a_3))-b_i^{(\omega)}],
    \end{equation*}
    we see that by Lemma~\ref{lem:Qomega},
    \begin{equation}\label{eq:R2}
      \Big|\mathbf{E}_{h_1,h_2,h_3\in V}e_p\Big(\sum_{\omega\in\{0,1\}^3}\sum_{i=1}^s\xi_i^{(\omega)}[Q_i(x+a_0+\omega\cdot(h_1+a_1,h_2+a_2,h_3+a_3))-b_i^{(\omega)}]\Big) \Big|\leq p^{-R/2}
    \end{equation}
    for each $x\in G$, unless
    \begin{equation*}
      \max_{i=1,2,3}\rank\sum_{\substack{\omega\in\{0,1\}^3 \\ \omega_i=1}}Q_\omega|_V,\ \max_{1\leq i<j\leq 3}\rank\sum_{\substack{\omega\in\{0,1\}^3 \\ \omega_i=\omega_j=1}}Q_\omega|_V<R.
    \end{equation*}

    By Lemma~\ref{lem:fixedeqns} and the rank assumption on
    $Q_1,\dots,Q_s$, this means that~\eqref{eq:R2} holds unless
    $\xi_i^{(\omega)}=(-1)^{|\omega|+1}\xi^{(111)}_i$ for all
    $i\in[s]$ and $\mathbf{0}\neq\omega\in\{0,1\}^3$. Thus, using that
    all $3$-fold additive discrete derivatives of $Q_i$ vanish, we get
    that~\eqref{eq:charsumrep} equals $p^{-7s}$ times
    \begin{equation*}
      \sum_{\eta_1,\dots,\eta_s\in\mathbf{F}_p}\mathbf{E}_{\xi_1,\dots,\xi_s\in\mathbf{F}_p}\mathbf{E}_{x\in V}e_p\Big(\sum_{i=1}^s\Big(\xi_i\Big[Q_i(x+a_0)-b_i^{(000)}\Big]+\eta_i\Big[Q_i(x+a_0)+\sum_{\mathbf{0}\neq \omega\in\{0,1\}^3}(-1)^{|\omega|}b_i^{(\omega)}\Big]\Big)\Big)
    \end{equation*}
    plus an error of size at most $p^{-R/2}$. By
    Lemma~\ref{lem:GaussSum} and the expression
    $Q_i(x+y)=Q_i(y)+Q'_i(y,x)$ from above, the innermost average above
    has size at most $p^{-R/2}$ unless $\xi_i=-\eta_i$ for all
    $i\in[s]$. The conclusion of the lemma now follows
    from~\eqref{eq:bcondition}.
  \end{proof}

\section{Obtaining a $99\%$ approximate quadratic on a quadratic level set}\label{sec:random}

In Section~\ref{sec:DRC}, we showed that if $\phi:G\to G$ is
$\delta$-approximately quadratic, then there exists a dense subset
$A\subset G$ containing many second-order $3$-dimensional cubes,
almost all of which are respected by $\phi$. We will now combine
Lemma~\ref{lem:bogo} with a random sampling argument to find a closely
related $(1-\varepsilon)$-approximate quadratic on a high-rank
quadratic level set.

This section and the previous one roughly correspond to Sections 4--6
of~\cite{GowersMilicevic2017}, but, instead of employing a random
sampling argument, Gowers and Mili\'cevi\'c work with functions to the
group algebra $\mathbf{C}[G]$. In our setting, this is equivalent to
defining a function on the support of $\mathcal{D}[1_A]$ by setting
its value to equal~\eqref{eq:psiprelim} for $h_1,h_2,h_3\in G$ chosen
uniformly at random subject to the restriction that the partial
$3$-dimensional cube $\square^*(x;h_1,h_2,h_3):=\{x+\omega\cdot(h_1,h_2,h_3):\mathbf{0}\neq\omega\in\{0,1\}^3\}$ is
contained in $A$. By modifying this distribution slightly, we can
obtain a suitable approximate quadratic quite easily by first moment
considerations. The following lemma describes our random sampling
procedure, along with its key properties.

\begin{lemma}\label{lem:rselect}
  Let $\theta>0$, $A\subset G$, and $\phi:A\to G$, and set, for all $x,b\in G$, $\nu_x(b)$ to equal
  \begin{equation*}
    \Big|\Big\{(h_1,h_2,h_3)\in G^3:\square^{*}(x;h_1,h_2,h_3)\subset A\text{ and }\sum_{\mathbf{0}\neq\omega\in\{0,1\}^3}(-1)^{|\omega|}\phi(x+\omega\cdot(h_1,h_2,h_3))=b\Big\}\Big|.
  \end{equation*}
  Define $\psi:G\to G$ randomly as follows: for each $x$
  independently,
  \begin{enumerate}
  \item if $\mathcal{D}[1_A](x)<\theta$, set $\psi(x)=c$ for a uniformly random $c\in G$, and
  \item if $\mathcal{D}[1_A](x)\geq\theta$, pick $b\in G$ at random with probability
    \begin{equation*}
      \nu'_x(b):=\frac{\nu_x(b)}{|G|^3\mathcal{D}[1_A](x)};
    \end{equation*}
    \begin{enumerate}
    \item if $\nu'_x(b)<\theta$, set $\psi(x)=c$ for a uniformly random $c\in G$, and
    \item if $\nu'_x(b)\geq\theta$, set $\psi(x)=b$.
    \end{enumerate}
  \end{enumerate}
  Then,
  \begin{equation}\label{eq:expectedpsi}
    \mathbf{E}_{\psi}\mathbf{E}_{x,h_1,h_2,h_3\in G}\Delta_{h_1,h_2,h_3}\mathcal{D}[1_A](x)1_{\partial_{h_1,h_2,h_3}\psi(x)=0}
  \end{equation}
  is at least
  \begin{equation}\label{eq:expnumber}
    \frac{|\{\text{second-order }3\text{-dimensional cubes in }A\text{ respected by }\phi\}|}{|G|^{28}}-60(\theta+p^{-n}).
  \end{equation}
  Further, setting
  \begin{equation*}
    S_{\phi,A,\psi}(\theta):=\left\{x\in G:\nu_x(\psi(x))\geq\theta^2|G|^3\right\},
  \end{equation*}
  then with probability at least $1-\frac{45}{p^{n/2}}$, the number of $3$-dimensional cubes
  $\square(x;h_1,h_2,h_3)\not\subset S_{\phi,A,\psi}(\theta)$ that are respected by $\psi$ is at most $p^{7n/2}$.
\end{lemma}
\begin{proof}
  By linearity of expectation,~\eqref{eq:expectedpsi} equals
  \begin{equation}\label{eq:expectedpsi2}
    \mathbf{E}_{x,h_1,h_2,h_3\in G}\Delta_{h_1,h_2,h_3}\mathcal{D}[1_A](x)\mathbf{E}_{\psi}1_{\partial_{h_1,h_2,h_3}\psi(x)=0}.
  \end{equation}
  The contribution to~\eqref{eq:expectedpsi2} coming from degenerate
  $\square(x;h_1,h_2,h_3)$ is at most $\binom{8}{2}p^{-n}=28 p^{-n}$,
  and the contribution to~\eqref{eq:expectedpsi2} coming from
  $(x,h_1,h_2,h_3)\in G^4$ for which $\mathcal{D}[1_A](y)<\theta$ for
  some $y=x+\omega\cdot(h_1,h_2,h_3)$ is less than $\theta$. When
  $\square(x;h_1,h_2,h_3)$ is nondegenerate and
  $\mathcal{D}[1_A](y)\geq\theta$ for all
  $y=x+\omega\cdot(h_1,h_2,h_3)$, we have that $\mathbf{E}_{\psi}1_{\partial_{h_1,h_2,h_3}\psi(x)=0}$ equals
  \begin{equation*}
    \sum_{\substack{b_{\omega}\in G \\ \sum_{\omega\in\{0,1\}^3}(-1)^{|\omega|}b_\omega=0}}\prod_{\omega\in\{0,1\}^3}\mathbf{E}_{\psi}1_{\psi(x+\omega\cdot(h_1,h_2,h_3))=b_\omega},
  \end{equation*}
  which is bounded below by
  \begin{equation}\label{eq:nuprime}
    \sum_{\substack{b_{\omega}\in G \\ \sum_{\omega\in\{0,1\}^3}(-1)^{|\omega|}b_\omega=0}}\prod_{\omega\in\{0,1\}^3}\nu'_{x+\omega\cdot(h_1,h_2,h_3)}(b_\omega)1_{\nu'_{x+\omega\cdot(h_1,h_2,h_3)}(b_\omega)\geq\theta}.
  \end{equation}
  Writing
  $1_{\nu'_{x+\omega\cdot(h_1,h_2,h_3)}(b_\omega)\geq\theta}=1-1_{\nu'_{x+\omega\cdot(h_1,h_2,h_3)}(b_\omega)<\theta}$
  and using that $\sum_{b\in G}\nu'_{y}(b)=1$ for all $y\in G$ for
  which $\mathcal{D}[1_A](y)\neq 0$ shows that~\eqref{eq:nuprime} is at least
  \begin{equation*}
    \sum_{\substack{b_{\omega}\in G \\ \sum_{\omega\in\{0,1\}^3}(-1)^{|\omega|}b_\omega=0}}\prod_{\omega\in\{0,1\}^3}\nu'_{x+\omega\cdot(h_1,h_2,h_3)}(b_\omega)-8\theta.
  \end{equation*}
  
  Now note that, by expanding the definition of $\nu_y(b)$, the sum
  \begin{equation}\label{eq:unnormalized}
    \sum_{x,h_1,h_2,h_3\in G}\sum_{\substack{b_{\omega}\in G \\ \sum_{\omega\in\{0,1\}^3}(-1)^{|\omega|}b_\omega=0}}\prod_{\omega\in\{0,1\}^3}\nu_{x+\omega\cdot(h_1,h_2,h_3)}(b_\omega)
  \end{equation}
  equals the number of second-order $3$-dimensional cubes in $A$ respected by $\phi$, and that
  \begin{equation*}
    \Delta_{h_1,h_2,h_3}\mathcal{D}[1_A](x)\cdot \prod_{\omega\in\{0,1\}^3}\nu'_{x+\omega\cdot(h_1,h_2,h_3)}(b_\omega)=|G|^{-24}\prod_{\omega\in\{0,1\}^3}\nu_{x+\omega\cdot(h_1,h_2,h_3)}(b_\omega)
  \end{equation*}
  when $\Delta_{h_1,h_2,h_3}\mathcal{D}[1_A](x)>0$. To complete the proof of the first
  claim in the lemma, it suffices to observe that, for all
  $x,h_1,h_2,h_3\in G$,
  \begin{equation*}
    0\leq \sum_{\substack{b_{\omega}\in G \\ \sum_{\omega\in\{0,1\}^3}(-1)^{|\omega|}b_\omega=0}}\prod_{\omega\in\{0,1\}^3}\nu_{x+\omega\cdot(h_1,h_2,h_3)}(b_\omega)\leq \prod_{\omega\in\{0,1\}^3}|G|^3\mathcal{D}[1_A](x+\omega\cdot(h_1,h_2,h_3)).
  \end{equation*}
  Thus, the contribution to~\eqref{eq:unnormalized} coming from
  degenerate $\square(x;h_1,h_2,h_3)$ is at most
  $28|G|^{27}$ and from $(x,h_1,h_2,h_3)$ such
  that $\mathcal{D}[1_A](y)<\theta$ for some
  $y\in\square(x;h_1,h_2,h_3)$ is at most
  $\theta|G|^{28}$. It follows
  that~\eqref{eq:expectedpsi2} is at least~\eqref{eq:expnumber}, as
  desired.
  
  To prove the second claim, note that by the symmetry of
  $x+\omega\cdot(h_1,h_2,h_3)$, an upper bound on the expected number
  of nondegenerate $\square(x;h_1,h_2,h_3)$ respected by $\psi$ that
  are not contained in $S_{\phi,A,\psi}(\theta)$ is $8$ times
  \begin{equation*}
    \mathbf{E}_{\psi}\sum_{x,h_1,h_2,h_3\in G}1_{\partial_{h_1,h_2,h_3}\psi(x)=0}1_{x\notin S_{\phi,A,\psi}(\theta)},
  \end{equation*}
  which is bounded above by
  \begin{align*}
    &\mathbf{E}_{\psi}\sum_{x,h_1,h_2,h_3\in G}1_{\partial_{h_1,h_2,h_3}\psi(x)=0}\left(1_{\mathcal{D}[1_A](x)<\theta}+1_{\mathcal{D}[1_A](x)\geq\theta\text{ and }\nu'_x(\psi(x))<\theta}\right)\\
    &= \sum_{x,h_1,h_2,h_3\in G}1_{\mathcal{D}[1_A](x)<\theta}\mathbf{E}_{\psi}1_{\partial_{h_1,h_2,h_3}\psi(x)=0}+\sum_{x,h_1,h_2,h_3\in G}1_{\mathcal{D}[1_A](x)\geq\theta}\mathbf{E}_{\psi}1_{\partial_{h_1,h_2,h_3}\psi(x)=0}1_{\nu'_x(\psi(x))<\theta}.
  \end{align*}
  The first sum above is at most $p^{3n}$, and the second sum can be written as
  \begin{align*}
    \sum_{x,h_1,h_2,h_3\in G}1_{\mathcal{D}[1_A](x)\geq\theta}\mathbf{E}_{\substack{\psi(x+\omega\cdot(h_1,h_2,h_3)) \\ \mathbf{0}\neq \omega\in\{0,1\}^3}}&1_{\nu'_x(-\sum_{\mathbf{0}\neq\omega\in\{0,1\}^3}(-1)^{|\omega|}\psi(x+\omega\cdot(h_1,h_2,h_3)))<\theta} \\
                                                                                                                                               &\cdot\mathbf{E}_{\psi(x)}1_{\psi(x)=-\sum_{\mathbf{0}\neq\omega\in\{0,1\}^3}(-1)^{|\omega|}\psi(x+\omega\cdot(h_1,h_2,h_3))},
  \end{align*}
  which is also at most $p^{3n}$. Since the number of degenerate
  $3$-dimensional cubes in $G$ is at most $28 p^{3n}$, the second
  claim in the lemma follows from Markov's inequality.
\end{proof}

We will also need a pair of straightforward lemmas. The first will let
us replace $\mathcal{D}[1_A]$ with a $\mathcal{B}$-measurable function
in~\eqref{eq:expectedpsi}, and the second will let us refine
collections of atoms of $\mathcal{B}$ to high-rank quadratic level
sets.
\begin{lemma}\label{lem:dapproxinf}
  Let $F,F':G\to[0,1]$, and assume that $\|F\|_{U^3(G)}^8\geq\gamma'$, $\|F-F'\|_{L^\infty}\leq\varepsilon$, and 
  \begin{equation*}
    \mathbf{E}_{x,h_1,h_2,h_3\in G}\Delta_{h_1,h_2,h_3}F(x)1_{\partial_{h_1,h_2,h_3}\psi(x)=0}\geq(1-\gamma)\|F\|_{U^3(G)}^8.
  \end{equation*}
  If $\varepsilon\leq\frac{\gamma\gamma'}{2^8}$, then
  \begin{equation*}
    \mathbf{E}_{x,h_1,h_2,h_3\in G}\Delta_{h_1,h_2,h_3}F'(x)1_{\partial_{h_1,h_2,h_3}\psi(x)=0}\geq(1-4\gamma)\|F'\|_{U^3(G)}^8.
  \end{equation*}
\end{lemma}
\begin{proof}
  By the reverse triangle inequality for the $U^3$-norm followed by the triangle
  inequality,
  \begin{equation*}
    \left|\|F\|_{U^3(G)}-\|F'\|_{U^3(G)}\right|\leq \|F-F'\|_{U^3(G)}\leq \varepsilon.
  \end{equation*}
  Given $f_\omega:G\to\mathbf{C}$ $1$-bounded for each
  $\omega\in\{0,1\}^3$, using the triangle inequality again yields
  \begin{equation*}
    \Big|\mathbf{E}_{x,h_1,h_2,h_3\in G}\prod_{\omega\in\{0,1\}^3}f_\omega(x+\omega\cdot(h_1,h_2,h_3))1_{\partial_{h_1,h_2,h_3}\psi(x)=0}\Big|\leq \min_{\omega\in\{0,1\}^3}\|f_\omega\|_{L^\infty}.
  \end{equation*}
  Hence, writing $F'=F+F'-F$ and expanding, as long as
  $\varepsilon<\frac{\gamma\gamma'}{2^8}$, we have that
  \begin{equation*}
    \mathbf{E}_{h_1,h_2,h_3\in G}\Delta_{h_1,h_2,h_3}F'(x)1_{\partial_{h_1,h_2,h_3}\psi(x)=0}\geq(1-2\gamma)\|F\|_{U^3(G)}^8\geq (1-2\gamma)\|F'\|_{U^3(G)}^8-8\varepsilon,
  \end{equation*}
  which is greater than $(1-4\gamma)\|F'\|_{U^3(G)}$.
\end{proof}

  \begin{lemma}\label{lem:rankrefine}
    Let $V\leq G$ be a subspace of $G$,
    $Q_1,\dots,Q_s:G\to\mathbf{F}_p$ be quadratic functions with $s\geq 1$,
    $a_\omega\in G$ for each $\omega\in\{0,1\}^3$, and
    $\mathbf{b}^{(\omega)}\in\mathbf{F}^s_p$ for each
    $\omega\in\{0,1\}^3$. Suppose that there exist
    $\lambda_1,\dots,\lambda_{s-1}\in\mathbf{F}_p$ such that
    $\lambda_1Q_1|_V+\dots+\lambda_{s-1}Q_{s-1}|_{V}-Q_s|_V$ has rank less than $R$ and that
    \begin{align*}
      \mathbf{E}_{x,h_1,h_2,h_3\in G}\prod_{\omega\in\{0,1\}^3}&1_{\mathcal{A}_\omega}(x+\omega\cdot(h_1,h_2,h_3))1_{\partial_{h_1,h_2,h_3}\psi(x)=0}\\
      &\geq(1-\gamma)\mathbf{E}_{x,h_1,h_2,h_3\in G}\prod_{\omega\in\{0,1\}^3}1_{\mathcal{A}_\omega}(x+\omega\cdot(h_1,h_2,h_3))>0,
    \end{align*}
    where
    $\mathcal{A}_\omega:=(a_\omega+V)\cap
    Q_1^{-1}(b_1^{(\omega)})\cap\dots\cap Q_s^{-1}(b_s^{(\omega)})$
    for each $\omega\in\{0,1\}^3$. Then, there exists a subspace
    $W\leq V$ of codimension at most $8(R+1)$ in $V$ and quadratic level sets $\mathcal{A}'_\omega=(a_\omega'+W)\cap
Q_1^{-1}(b_1^{(\omega)})\cap\dots\cap
Q_{s-1}^{-1}(b_{s-1}^{(\omega)})$ for each $\omega\in\{0,1\}^3$ such that
    \begin{align*}
      \mathbf{E}_{x,h_1,h_2,h_3\in G}\prod_{\omega\in\{0,1\}^3}&1_{\mathcal{A}'_\omega}(x+\omega\cdot(h_1,h_2,h_3))1_{\partial_{h_1,h_2,h_3}\psi(x)=0}\\
                                                               &\geq(1-\gamma)\mathbf{E}_{x,h_1,h_2,h_3\in G}\prod_{\omega\in\{0,1\}^3}1_{\mathcal{A}'_\omega}(x+\omega\cdot(h_1,h_2,h_3))>0.
    \end{align*}
  \end{lemma}
  \begin{proof}
    Define $Q:G\to\mathbf{F}_p$ by
    $Q:=\lambda_1Q_1+\dots+\lambda_{s-1}Q_{s-1}-Q_s$, so that $Q|_V$
    has rank less than $R$ on $V$. By the discussion at the
    beginning of Section~\ref{sec:bogo}, for each $\omega\in\{0,1\}^3$
    there exists a subspace $V'_{a_\omega}$ of codimension at most
    $R+1$ such that $Q$ is constant on the cosets of $V'_{a_\omega}$
    contained in $a_\omega+V$. Set
    \begin{equation*}
      W:=\bigcap_{\omega\in\{0,1\}^3}V'_{a_\omega},
    \end{equation*}
    so that $W$ has codimension at most $8(R+1)$ in $V$. Since $Q_s(x)$
    is completely determined by the values of
    $Q_1(x),\dots,Q_{s-1}(x)$, and $Q(x)$, it follows that each
    quadratic level set $\mathcal{A}_\omega$ can be partitioned into
    quadratic level sets $\mathcal{A}'_\omega$ of the form
    $\mathcal{A}'(a',b_1,\dots,b_{s-1}):=(a'+W)\cap Q_1^{-1}(b_1)\cap\dots\cap Q_{s-1}^{-1}(b_{s-1})$ with
    $a'\in a_{\omega}+V$ and $b_1,\dots,b_{s-1}\in\mathbf{F}_p$. Thus, we can write
    \begin{equation*}
      1_{\mathcal{A}_\omega}(x)=\sum_{a+W\subset a_\omega+V}\epsilon_\omega(a)1_{\mathcal{A}'(a,b_1^{(\omega)},\dots,b_{s-1}^{(\omega)})}
    \end{equation*}
    for each $\omega\in\{0,1\}^3$, where $\epsilon(a)_\omega\in\{0,1\}$ for each $a\in G$. So,
    \begin{align*}
      &\sum_{\substack{ a'_\omega+W\subset a_\omega+V \\ \omega\in\{0,1\}^3}}\mathbf{E}_{x,h_1,h_2,h_3\in G}\prod_{\omega\in\{0,1\}^3}\epsilon(a_\omega')1_{\mathcal{A}'(a_\omega',b_1^{(\omega)},\dots,b_{s-1}^{(\omega)})}(x+\omega\cdot(h_1,h_2,h_3))1_{\partial_{h_1,h_2,h_3}\psi(x)=0}\\
      &\geq(1-\gamma)\sum_{\substack{ a'_\omega+W\subset a_\omega+V \\ \omega\in\{0,1\}^3}}\mathbf{E}_{x,h_1,h_2,h_3\in G}\prod_{\omega\in\{0,1\}^3}\epsilon(a_\omega')1_{\mathcal{A}'(a_\omega',b_1^{(\omega)},\dots,b_{s-1}^{(\omega)})}(x+\omega\cdot(h_1,h_2,h_3))>0
    \end{align*}
    and the conclusion of the lemma now follows from the pigeonhole principle.
  \end{proof}
  
Finally, combining the main results of the previous three sections
with the above lemmas and the pigeonhole principle produces our desired
$(1-\varepsilon)$-approximate quadratic.

  \begin{lemma}\label{lem:psiapprox}
    Fix $p\geq 5$. There exists a constant $C\geq 1$ such that the
    following holds. Let $\phi:G\to G$ be $\delta$-approximately
    quadratic. Assume that
    $n\geq 2^{100}\lceil\log_p(2/\delta^{10000})\rceil$ and
    $R\geq C\delta^{-C}$. Then, there exists a real number
    $\theta=\Omega((\delta/p)^{O(1)})$, a subspace $V\leq G$ with
    $\codim{V}\ll\delta^{-O(1)}R$, quadratic functions
    $Q_1,\dots,Q_s:G\to\mathbf{F}_p$ with $s\ll\delta^{-O(1)}$ and
    $\lambda_1Q_1|_V+\dots+\lambda_sQ_s|_V$ having rank at least $R$
    whenever all $\lambda_1,\dots,\lambda_s\in\mathbf{F}_p$ are not
    all zero, and a function $\psi: G\to G$ that is
    $(1-2^{-40}p^{-20}\delta^{100})$-approximately quadratic on some
    quadratic level set
    $(a+V)\cap Q_1^{-1}(b_1)\cap\dots\cap Q_s^{-1}(b_s)$ and respects
    at most $p^{7n/2}$ $3$-dimensional cubes not wholly contained in
    \begin{equation}\label{eq:Sphi}
      \Big\{x\in G:\big|\big\{(h_1,h_2,h_3)\in G^3:\sum_{\mathbf{0}\neq\omega\in\{0,1\}^3}(-1)^{|\omega|}\phi(x+\omega\cdot(h_1,h_2,h_3))=\psi(x)\big\}\big|\geq\theta|G|^3\Big\}.
    \end{equation}
  \end{lemma}
  \begin{proof}
    By applying Lemma~\ref{lem:DRC} with
    $\gamma=2^{-60}p^{-20}\delta^{100}$, we get that there exists
    $A\subset G$ of density $\Omega((\delta/p)^{O(1)})$ containing at
    least $2^{-3500}p^{-1100}\delta^{7000}p^{28n}$ second-order
    $3$-dimensional cubes such that $\phi|_A$ respects at least a
    $(1-\gamma)$-proportion of them. Setting
    $\theta=2^{-8000}p^{-2400}\delta^{10000}$ and applying
    Lemma~\ref{lem:rselect}, we get that
    \begin{equation*}
      \mathbf{E}_{\psi}\mathbf{E}_{x,h_1,h_2,h_3\in G}\Delta_{h_1,h_2,h_3}\mathcal{D}[1_A](x)1_{\partial_{h_1,h_2,h_3}\psi(x)=0}\geq(1-2\gamma)\|\mathcal{D}[1_A]\|_{U^3(G)}^8
    \end{equation*}
    and that, with probability at least $1-2^{-10000}p^{-3000}\delta^{100000}$, say,
    $\psi$ respects at most $p^{7n/2}$ $3$-dimensional cubes not
    wholly contained in $S_{\phi,A,\psi}(\theta)$ (which, note, is a subset of~\eqref{eq:Sphi}). Thus, there exists a $\psi:G\to G$
    respecting at most $p^{7n/2}$ $3$-dimensional cubes not wholly
    contained in~\eqref{eq:Sphi} and for which
    \begin{equation*}
      \mathbf{E}_{x,h_1,h_2,h_3\in G}\Delta_{h_1,h_2,h_3}\mathcal{D}[1_A](x)1_{\partial_{h_1,h_2,h_3}\psi(x)=0}\geq(1-4\gamma)\|\mathcal{D}[1_A]\|_{U^3(G)}^8.
    \end{equation*}
    We now apply Lemma~\ref{lem:bogo} with
    $\eta=2^{-4000}p^{-1200}\delta^{10000}$ to obtain a quadratic
    factor $\mathcal{B}$ of $G$ of complexity $(r,s)$ at most
    $(O(\delta^{-O(1)}),O(\delta^{-O(1)}))$ and a
    $\mathcal{B}$-measurable function $F:G\to[0,1]$ such that
    $\|\mathcal{D}[1_A]-F\|_{L^\infty}\leq\eta$. Thus, as a
    consequence of Lemma~\ref{lem:dapproxinf}, we have that
    \begin{equation}\label{eq:theabove}
      \mathbf{E}_{x,h_1,h_2,h_3\in G}\Delta_{h_1,h_2,h_3}F(x)1_{\partial_{h_1,h_2,h_3}\psi(x)=0}\geq(1-16\gamma)\|F\|_{U^3(G)}^8.
    \end{equation}
    Further, since $A$ contains at least
    $2^{-3500}p^{-1100}\delta^{7000}p^{28n}$ second-order
    $3$-dimensional cubes, we certainly have that
    $\|F\|_{U^3}^8\geq
    2^{-3500}p^{-1100}\delta^{7000}-8\eta>0$.

    As $F$ is constant on atoms of $\mathcal{B}$, we can write
    \begin{equation*}
      F(x)=\sum_{\mathcal{A}\text{ atom of }\mathcal{B}}c(\mathcal{A})1_{\mathcal{A}}(x)
    \end{equation*}
    for some $c(\mathcal{A})\in[0,1]$. Plugging this
    into~\eqref{eq:theabove} yields
    \begin{align*}
      \sideset{}{'}\sum_{(\mathcal{A}_\omega)_{\omega}}&\Big(\prod_{\omega\in\{0,1\}^3}c(\mathcal{A}_\omega)\Big)\mathbf{E}_{x,h_1,h_2,h_3\in G}\prod_{\omega\in\{0,1\}^3}1_{\mathcal{A}_\omega}(x+\omega\cdot(h_1,h_2,h_3))1_{\partial_{h_1,h_2,h_3}\psi(x)=0} \\
                                                                                                        &\geq (1-16\gamma)\sideset{}{'}\sum_{(\mathcal{A}_\omega)_{\omega}}\Big(\prod_{\omega\in\{0,1\}^3}c(\mathcal{A}_\omega)\Big)\mathbf{E}_{x,h_1,h_2,h_3\in G}\prod_{\omega\in\{0,1\}^3}1_{\mathcal{A}_\omega}(x+\omega\cdot(h_1,h_2,h_3))
    \end{align*}
    where $'$ restricts summation to octuples of atoms
    $\mathcal{A}_\omega=\mathcal{A}(a_\omega,\mathbf{b}^{(\omega)})$ of
    $\mathcal{B}$ with the $a_\omega$'s
    satisfying~\eqref{eq:acondition} and the $\mathbf{b}^{(\omega)}$'s
    satisfying~\eqref{eq:bcondition}. It follows from the pigeonhole
    principle and the fact that $\|F\|_{U^3}>0$ that there exists an
    octuple of atoms for which
    \begin{align*}
      \mathbf{E}_{x,h_1,h_2,h_3\in G}\prod_{\omega\in\{0,1\}^3}&1_{\mathcal{A}_\omega}(x+\omega\cdot(h_1,h_2,h_3))1_{\partial_{h_1,h_2,h_3}\psi(x)=0}\\
      &\geq(1-16\gamma)\mathbf{E}_{x,h_1,h_2,h_3\in G}\prod_{\omega\in\{0,1\}^3}1_{\mathcal{A}_\omega}(x+\omega\cdot(h_1,h_2,h_3)).
    \end{align*}

    By repeated applications of Lemma~\ref{lem:rankrefine}, there
    exists a subspace $W\leq G$ of codimension $r'\ll\delta^{-O(1)}R$
    and some $s'\leq s\ll\delta^{-O(1)}$ such that
    $\lambda_1Q_1|_W+\dots+\lambda_{s'}Q_{s'}|_W$ has rank at least
    $R$ on $W$ whenever $\lambda_1,\dots,\lambda_{s'}\in\mathbf{F}_p$
    are not all zero and quadratic level sets
    $\mathcal{A}'_\omega=(a_\omega'+W)\cap
    Q_1^{-1}(b_1^{(\omega)})\cap\dots\cap
    Q_{s'}^{-1}(b_{s'}^{(\omega)})$ for each $\omega\in\{0,1\}^3$, we
    have
\begin{align*}
  \mathbf{E}_{x,h_1,h_2,h_3\in G}\prod_{\omega\in\{0,1\}^3}&1_{\mathcal{A}'_\omega}(x+\omega\cdot(h_1,h_2,h_3))1_{\partial_{h_1,h_2,h_3}\psi(x)=0}\\
                                                           &\geq(1-16\gamma)\mathbf{E}_{x,h_1,h_2,h_3\in G}\prod_{\omega\in\{0,1\}^3}1_{\mathcal{A}'_\omega}(x+\omega\cdot(h_1,h_2,h_3))>0.
\end{align*}
For the right-hand side above to be nonzero, the $a'_\omega$'s and
$b_{i}^{(\omega)}$'s must satisfy~\eqref{eq:acondition}
and~\eqref{eq:bcondition}. It follows from orthogonality of characters
and the Gowers--Cauchy--Schwarz inequality in the form of Theorem~\ref{thm:GCSnorm} that
  \begin{equation*}
    (1-32\gamma)p^{-4r'-7s'}\leq \mathbf{E}_{\xi\in G}\prod_{\omega\in\{0,1\}^3}\|f_{\omega,\xi}\|_{U^3(G)}\leq\prod_{\mathbf{0}\neq\omega\in\{0,1\}^3}\|1_{\mathcal{A}_\omega}\|_{U^3(G)} \cdot\mathbf{E}_{\xi\in G}\|f_{\mathbf{0},\xi}\|_{U^3(G)},
  \end{equation*}
  where $f_{\omega,\xi}(x)=e_p(\xi\cdot\psi(x))1_{\mathcal{A}_\omega'}(x)$. Thus, by Lemma~\ref{lem:count} and H\"older's inequality,
  \begin{equation*}
    \mathbf{E}_{\xi\in G}\|f_{\mathbf{0},\xi}\|_{U^3(G)}^8\geq (1-2^{15}\gamma)p^{-4r'-7s'}.
  \end{equation*}
  Expanding the definition of $f_{\mathbf{0},\xi}$ and using
  orthogonality of characters again yields that $\psi$ is
  $(1-2^{16}\gamma)$-approximately quadratic on
  $\mathcal{A}_{\mathbf{0}}'$, completing the proof.
\end{proof}
  
  \section{A relative inverse theorem for $99\%$ Freiman homomorphisms}\label{sec:psdhom}

  We could now quickly finish the proof of
  Theorem~\ref{thm:approxpoly} by invoking results of Kazhdan and
  Ziegler~\cite{KazhdanZiegler2017,KazhdanZiegler2019,KazhdanZiegler2020}
  implying that a $(1-\varepsilon)$-approximate quadratic on a
  high-rank quadratic level set must agree almost everywhere with a
  genuine quadratic map. However, we have a new proof of this fact,
  which we will present in the next section. For the sake of
  completeness, in this section we will give a short proof of one of
  the ingredients: the inverse theorem for
  $(1-\varepsilon)$-approximately linear functions on $U^2$-uniform
  sets, which is a special case (with bounds that are slightly worse
  in an inconsequential way) of another result of Kazhdan and
  Ziegler~\cite[Theorem 1.7]{KazhdanZiegler2019}. We begin by
  recalling the definition of uniform sets:

  \begin{definition}
    For any finite abelian group $G_1$, a subset $A\subset G_1$ of
    density $\alpha$ is \textit{ $\delta$-uniform} if
    $\|1_A-\alpha\|_{U^2(G_1)}\leq\delta$.
  \end{definition}

  The main result of this section is a relative version of (a
  corollary of) the famous linearity test of Blum, Luby, and
  Rubinfeld~\cite{BlumLubyRubinfeld1993}, which we also recall.
  
  \begin{theorem}\label{thm:BLR}
    Let $G_1$ and $G_2$ be finite abelian groups and
    $0<\varepsilon<\frac{1}{6}$, and suppose that $\phi:G_1\to G_2$
    satisfies $\phi(x+y)=\phi(x)+\phi(y)$ for at least a
    $(1-\varepsilon)$-proportion of pairs $(x,y)\in G_1^2$. Then,
    there exists a homomorphism $\Phi:G_1\to G_2$ such that
    $\phi(x)=\Phi(x)$ for at least a $(1-2\varepsilon)$-proportion of
    $x\in G_1$.
  \end{theorem}

  It is well-known that one can obtain a version of this theorem for
  approximately linear functions by making minor modifications to the
  argument in~\cite{BlumLubyRubinfeld1993}. One can also derive it
  using the cocycle identity; we include the short proof below for
  completeness.
    \begin{corollary}\label{cor:99afflin}
      Let $G_1$ and $G_2$ be finite abelian groups and
      $0<\varepsilon<2^{-10}$, and suppose that $\psi:G_1\to G_2$ is
      $(1-\varepsilon)$-approximately linear. Then, there exists a
      homomorphism $\Psi:G_1\to G_2$ and an element $c\in G_2$ such that $\psi(x)=\Psi(x)+c$ for
      at least a $(1-7\sqrt{\varepsilon})$-proportion of $x\in G_1$.
    \end{corollary}
    \begin{proof}
      By hypothesis, $\partial_h\psi(x)=\partial_h\psi(y)$ for at least a
      $(1-\sqrt{\varepsilon})$-proportion of $(x,y)\in G_1^2$ for at
      least a $(1-\sqrt{\varepsilon})$-proportion of $h\in G_1$. For
      each such $h$, let $\phi(h)$ be the most popular value of
      $\partial_{h}\psi(x)$, and pick $\phi(h)$ arbitrarily for all
      other $h\in G_1$. By the cocycle identity
      $\partial_{h+k}\phi(x)=\partial_h\phi(x)+\partial_k\phi(x+h)$
      and the union bound, it follows that $\phi(h+k)=\phi(h)+\phi(k)$
      for at least a $(1-3\sqrt{\varepsilon})$-proportion of pairs
      $(h,k)\in G_1^2$. Theorem~\ref{thm:BLR} can now be applied to
      $\phi$, and yields a homomorphism $\Phi:G_1\to G_2$ such that
      $\phi(h)=\Phi(h)$ for at least a
      $(1-6\sqrt{\varepsilon})$-proportion of $h\in G_1$. It follows
      that $\psi(x)-\psi(x+h)=\Phi(h)$ for at least a
      $(1-7\sqrt{\varepsilon})$-proportion of pairs $(x,h)\in
      G_1^2$. Making the change of variables $h\mapsto h-x$ and
      applying the pigeonhole principle to fix $x$ yields
      $\psi(h)=-\Phi(h)+\Phi(x)+\psi(x)$ for at least a
      $(1-7\sqrt{\varepsilon})$-proportion of $h\in
      G_1$, completing the proof.
    \end{proof}

    We will show that the above corollary still holds when $\psi$ is
    $(1-\varepsilon)$-approximately linear on a sufficiently uniform
    subset of $G_1$. The proof uses some of the same ideas as the proof of
    Theorem~\ref{thm:approxpoly}, but is much simpler.

    For $f:G_1\to \mathbf{C}$, define its \textit{$U^2$-dual function}
    by
    \begin{equation*}
      \mathcal{D}_2[f](x):=\mathbf{E}_{h_1,h_2\in G_1}\overline{f(x+h_1)f(x+h_2)}f(x+h_1+h_2).
    \end{equation*}
    If $A\subset G_1$ has density $\alpha$ and is $\delta$-uniform for
    $\delta$ sufficiently small, then $\mathcal{D}_2[1_A]$ is very
    close to the constant function $\alpha^3$ on $G_1$. Indeed, for all $x\in G_1$, we have
    \begin{equation*}
      \mathcal{D}_2[1_A](x)=\alpha^3+\mathbf{E}_{h_1,h_2\in G_1}1_A(x+h_1)1_A(x+h_2)(1_A-\alpha)(x+h_1+h_2)
    \end{equation*}
    It then follows from the Gowers--Cauchy--Schwarz inequality in the
    form of Theorem~\ref{thm:GCSbox} that
    $\|\mathcal{D}_2[1_A]-\alpha^3\|_{L^\infty}\leq\|1_A-\alpha\|_{U^2(G_1)}\leq\delta$, which
    says that all elements $x\in G_1$ can be expressed as
    $x=a_1+a_2-a_3$ with $a_1,a_2,a_3\in A$ in very close to $\alpha^3|G_1|^2$ different
    ways. If $\phi$ is $(1-\varepsilon)$-approximately linear, and
    thus respects almost all additive quadruples in $A$, then $\phi$
    also respects almost all additive sextuples
    $a_1+a_2+a_3=a_1'+a_2'+a_3'$ in $A$ by an application of the
    Cauchy--Schwarz inequality combined with the $\delta$-uniformity
    of $A$. A further application of the Gowers--Cauchy--Schwarz
    inequality combined with the $\delta$-uniformity of $A$ similarly
    shows that $\phi$ respects almost all second-order $2$-dimensional
    cubes in $A$. If we then define $\psi:G_1\to G_2$ by setting
    $\psi(x)$ to equal $\phi(a_1)+\phi(a_2)-\phi(a_3)$ for
    $a_1+a_2-a_3=x$ chosen uniformly at random, it follows that $\psi$
    typically agrees with $\phi$ on almost all of $A$ and respects
    almost all $2$-dimensional cubes in $G_1$. Applying
    Corollary~\ref{cor:99afflin} to $\psi$ and making a change of
    variables then proves that $\phi$ almost always agrees with an
    affine-linear function on $A$.

    We will state and prove the following lemma in the setting of
    finite dimensional vector spaces over finite fields, since that is
    the case we will need, but of course the argument sketched above does not require
    this.
  \begin{lemma}\label{lem:inv}
    Let $G_1$ and $G_2$ be finite-dimensional vector spaces over
    $\mathbf{F}_p$, $0<\varepsilon<2^{-15}$, $A\subset G_1$ have
    density $\alpha>0$, and $\phi:A\to G_2$ be
    $(1-\varepsilon)$-approximately linear on $A$. If $A$ is $\delta$-uniform
    with $\delta<(\alpha\varepsilon)^{32}$, then there exists an
    affine-linear $\Phi:G_1\to G_2$ such that $\phi(x)=\Phi(x)$ for at
    least a $(1-16\sqrt[4]{\varepsilon})$-proportion of $x\in A$.
  \end{lemma}
  \begin{proof}
    Using orthogonality of characters, $\phi$ being
    $(1-\varepsilon)$-approximately linear on $A$ means
    \begin{equation*}
      \mathbf{E}_{\xi\in G_2}\mathbf{E}_{x,h_1,h_2\in G_1}\Delta_{h_1,h_2}f_\xi(x)\geq(1-\varepsilon)\|1_A\|_{U^2(G_1)}^4,
    \end{equation*}
    where $f_\xi(x):=1_A(x)e_p(\xi\cdot \phi(x))$. By the Cauchy--Schwarz inequality,
    \begin{equation*}
      \mathbf{E}_{\xi\in G_2}\mathbf{E}_{x\in G_1}1_A(x)\big|\mathbf{E}_{h_1,h_2\in G_1}\overline{f_\xi(x+h_1)f_\xi(x+h_2)}f_\xi(x+h_1+h_2)\big|^2\geq \alpha^{-1}(1-\varepsilon)^2\|1_A\|_{U^2(G_1)}^8.
    \end{equation*}
    Now, any average of the form
    \begin{equation*}
      \mathbf{E}_{x,h_1,h_2,k_1,k_2\in G_1}g_0(x)g_1(x+h_1)g_2(x+h_2)g_3(x+h_1+h_2)g_4(x+k_1)g_5(x+k_2)g_6(x+k_1+k_2)
    \end{equation*}
    for $1$-bounded $g_i$ is bounded above by $\|g_0\|_{U^2(G_1)}$, as
    can be seen by making the change of variables $h_1\mapsto h_1-x$,
    $h_2\mapsto h_2-x$, $k_1\mapsto k_1-x$, $k_2\mapsto k_2-x$, and
    $x\mapsto x+h_1+k_1$ and then applying the Gowers--Cauchy--Schwarz
    inequality in the form of Theorem~\ref{thm:GCSbox}. Thus, by writing $1_A=\alpha+1_A-\alpha$, we obtain using the
    $\delta$-uniformity of $A$ that
    \begin{equation*}
      \mathbf{E}_{\xi\in G_2}\mathbf{E}_{x,h_1,h_2\in G_1}\mathcal{D}_2[f_{-\xi}](x)\overline{f_\xi(x+h_1)f_\xi(x+h_2)}f_\xi(x+h_1+h_2) \geq \alpha^{-2}(1-\varepsilon)^2\|1_A\|_{U^2(G_1)}^8-\alpha^{-1}\delta.
    \end{equation*}
    Since $|\|1_A\|_{U^2(G_1)}-\alpha|\leq \delta$ by the reverse
    triangle inequality and the assumption that $A$ is
    $\delta$-uniform, it follows that as long as
    $\delta<\varepsilon\alpha^{16}$, say, the right-hand side above is
    greater than
    $\alpha^{-2}(1-3\varepsilon)\|1_A\|_{U^2(G_1)}^8$. There are two
    consequences of this (again provided that $\delta<\varepsilon\alpha^{16}$): $\phi$ respects at least a
    $(1-4\varepsilon)$-proportion of the additive sextuples in $A$
    (as, by the bound
    $\|\mathcal{D}_2[1_A]-\alpha^3\|_{L^\infty}\leq
    \delta$, the number of additive sextuples in $A$
    differs from $\alpha^6|G_1|^5$ by a quantity of size at most
    $\delta|G_1|^5$), and, by the
    Gowers--Cauchy--Schwarz inequality in the form of
    Theorem~\ref{thm:GCSnorm}, we have that
    \begin{equation*}
      \mathbf{E}_{\xi\in G_2}\|\mathcal{D}_2[f_\xi]\|_{U^2(G_1)}\geq \alpha^{-2}(1-3\varepsilon)\|1_A\|_{U^2(G_1)}^5\geq(1-4\varepsilon)\alpha^{3}.
    \end{equation*}
    Using that
    $\|\mathcal{D}_2[1_A]-\alpha^3\|_{L^\infty}\leq\delta$ again, we get that the count of second-order
    $2$-dimensional cubes in $A$ differs from $\alpha^{12}|G_1|^{11}$ by
    a quantity of size at most $\delta^{2}|G_1|^{11}$. It
    now follows from H\"older's inequality that $\phi$ respects at
    least a $(1-17\varepsilon)$-proportion of second-order
    $2$-dimensional cubes in $A$.

    For each pair $(x,b)\in G_1\times G_2$, set
    \[
      \nu_x'(b):=\frac{|\{(a_1,a_2,a_3)\in A: a_1+a_2-a_3=x\text{ and
        }\phi(a_1)+\phi(a_2)-\phi(a_3)=b\}|}{|G_1|^2\mathcal{D}_2[1_A](x)}
    \]
    (noting that the upper bound on $\delta$ ensures that
    $\mathcal{D}_2[1_A](x)>0$), and define $\psi:G_1\to G_2$ randomly by
    setting $\psi(x)=b$ with probability $\nu'_x(b)$ for each $x\in G_1$
    independently. First, the expected proportion of $x\in G_1$ for which
    $\psi(x)$ is a ``popular'' value of
    $\phi(a_1)+\phi(a_2)-\phi(a_3)$ is large. Indeed, by linearity of
    expectation and the uniformity of $A$, we have
    \begin{align*}
      \mathbf{E}_{\psi}\sum_{x\in G_1}\nu'_x(\psi(x))=\sum_{x\in G_1}\mathbf{E}_{\psi}\nu'_x(\psi(x))=\sum_{x\in G_1}\sum_{b\in G_2}\nu'_{x}(b)^2\geq (1-8\varepsilon)|G_1|
    \end{align*}
    since $\phi$ respects at least a $(1-4\varepsilon)$-proportion of
    additive sextuples in $A$. Similarly, the expected proportion of
    additive quadruples in $G_1$ respected by $\psi$ is at least
    \begin{equation}\label{eq:U2nu'}
      \mathbf{E}_{x,h_1,h_2\in G_1}\sum_{\substack{a,b,c,d\in G_2 \\ a+b=c+d}}\nu'_{x}(a)\nu'_{x+h_1}(b)\nu'_{x+h_2}(c)\nu'_{x+h_1+h_2}(d)-3|G_1|^{-3},
    \end{equation}
    where the second term comes from the degenerate $2$-dimensional cubes. Expanding the definition
    of $\nu'_b$ and using the uniformity of $A$ yields that the first term in~\eqref{eq:U2nu'} is at least
    \begin{equation*}
      \frac{(1-\delta)^4}{\alpha^{12}}\cdot\frac{|\{\text{second-order }2\text{-dimensional cubes in }A\text{ respected by }\phi\}|}{|G_1|^{11}}\geq 1-5\varepsilon.
    \end{equation*}
    
    If $|G_1|\geq\varepsilon^{-1}$, then~\eqref{eq:U2nu'} is at least
    $1-8\varepsilon$. By Markov's inequality, there must exist a
    choice of $\psi$ that is $(1-16\varepsilon)$-approximately linear
    and for which
    $\mathbf{E}_{x\in G_1}\nu'_x(\psi(x))\geq 1-16\varepsilon$. We can
    now apply Corollary~\ref{cor:99afflin} to $\psi$ to obtain an
    affine-linear function $\Psi:G_1\to G_2$ such that
    $\psi(x)=\Psi(x)$ for at least a
    $(1-32\sqrt{\varepsilon})$-proportion of $x\in G_1$. Since
    $\mathbf{E}_{x\in G_1}\nu'_x(\psi(x))\geq 1-16\varepsilon$, we
    have that $\nu'_x(\psi(x))\geq 1-4\sqrt[4]{\varepsilon}$ for at
    least a $(1-4\sqrt[4]{\varepsilon})$-proportion of the $x\in
    G_1$. Thus, for at least a $(1-8\sqrt[4]{\varepsilon})$-proportion
    of $x\in G_1$, we have
    $\nu'_x(\psi(x))\geq 1-4\sqrt[4]{\varepsilon}$ and
    $\psi(x)=\Psi(x)$. This means that there exist at least
    $(1-16\sqrt[4]{\varepsilon})|A|^3$ choices of
    $(a_1,a_2,a_3)\in A^3$ such that
    $\phi(a_1)+\phi(a_2)-\phi(a_3)=\Psi(a_1+a_2-a_3)$. By the
    pigeonhole principle, it follows that there exist $a_2,a_3\in A$
    such that $\phi(a)=\Psi(a)+\Psi(a_2-a_3)+\phi(a_3)-\phi(a_2)$ for
    at least a $(1-16\sqrt[4]{\varepsilon})$-proportion of $a\in A$,
    as desired.

    When $|G_1|<\varepsilon^{-1}$, the only nonempty subset of
    $G_1$ that is $\delta$-uniform for any $\delta<\varepsilon^{4}$ is
    $G_1$ itself (as can easily be seen using Parseval's identity),
    and Corollary~\ref{cor:99afflin} yields the result.
  \end{proof}

  In the next section, we will need to apply Lemma~\ref{lem:inv} on
  affine subspaces of $G$. To that end, we will define the notion of
  uniformity relative to an affine subspace.
  
  \begin{definition}
    Let $V\leq G$ be a subspace and $A\subset V$ have density
    $\alpha$. We say that $A$ is \textit{$\delta$-uniform relative to
      $V$} if $\|1_A-\alpha\|_{U^2(V)}\leq\delta$. If $w+V$ is a coset
    of $V$ and $B\subset w+V$, we say that $B$ is
    \textit{$\delta$-uniform relative to $w+V$} if $B-w$ is
    $\delta$-uniform relative to $V$.
  \end{definition}

  We then obtain the following immediate consequence of
  Lemma~\ref{lem:inv}

  \begin{lemma}\label{lem:inv2}
    Let $0<\varepsilon<2^{-15}$, $W\subset G$ be an affine subspace,
    $A\subset W$ be a subset of density $\alpha>0$ that is
    $\delta$-uniform relative to $W$, and $\phi:A\to G$ be
    $(1-\varepsilon)$-approximately linear on $A$. If
    $\delta<(\alpha\varepsilon)^{32}$, then there exists an
    affine-linear $\Phi:W\to G$ such that $\phi(x)=\Phi(x)$ for at
    least a $(1-16\sqrt[4]{\varepsilon})$-proportion of $x\in A$.
  \end{lemma}
  
  \section{Classifying $99\%$ approximate quadratics on quadratic level sets}\label{sec:coho}
  
  In this section we will, as promised, prove that a
  $(1-\varepsilon)$-approximate quadratic on a high-rank quadratic
  level set almost always agrees with the restriction of a quadratic
  map.
  
  \begin{lemma}\label{lem:quadext}
    Let $p\geq 3$, $R\in\mathbf{N}$, and
    $Q_1,\dots,Q_s:\mathbf{F}_p^d\to\mathbf{F}_p$ be quadratic
    functions satisfying
    \begin{equation}\label{eq:rank}
      \min_{(0,\dots,0)\neq (\lambda_1,\dots,\lambda_s)\in\mathbf{F}_p^s}\rank(\lambda_1Q_1+\dots+\lambda_sQ_s)\geq R,
    \end{equation}
    and set $Z:=Q_1^{-1}(0)\cap\dots\cap Q_s^{-1}(0)$ when $s\geq 1$
    and $Z=\mathbf{F}_p^d$ when $s=0$. Assume that
    $0<\varepsilon<2^{-160}p^{-16}$ and
    $R\geq 2^{20}(s+\log_p{\varepsilon^{-1}})$. If $\psi:Z\to G$ is
    $(1-\varepsilon)$-approximately quadratic, then there exists a
    quadratic map $Q:\mathbf{F}_p^d\to G$ with $\psi(x)=Q(x)$ for at
    least a $(1-2^{10}\sqrt[16]{\varepsilon})$-proportion of $x\in Z$.
  \end{lemma}
  
  Our proof of Lemma~\ref{lem:quadext} is more cohomological in flavor
  than those of Kazhdan and Ziegler, and proceeds as follows. Assume
  for simplicity that each $Q_i$ has the form $Q_i(x)=x^TM_ix$ for
  some symmetric $M_i\in\M_d(\mathbf{F}_p)$. By the high-rank
  condition, $\psi:Z\to G$ being $(1-\varepsilon)$-approximately
  quadratic implies that the derivatives $\partial_{h}\psi$ are
  $(1-\sqrt{\varepsilon})$-approximately linear on a highly uniform
  subset of a coset $W_h$ of the subspace
  $W_h^0:=\spn(M_1h,\dots,M_sh)^{\perp}$ for almost all $h$. Applying
  Lemma~\ref{lem:inv2} yields, for each such $h$, a linear map
  $L[h]$ from $W_h$ to $G$ and a function $f:\mathbf{F}_p^d\to G$ such
  that $L[h]x+f(h)=\partial_h\psi(x)$ for almost every
  $x\in Z\cap W_h$.

  By the cocycle equation for $\partial_h\psi$ and
  the high-rank assumption, it follows that, for almost all
  $h,k\in \mathbf{F}_p^d$, we must have $(L[h+k]-L[h]-L[k])x=0$ for
  all $x\in W^0_h\cap W^0_k$, which forces the rows of the matrix
  $d^1L[h,k]:=L[h+k]-L[h]-L[k]$ to lie in the span of
  $h^TM_1,k^TM_1,\dots,h^TM_s,k^TM_s$. For such $h$ and $k$, this
  means that there exist $n\times s$ matrices $D(h,k)$ and $E(h,k)$
  with entries in $\mathbf{F}_p$ such that
  \begin{equation*}
    d^1L[h,k]=
    D(h,k)\begin{pmatrix}
      h^TM_1 \\
      \vdots \\
      h^TM_s
    \end{pmatrix}+E(h,k)\begin{pmatrix}
      k^TM_1 \\
      \vdots \\
      k^TM_s
    \end{pmatrix}
  \end{equation*}
  Now, the key idea is that, since $d^1L[h,k]$ is a $2$-coboundary
  $\mathbf{F}_p^d\times\mathbf{F}_p^d\to\mathbf{F}_p^{n\times d}$, it is a $2$-cocycle,
  i.e., it satisfies the equation
  \begin{equation}
    \label{eq:2cocycleeqn}
    d^1L[k,l]-d^1L[h+k,l]+d^1L[h,k+l]-d^1L[h,k]=0
  \end{equation}
  for all $h,k,l\in\mathbf{F}_p^d$. From this, one can deduce that
  $D(h,k)$ and $E(h,k)$ almost always satisfy a certain system of
  three equations (\eqref{eq:A}, \eqref{eq:AB}, and \eqref{eq:B}
  below). Equations~\eqref{eq:A} and~\eqref{eq:B} force $D(h,k)$ and
  $E(h,k)$ to almost always agree with $1$-coboundaries
  $\mathbf{F}_p^d\to(\mathbf{F}_p^d\to\mathbf{F}_p^{n\times s})$ for
  the nontrivial $\mathbf{F}_p^d$-action $m\cdot F(x)=F(x+m)$, i.e.,
  there exist $D',E':\mathbf{F}_p^d\to\mathbf{F}_p^{n\times s}$ such
  that $D(h,k)$ and $E(h,k)$ almost always agree with $D'(h+k)-D'(h)$
  and $E'(h+k)-E'(k)$, respectively. Equation~\eqref{eq:AB} further implies
  that $D'-E'$ must be $(1-O(\varepsilon^{\Omega(1)}))$-approximately
  linear, and then Lemma~\ref{lem:quadext} follows from
  Corollary~\ref{cor:99afflin} and standard manipulations.
  
  To execute this proof, we will need a few preliminaries. We begin
  with three very simple lemmas from linear algebra.
  \begin{lemma}\label{lem:span}
    Let $w_1,\dots,w_k\in \mathbf{F}_p^d$ and set
    $V=\spn(w_1,\dots,w_k)^{\perp}$. If $M\in\mathbf{F}_p^{n\times d}$ is an $n\times d$ matrix such that $Mv=0$ for all $v\in V$, then
    the rows of $M$ all lie in $\spn(w_1^T,\dots,w_k^T)$.
  \end{lemma}
  \begin{proof}
    Let $L$ be the matrix with rows $w_1^T,\dots,w_k^T$, so that
    $\ker{L}=V\leq \ker{M}$. Since $\Img{N^T}=(\ker{N})^{\perp}$ for
    any matrix $N$, we have
    $\Img{M^T}=(\ker{M})^{\perp}\leq(\ker{L})^{\perp}=\Img{L^T}$.
  \end{proof}
  \begin{lemma}\label{lem:sym}
    Let $V$ be a vector space over $\mathbf{F}_p$ and
    $B:\mathbf{F}_p^d\times\mathbf{F}_p^d\to V$ be a bilinear map such
    that $B(h,k)=B(k,h)$ for at least a $(1-\varepsilon)$-proportion
    of pairs $(h,k)\in\mathbf{F}_p^d\times\mathbf{F}_p^d$. If
    $\varepsilon<(\frac{p-1}{p})^2$, then $B(h,k)=B(k,h)$ for all
    $h,k\in\mathbf{F}_p^d$.
  \end{lemma}
  \begin{proof}
    By hypothesis,
    $|\ker{B(h,\cdot)-B(\cdot,h)}|\geq(1-\sqrt{\varepsilon})p^d>p^{d-1}$
    for at least a $(1-\sqrt{\varepsilon})$-proportion of
    $h\in \mathbf{F}^d_p$, and thus we must have $B(h,k)=B(k,h)$ for
    all $k\in\mathbf{F}_p^d$ for at least a
    $(1-\sqrt{\varepsilon})$-proportion of $h\in\mathbf{F}_p^d$. This
    similarly forces $\ker{B(\cdot,k)-B(k,\cdot)}=\mathbf{F}_p^d$ for
    all $k\in\mathbf{F}_p^d$, so that $B$ must be symmetric.
  \end{proof} 
  
  \begin{lemma}\label{lem:hkl}
    Let $M_1,\dots,M_s\in\M_d(\mathbf{F}_p)$ be symmetric matrices
    such that $\lambda_1M_1+\dots+\lambda_sM_s$ has rank at least $R$
    whenever $\lambda_1,\dots,\lambda_s\in\mathbf{F}_p$ are not all zero. The
    proportion of triples
    $(h,k,l)\in\mathbf{F}_p^d\times\mathbf{F}_p^d\times\mathbf{F}_p^d$
    for which
    \begin{equation}\label{eq:hkl1}
      M_1^Th,M_1^Tk,M_1^Tl,\dots,M_s^Th,M_s^Tk,M_s^Tl
    \end{equation}
    are linearly independent is at least $1-p^{3s-R}$.
  \end{lemma}
  \begin{proof}
    If the vectors~\eqref{eq:hkl1} are linearly dependent, then there are
    $\lambda_1,\eta_1,\xi_1,\dots,\lambda_s,\eta_s,\xi_s\in\mathbf{F}_p$
    not all zero for which
    \begin{equation*}
      \lambda_1M_1^Th+\eta_1M_1^Tk+\xi_1M_1^Tl+\dots+\lambda_sM_s^Th+\eta_sM_s^Tk+\xi_sM_s^Tl=0.
    \end{equation*}
    Without loss of generality, we may as well assume that some
    $\lambda_i$ is nonzero, so that for any fixed
    $k,l\in\mathbf{F}_p^d$ there are at most $p^{d-R}$ possible
    choices of $h\in\mathbf{F}_p^d$ for which the above equality
    holds. Thus, there are at most $p^{3s+2d+d-R}=p^{3d+3s-R}$ possible choices of
    triples $(h,k,l)$ for which the vectors~\eqref{eq:hkl1} are not linearly
    independent.
  \end{proof}

  We will also require various facts about high rank quadratic level
  sets, almost all of whose proofs are standard applications of the
  Gauss sum bounds in Lemma~\ref{lem:GaussSum}.
  
  \begin{lemma}\label{lem:varietyfacts}
    Let $p\geq 3$, $R\in\mathbf{N}$, and
    $Q_1,\dots,Q_s:\mathbf{F}_p^d\to\mathbf{F}_p$ be quadratic
    functions satisfying the rank assumption~\eqref{eq:rank}. Set $\mathcal{Q}:=(Q_1,\dots,Q_s)$, $Z:=\mathcal{Q}^{-1}(0)$, and, for all $h,k\in\mathbf{F}_p^d$,
    \begin{equation*}
      W_h:=\{x\in\mathbf{F}_p^d: \partial_hQ_i(x)=0\text{ for all }i=1,\dots,s\},  
    \end{equation*}
    $W_h^0:=W_h-W_h$, $Z_h:=Z\cap (Z-h)$,
    $W_{h,k}:=W_h\cap W_{h+k}\cap (W_k-h)$,
    $W_{h,k}^0:=W_h^0\cap W_{k}^0$, and $Z_{h,k}:= Z_h\cap Z_{h+k}\cap (Z_k-h)$. Denote the set of
    $h\in\mathbf{F}_p^d$ for which $\codim W_h^0=s$ by $X$ and the set
    of pairs $(h,k)\in\mathbf{F}_p^d\times\mathbf{F}_p^d$ for which
    $\codim W_{h,k}^0=2s$ by $Y$. Then,
    \begin{enumerate}
    \item $Z$ is $p^{-R/4}$-uniform and has density
      $p^{-s}+\theta p^{-R/2}$ for some $\theta\in[-1,1]$,
    \item for all $h\in X$, $Z_h$ is $p^{s/2-R/4}$-uniform relative to
      $W_h$ and the density of $Z_h$ in $W_h$ equals
      $p^{-s}+\theta(h)p^{s-R/2}$ for some $\theta(h)\in[-1,1]$,
    \item for all $h\in \mathbf{F}_p^d$ and $x\in Z_h$, the expected
      proportion of $k\in\mathbf{F}_p^d$ for which $x\in Z_{h,k}$ is at least
      $p^{-s}-p^{-R/2}$, for all $k\in\mathbf{F}_p^d$ and
      $x\in Z_{k}-h$, the expected proportion of $h\in\mathbf{F}_p^d$
      for which $x\in Z_{h,k}$ is at least $p^{-s}-p^{-R/2}$, and
      for all $l\in\mathbf{F}_p^d$ and $x\in Z_l$, the expected
      proportion of $(h,k)\in(\mathbf{F}_p^d)^2$ with $h+k=l$ for which
      $x\in Z_{h,k}$ is at least $p^{-s}-p^{-R/2}$,
    \item for all $(h,k)\in Y$, $Z_{h,k}$ is $p^{s-R/4}$-uniform
      relative to $W_{h,k}$ and the density of $Z_{h,k}$ in $W_{h,k}$
      equals $p^{-s}+\theta'(h,k)p^{2s-R/2}$ for some
      $\theta'(h,k)\in[-1,1]$,
    \item for all $(h,k)\in Y$, the intersection of $Z_{h,k}$ with any
      codimension $1$ affine subspace of $W_{h,k}$ has size at least $p^{d-3s-1}-2p^{d+2s+1-R/2}$, and
    \item we have $|\mathbf{F}_p^d\setminus X|\leq p^{d-(R-s)}$ and $|\mathbf{F}_p^d\times\mathbf{F}_p^d\setminus Y|\leq p^{2d-(R-2s)}$.
    \end{enumerate}
  \end{lemma}
  \begin{proof}
    For each $i=1,\dots,s$, write $Q_i(x)=x^TM_ix+z_i\cdot x+c_i$ for
    $M_i\in\M_d(\mathbf{F}_p)$ symmetric, $z_i\in\mathbf{F}_p^d$, and
    $c_i\in\mathbf{F}_p$, so that
    $W_h^0=\spn(M_1h,\dots,M_sh)^{\perp}$. By orthogonality of
    characters,
    \begin{equation*}
      \mathbf{E}_{x\in\mathbf{F}_p^d}1_Z(x)e_p(-\xi\cdot x)=\mathbf{E}_{\lambda_1,\dots,\lambda_s\in\mathbf{F}_p}\mathbf{E}_{x\in\mathbf{F}_p^d}e_p\Big(x^T\Big[\sum_{i=1}^s\lambda_iM_i\Big]x+\Big[\sum_{i=1}^s\lambda_iz_i\Big]\cdot x+\sum_{i=1}^s\lambda_ic_i-\xi\cdot x\Big).
    \end{equation*}
    The inner average has magnitude at most $p^{-R/2}$ whenever
    $(\lambda_1,\dots,\lambda_s)\neq (0,\dots,0)$ by
    Lemma~\ref{lem:GaussSum}. If $\lambda_1=\dots=\lambda_s=0$, then
    the inner average is zero unless $\xi=0$, in which case it
    equals $1$. Part (1) of the lemma
    now immediately follows from the fact that
    $\|f\|_{U^2(\mathbf{F}_p^d)}^4=\|\widehat{f}\|_{\ell^4}^4\leq\|\widehat{f}\|_{L^\infty}^2\|f\|_{L^2}^2\leq
    \|\widehat{f}\|_{L^\infty}^2$ for any $1$-bounded
    $f:\mathbf{F}_p^d\to\mathbf{C}$.
    
    The proofs of parts (2)--(5) of the lemma are similar. Observe
    that $Z_h=Z\cap W_h$. To show part (2) of the lemma, then, we want
    to estimate $\mathbf{E}_{w\in W_h}1_Z(w)e_p(-\xi\cdot w)$ for
    $\xi=0$ and $\xi\notin (W_h^0)^{\perp}$. By orthogonality of
    characters, this average equals $p^s$ times
    \begin{equation*}
      \mathbf{E}_{\substack{\lambda_i,\eta_i\in\mathbf{F}_p \\ i\in[s]}}\mathbf{E}_{w\in \mathbf{F}_p^s}e_p\Big(\sum_{i=1}^s\lambda_iQ_i(w)+\sum_{i=1}^s\eta_i(2h^TM_iw+h^TM_ih+z_i^Th)-\xi\cdot w\Big).
    \end{equation*}
    The inner average has magnitude at most $p^{-R/2}$ whenever
    $\lambda_1,\dots,\lambda_s$ are not all zero by
    Lemma~\ref{lem:GaussSum}. If $\lambda_1=\dots=\lambda_s=0$, then
    the inner average is zero when $\xi\notin(W_h^0)^\perp$ or
    $\xi=0$ and $\eta_1,\dots,\eta_s$ are not all zero, and equals $1$
    when $\xi=0$ and $\eta_1=\dots=\eta_s=0$. Part (2) now follows
    analogously to part (1).

    For parts (3)--(6), observe that
    \begin{equation*}
      2k^TM_i(x+h)+k^TM_ik=2(h+k)^TM_ix+(h+k)^TM_i(h+k)-(2h^TM_ix+h^TM_ih)
    \end{equation*}
    for all $x,h,k\in\mathbf{F}_p^n$ and $i=1,\dots,s$. Thus,
    $x\in W_{h}\cap W_{h+k}$ if and only if $x\in W_h\cap (W_k-h)$,
    and hence $W_{h,k}=W_h\cap (W_{k}-h)=W_h\cap W_{h+k}$. It follows
    that $W_{h,k}$ is a coset of the subspace
    $W_{h,k}^0$. Using that $Z_l= Z\cap W_l$ for all $l\in H$, it also
    follows that $Z_{h,k}=Z\cap W_{h,k}$.

    Thus, for any fixed $x\in Z_h$, the proportion of $k\in\mathbf{F}_p^d$ for which $x\in Z_{h,k}$ equals
    \begin{equation*}
      \mathbf{E}_{\lambda_1,\dots,\lambda_s\in\mathbf{F}_p}\mathbf{E}_{k\in\mathbf{F}_p^d}e_p\Big(\sum_{i=1}^s\lambda_i(2(h+k)^TM_ix+(h+k)^TM_i(h+k))\Big).
    \end{equation*}
    The inner average has magnitude at most $p^{-R/2}$ whenever
    $\lambda_1,\dots,\lambda_s$ are not all zero, and equals $1$ when
    $\lambda_1=\dots=\lambda_s=0$, yielding the first claim of part
    (3). The other two claims follow analogously.
    
    For parts (4) and (5) we want to estimate
    $\mathbf{E}_{w\in W_{h}\cap W_{h+k}}1_{Z}(w)e_p(-\xi\cdot w)$ for
    $\xi=0$ and $\xi\notin(W_{h,k}^0)^{\perp}$. By orthogonality of
    characters, the average equals $p^{2s}$ times
    \begin{align*}
      \mathbf{E}_{\substack{\lambda_i,\eta_i,\zeta_i\in\mathbf{F}_p \\ i\in[s]}}\mathbf{E}_{w\in \mathbf{F}_p^s}e_p\Big(&\sum_{i=1}^s\lambda_iQ_i(w)+\sum_{i=1}^s\eta_i(2h^TM_iw+h^TM_ih+z_i^Th)\\
      &+\sum_{i=1}^s\zeta_i(2(h+k)^TM_iw+(h+k)^TM_i(h+k)+z_i^T(h+k))-\xi\cdot w\Big).
    \end{align*}
    By Lemma~\ref{lem:GaussSum}, the inner average has magnitude
    at most $p^{-R/2}$ whenever $\lambda_1,\dots,\lambda_s$ are not
    all zero. If $\lambda_1=\dots=\lambda_s=0$, then the inner
    average is zero when $\xi\notin(W_{h,k}^0)^\perp$ or $\xi=0$ and
    $\eta_1,\zeta_1,\dots,\eta_s,\zeta_s$ are not all zero and equals
    $1$ when $\xi=0$ and $\eta_1=\zeta_1=\dots=\eta_s=\zeta_s$. Part
    (4) now follows in the same manner as parts (1) and (2).

    For part (5), let $(h,k)\in Y$ and
    $W'=W_{h,k}\cap\{x\in\mathbf{F}_p^d:\xi\cdot x=c\}$ for some
    vector $\xi\notin (W_{h,k}^0)^{\perp}$ and
    $c\in\mathbf{F}_p$. Then, $|Z_{h,k}\cap W'|$ equals $p^{d-2s}$
    times
    \begin{equation*}
      \mathbf{E}_{w\in W_{h,k}}1_{Z}(w)1_{W'}(w)=\frac{|Z_{h,k}|}{p^{d-2s-1}}+\frac{1}{p}\sum_{0\neq\eta\in\mathbf{F}_p}\widehat{1_Z}(\eta\xi)e_p(-\eta c)\geq p^{-s-1}-p^{4s+1-R/2}-p^{2s-R/2}
    \end{equation*}
    using the bounds obtained in the proof of part (4), and so the desired
    conclusion follows.

    Since $\codim{W_h^0}=s$ and $\codim{W_{h,k}^0}=2s$, respectively,
    exactly when $M_1^Th,\dots,M_s^Th$ and
    $M_1^Th,\dots,M_s^Th,M_1^Tk,\dots,M_s^Tk$ are linearly
    independent, part (6) of the lemma follows from the same argument
    used to prove Lemma~\ref{lem:hkl}.
  \end{proof}

  Finally, we will need the 99\% inverse theorem for what are
  essentially $1$-cocycles with coefficients in the space of functions
  $\mathbf{F}_p^d\to\mathbf{F}_p^{n\times s}$.
  \begin{lemma}\label{lem:cocycin}
    Let $0<\varepsilon\leq 1$, and suppose that
    $f:\mathbf{F}_p^d\times\mathbf{F}_p^d\to \mathbf{F}_p^{n\times s}$ satisfies
    \begin{equation*}
      f(x,y+z)=f(x,y)+f(x+y,z)
    \end{equation*}
    for at least a $(1-\varepsilon)$-proportion of
    $(x,y,z)\in\mathbf{F}_p^d\times\mathbf{F}_p^d\times\mathbf{F}_p^d$. Then,
    there exists $F:\mathbf{F}_p^d\to\mathbf{F}^{n\times s}_p$ such
    that $f(x,y)=F(x+y)-F(x)$ for at least a
    $(1-\varepsilon)$-proportion of
    $(x,y)\in\mathbf{F}_p^d\times\mathbf{F}_p^d$.
  \end{lemma}
  \begin{proof}
    Making the change of variables $y\mapsto y-x$, we get that
    \begin{equation}\label{eq:cov}
      f(y,z)=f(x,y+z-x)-f(x,y-x)
    \end{equation}
    for at least a $(1-\varepsilon)$-proportion of triples
    $(x,y,z)\in\mathbf{F}_p^d\times\mathbf{F}_p^d\times\mathbf{F}_p^d$. So,
    by the pigeonhole principle, there exists $x\in\mathbf{F}_p^d$ for
    which~\eqref{eq:cov} holds for at least a
    $(1-\varepsilon)$-proportion of pairs
    $(y,z)\in\mathbf{F}_p^d\times\mathbf{F}_p^d$. The conclusion of
    the lemma now follows by setting $F(w):=f(x,w-x)$.
  \end{proof}

  Now, we can prove Lemma~\ref{lem:quadext}.

  \begin{proof}[Proof of Lemma~\ref{lem:quadext}]
    We begin by applying Lemma~\ref{lem:inv} to $\partial_h\psi$. Let
    $h\in X$. By part~(2) of Lemma~\ref{lem:varietyfacts}, $Z_h$ has
    density $\alpha(h)=p^{-s}+\theta(h)p^{s-R/2}$ with
    $|\theta(h)|\leq 1$ and is $p^{s/2-R/4}$-uniform relative to
    $W_h$. By the lower bound assumption on $R$, we certainly have
    that $Z_h$ is $\delta(h)$-uniform relative to $W_h$ for
    $\delta(h)<(\alpha(h)\varepsilon)^{32}$, and by hypothesis, we have
    that $\partial_h\psi$ is $(1-\sqrt{\varepsilon})$-approximately
    linear on $Z_h$ for at least a $(1-\sqrt{\varepsilon})$-proportion
    of $h\in \mathbf{F}_p^d$. Thus, by part (5) of
    Lemma~\ref{lem:varietyfacts}, there exists $X'\subset X$ of
    density at least $(1-\sqrt{\varepsilon}-p^{d-R})$ in
    $\mathbf{F}_p^d$ such that $\partial_h\psi$ is
    $(1-\sqrt{\varepsilon})$-approximately linear on $Z_h$ and such
    that $Z_h$ is $\delta(h)$-uniform relative to $W_h$ for
    $\delta(h)<(\alpha(h)\varepsilon)^{32}$. It follows from
    Lemma~\ref{lem:inv2} that, for all $h\in X'$, there exists an
    $n\times d$ matrix $L[h]$ with coefficients in $\mathbf{F}_p$ and
    $f(h)\in G$ such that
    \begin{equation*}
      \partial_h\psi(x)=L[h]x+f(h)
    \end{equation*}
    for at least a $(1-16\sqrt[8]{\varepsilon})$-proportion of
    $x\in Z_h$. Denote the set of such $x$ by $Z'_h$, and extend $L$
    and $f$ to functions on all of $\mathbf{F}_p^d$ by picking their
    values on $\mathbf{F}_p^d\setminus X'$ arbitrarily from
    $\mathbf{F}_p^{n\times d}$ and $G$, respectively.

    Now, for any fixed $h\in X'$, we have by part (3) of
    Lemma~\ref{lem:varietyfacts} and linearity of expectation that the
    expected size of the intersection of $Z_h'$ with $Z_{h,k}$ is at
    least $(1-16\sqrt[8]{\varepsilon})|Z_h|(p^{-s}-p^{-R/2})$, which,
    by parts (2), (4), and (6) of Lemma~\ref{lem:varietyfacts} and the
    lower bound on $R$, is at least
    $(1-32\sqrt[8]{\varepsilon})|Z_{h,k}|$. Thus, by Markov's
    inequality,
    $|Z'_h\cap Z_{h,k}|\geq (1-32\sqrt[16]{\varepsilon})|Z_{h,k}|$ for
    at least a $(1-\sqrt[16]{\varepsilon})$-proportion of
    $k\in\mathbf{F}_p^d$. Similarly, for any fixed $k\in X'$,
    $|(Z'_k-h)\cap Z_{h,k}|\geq (1-32\sqrt[16]{\varepsilon})|Z_{h,k}|$
    for at least a $(1-\sqrt[16]{\varepsilon})$-proportion of
    $h\in\mathbf{F}_p^d$ and, for any fixed $l\in X'$,
    $|Z_{l}\cap Z_{h,k}|\geq(1-32\sqrt[16]{\varepsilon})|Z_{h,k}|$ for
    at least a $(1-\sqrt[16]{\varepsilon})$-proportion of
    $(h,k)\in\mathbf{F}_p^d\times\mathbf{F}_p^d$ with $h+k=l$. It
    follows that, for at least
    $(1-32\sqrt[16]{\varepsilon})$-proportion of pairs
    $(h,k)\in\mathbf{F}_p^d\times\mathbf{F}_p^d$, we have
    $h,k,h+k\in X'$ and that $\partial_h\psi(x)=L[h]x+f(h)$,
    $\partial_k\psi(x+h)=L[k](x+h)+f(k)$, and
    $\partial_{h+k}\psi(x)=L[h+k]x+f(h+k)$ for at least a
    $(1-128\sqrt[16]{\varepsilon})$-proportion of $x\in
    Z_{h,k}$. Denote the set of such pairs $(h,k)$ by $X'$.
    
    Thus, by the cocycle equation
    $\partial_{h+k}\psi(x)=\partial_h\psi(x)+\partial_k\psi(x+h)$, we
    have that
    \begin{equation}\label{eq:Leq}
      (L[h+k]-L[h]-L[k])x=f(h)+f(k)-f(h+k)+L(k)h
    \end{equation}
    for at least a $(1-128\sqrt[16]{\varepsilon})$-proportion of
    $x\in Z_{h,k}$ whenever $(h,k)\in X''$. Note that the set of
    $x\in W_{h,k}$ for which~\eqref{eq:Leq} holds is an affine
    subspace of $W_{h,k}$. If this affine subspace were a proper
    subset of $W_{h,k}$, then its complement in $W_{h,k}$ must contain
    an affine subspace of codimension $1$ in $W_{h,k}$, which is
    impossible when $(h,k)\in Y$ by part (4) of
    Lemma~\ref{lem:varietyfacts}, the upper bound assumption on
    $\varepsilon$, and the lower bound assumption on
    $R$. Thus,~\eqref{eq:Leq} must hold for all $x\in W_{h,k}$. Since
    $W_{h,k}$ is a coset of $W_{h,k}^0$, for all
    $(h,k)\in Y':=X''\cap Y$ we have
    \begin{equation}
      \label{eq:Leq2}
      (L[h+k]-L[h]-L[k])x=0
    \end{equation}
    for every $x\in W_{h,k}^0$.

    Whenever $(h,k)\in Y'$, Lemma~\ref{lem:span} says that the rows of
    the matrix
    \[
      d^1L[h,k]:=L[h+k]-L[h]-L[k]
    \]
    must lie in $\spn(h^TM_1,\dots,h^TM_s,k^TM_1,\dots,k^TM_s)$. Thus,
    denoting the matrix with rows $h^TM_1,\dots,h^TM_s$ by $M(h)$, we can
    write
    \begin{equation*}
    d^1L[h,k]=
    D(h,k)M(h)+E(h,k)M(k)
  \end{equation*}
  for some $n\times s$ matrices
  $D(h,k),E(h,k)\in\mathbf{F}_p^{n\times s}$ for all $(h,k)\in Y'$
  (when $s=0$, this simply says that $d^1L[h,k]=0$ for all
  $(h,k)\in Y'$, and the rest of this paragraph and the next can
  be skipped). Since $d^1L[h,k]$ is a $2$-coboundary, it is a
  $2$-cocycle, meaning that it satisfies~\eqref{eq:2cocycleeqn} for
  all $h,k,l\in\mathbf{F}_p^d$. It follows that, whenever
  $(k,l),(h+k,l),(h,k+l),(h,k)\in Y'$,
  \begin{align*}
    &D(k,l)M(k)+E(k,l)M(l)-D(h+k,l)M(h+k)-E(h+k,l)M(l)\\
    &+D(h,k+l)M(h)+E(h,k+l)M(k+l)-D(h,k)M(h)-E(h,k)M(k)=0.
  \end{align*}
  So, if, in addition,
  $M_1^Th,M_1^Tk,M_1^Tl,\dots,M_d^Th,M_d^Tk,M_d^Tl$ are linearly
  independent, then, by looking at each row of the left-hand side above, we see that
    \begin{equation}\label{eq:A}
      -D(h+k,l)+D(h,k+l)-D(h,k)=0,
    \end{equation}
    \begin{equation}
      \label{eq:AB}
      D(k,l)-D(h+k,l)+E(h,k+l)-E(h,k)=0,
    \end{equation}
    and
    \begin{equation}
      \label{eq:B}
      E(k,l)-E(h+k,l)+E(h,k+l)=0
    \end{equation}
    must all hold. Thus, by the lower bound on $R$, part (5) of
    Lemma~\ref{lem:varietyfacts}, and Lemma~\ref{lem:hkl}, the
    equations \eqref{eq:A}, \eqref{eq:AB}, and \eqref{eq:B} must
    simultaneously hold for at least a
    $(1-16\sqrt{\varepsilon})$-proportion of triples
    $(h,k,l)\in\mathbf{F}_p^d\times\mathbf{F}_p^d\times\mathbf{F}_p^d$.

    Applying Lemma~\ref{lem:cocycin} (for~\eqref{eq:A} and~\eqref{eq:B}) now produces
    $D',E':\mathbf{F}_p^d\to\mathbf{F}_p^{n\times s}$ such that
    \begin{equation*}
      D(x,y)=D'(x+y)-D'(x)\qquad\text{and}\qquad E(x,y)=E'(x+y)-E'(y)
    \end{equation*}
    for at least a $(1-32\sqrt{\varepsilon})$-proportion of pairs
    $(x,y)\in\mathbf{F}_p^d\times\mathbf{F}_p^d$. Plugging this
    into~\eqref{eq:AB} then implies that $(D'-E')$ is a
    $(1-128\sqrt{\varepsilon})$-approximately linear function
    $\mathbf{F}_p^d\to \mathbf{F}_p^{n\times s}$. Thus, by
    Corollary~\ref{cor:99afflin}, there exists a homomorphism
    $\Phi:\mathbf{F}_p^d\to\mathbf{F}_p^{n\times s}$ and an
    $n\times s$ matrix $N$ such that $E'(x)=D'(x)+\Phi(x)+N$ for at
    least a $(1-2^{10}\sqrt[4]{\varepsilon})$-proportion of
    $x\in\mathbf{F}_p^d$.

    The upshot is that, for at least a
    $(1-2^{11}\sqrt[4]{\varepsilon})$-proportion of pairs
    $(h,k)\in\mathbf{F}_p^d\times\mathbf{F}_p^d$,
    \begin{equation*}
      d^1L[h,k]=(D'(h+k)-D'(h))M(h)+(D'(h+k)-D'(k)+\Phi(h))M(k),
    \end{equation*}
    i.e., that
    \begin{equation*}
      d^1L'[h,k]=\Phi(h)M(k),
    \end{equation*}
    where $L'[h]:=L[h]-D'(h)M(h)$. By the symmetry of $d^1L'[h,k]$ in
    $h$ and $k$, there exists $Y''\subset Y'$ of size at
    least $(1-2^{13}\sqrt[4]{\varepsilon})p^{2d}$ such that
    $\Phi(h)M(k)=\Phi(k)M(h)$ for all $(h,k)\in Y''$. However, the
    rows of $\Phi(h)M(k)$ are linear combinations of the rows of
    $M(k)$ and the rows of $\Phi(k)M(h)$ are linear combinations of
    the rows of $M(h)$, while $h^TM_1,k^TM_1,\dots,h^TM_s,k^TM_s$ are
    linearly independent whenever $Y''\subset Y$. Hence,
    $d^1L'(h,k)=0$ for all $(h,k)\in Y''$, and thus for at least a
    $(1-2^{13}\sqrt[4]{\varepsilon})$-proportion of
    $(h,k)\in\mathbf{F}_p^d\times\mathbf{F}_p^d$.

    By Theorem~\ref{thm:BLR}, it follows that there exists a linear
    map $\Psi:\mathbf{F}_p^d\to\mathbf{F}_p^{n\times d}$ such that
    $L'[h]=\Psi[h]$ for at least a
    $(1-2^{14}\sqrt{\varepsilon})$-proportion of
    $h\in\mathbf{F}_p^d$. Recalling our definition of $L'$, this means
    that, for at least a $(1-2^{15}\sqrt{\varepsilon})$-proportion of
    $h\in\mathbf{F}_p^d$, we have
    $\partial_h\psi(x)=[\Psi[h]+D'(h)M(h)]x+f(h)$ for at least a
    $(1-16\sqrt[8]{\varepsilon})$-proportion of $x\in Z_h$. Since
    $M(h)x=0$ for all $x\in W_h^0$, we have that $D'(h)M(h)x$ is
    constant on $Z_h$. Thus, there exists $f':\mathbf{F}_p^d\to G$
    such that, for at least a
    $(1-2^{15}\sqrt{\varepsilon})$-proportion of $h\in\mathbf{F}_p^d$,
    \begin{equation*}
      \partial_h\psi(x)=\Psi[h]x+f'(h)
    \end{equation*}
    for at least a $(1-16\sqrt[8]{\varepsilon})$-proportion of
    $x\in Z_h$.

    We would now like to absorb the bilinear term
    $B(h,x):=\Psi[h]x$ into $\partial_h\psi$, but we can only do
    this if $B$ is symmetric. By applying the cocycle equation for
    $\psi$ again, we get that
    \begin{equation*}
      B(h+k,x)+f'(h+k)=B(h,x)+f'(h)+B(k,x+h)+f'(k),
    \end{equation*}
    i.e., using the bilinearity of $B$,
    \begin{equation*}
      f'(h+k)-f'(h)-f'(k)=B(h,k)
    \end{equation*}
    for at least a $(1-2^{17}\sqrt[8]{\varepsilon})$-proportion of
    $h,k\in\mathbf{F}_p^d$. By the symmetry of the left-hand side of
    this equation in $h$ and $k$, it follows that $B(h,k)=B(k,h)$ for
    at least a $(1-2^{18}\sqrt[8]{\varepsilon})$-proportion of pairs
    $(h,k)\in\mathbf{F}_p^d\times\mathbf{F}_p^d$. Lemma~\ref{lem:sym}
    then says that $B$ must be symmetric.

    Define $Q:\mathbf{F}_p^d\to G$ by $Q(x):=\frac{B(x,x)}{2}$, and
    set $\psi'(x):=\psi(x)+Q(x)$ and
    $f''(h):=f'(h)+\frac{B(h,h)}{2}$. Then, for at least a
    $(1-2^{15}\sqrt{\varepsilon})$-proportion of $h\in\mathbf{F}_p^d$,
    we have $\partial_h\psi'(x)=f''(h)$ for at least a
    $(1-16\sqrt[8]{\varepsilon})$-proportion of $x\in Z_h$. Doubling
    the $x$ variable, we get that, for at least a
    $(1-2^{15}\sqrt{\varepsilon})$-proportion of $h\in\mathbf{F}_p^d$,
    $\partial_{h}\psi'(x)=\partial_{h}\psi'(x+k)$ for at least a
    $(1-32\sqrt[8]{\varepsilon})$-proportion of $x,x+k\in Z_h$. This
    means that $\psi'$ is
    $(1-2^{16}\sqrt[8]{\varepsilon})$-approximately linear on
    $Z$. Since $Z$ is $p^{-R/4}$-uniform of density
    $p^{-s}+\theta p^{-R/2}$ in $\mathbf{F}_p^d$ for some
    $|\theta|\leq 1$, it now follows from Lemma~\ref{lem:inv} and the
    lower bound on $R$ that there exists an affine-linear
    $\Phi:\mathbf{F}_p^d\to G$ such that $\psi'(x)=\Phi(x)$ for at
    least a $(1-2^{10}\sqrt[16]{\varepsilon})$-proportion of $x\in
    Z$. Recalling the definition of $\phi$ completes the proof.
  \end{proof}

  \section{Finishing the proof}\label{sec:finish}
  
  Now, we can finish the proof of Theorem~\ref{thm:approxpoly}.
  \begin{proof}[Proof of Theorem~\ref{thm:approxpoly}]
    Let $R\geq C\delta^{-C}$ for $C$ the constant from
    Lemma~\ref{lem:psiapprox}, and assume that
    $n\geq 2^{100}\lceil\log_p(2/\delta^{10000})\rceil$. By
    Lemma~\ref{lem:psiapprox}, there exists a quadratic level set
    $\mathcal{A}:=(a+V)\cap\mathcal{Q}^{-1}(\mathbf{b})$ where
    $V\leq G$ is a subspace of codimension $r\ll\delta^{-O(1)}R$,
    $\mathcal{Q}=(Q_1,\dots,Q_s)$ with $s\ll \delta^{-O(1)}$ where
    $Q_1,\dots,Q_s:G\to\mathbf{F}_p$ are quadratic functions such that
    the rank of $\lambda_1Q_1|_V+\dots+\lambda_s Q_s|_{V}$ on $V$ is
    always at least $R$ when
    $\lambda_1,\dots,\lambda_s\in\mathbf{F}_p$ are not all zero,
    $a\in G$, and $\mathbf{b}\in\mathbf{F}_p^s$, along with a real
    number $\theta\gg(\delta/p)^{O(1)}$ and a map $\psi:G\to G$ that
    is $(1-2^{-40}p^{-20}\delta^{100})$-approximately quadratic on
    $\mathcal{A}$ and such that the number of $3$-dimensional cubes
    not wholly contained in the set~\eqref{eq:Sphi}, which we will
    denote by $S$, respected by $\psi$ is at most $p^{7n/2}$. By
    making a change of variables, the function $\psi':G\to G$ defined
    by $\psi(x-a)$ is a $(1-2^{-40}p^{-20}\delta^{100})$-approximate
    quadratic on the quadratic level set
    $\mathcal{A}':=V\cap\mathcal{Q}'(\mathbf{b})$, where
    $\mathcal{Q}'=(Q_1',\dots,Q_s')$ with $Q_i'(x)=Q_i(x-a)$. By
    fixing a basis and identifying $V$ with $\mathbf{F}_p^{n-r}$, as
    long as $R\geq 2^{40}(s+\log_p(2/\delta^{100}))$, say, we can
    apply Lemma~\ref{lem:quadext} to obtain a quadratic map $Q:V\to G$
    such that $\psi(x)=Q(x)$ for at least a
    $(1-2^{11}\delta^{8})$-proportion of $x\in\mathcal{A}$.

    Extend $Q$ to a quadratic map defined on all of $G$ and fix
    $R:=2^{40}\lceil s+\log_p(2/\delta^{100})\rceil$. We now have that
    $\psi(x)=Q(x)$ for at least $c_\delta p^n$ elements $x\in G$,
    where $c_\delta\gg p^{-1/\delta^{O(1)}}$ by the upper bounds on
    $r$ and $s$ and Lemma~\ref{lem:atomsize}. Arguing as in the proof
    of Lemma~\ref{lem:SfQ}, we will show that, in fact $\psi(x)=Q(x)$
    for at least $\frac{c_\delta}{2}p^n$ elements of $S$. Indeed, if
    $\psi(x)=Q(x)$ for more than $\frac{c_\delta}{2}$ elements of
    $S':= G\setminus S$, then by the Gowers--Cauchy--Schwarz
    inequality in the form of Theorem~\ref{thm:GCSnorm}, $\psi$ must
    respect more than $(\frac{c_\delta}{2})^8p^{4n}$ $3$-dimensional
    cubes not contained in $S'$, which forces
    $(\frac{c_\delta}{2})^8p^{4n}\leq p^{7n/2}$. This yields a
    contradiction as long as $p^n>(\frac{2}{c_\delta})^{16}$.

    Thus, as long as $p^n>(\frac{2}{c_\delta})^{16}$, we have that
    $\psi(x)=Q(x)$ for at least $\frac{c_\delta}{2}|G|$ elements of
    $S$. By the definition of $S$, it follows that
    \begin{equation*}
      \sum_{\mathbf{0}\neq\omega\in\{0,1\}^3}(-1)^{|\omega|}\phi(x+\omega\cdot(h_1,h_2,h_3))=Q(x)
    \end{equation*}
    for $\gg\delta^{O(1)}c_\delta|G|^4$ quadruples
    $(x,h_1,h_2,h_3)\in G^4$. Setting $\phi'(x)=\phi(x)+Q(x)$ and
    using that $Q$ is a quadratic map, we obtain that
    \begin{equation*}
      \sum_{\mathbf{0}\neq\omega\in\{0,1\}^3}(-1)^{|\omega|}\phi'(x+\omega\cdot(h_1,h_2,h_3))=0
    \end{equation*}
    for $\gg \delta^{O(1)}c_\delta|G|^4$ choices of
    $x,h_1,h_2,h_3\in G$. Making the change of variables
    $h_3\mapsto h_3-x$ above and using the Cauchy--Schwarz
    inequality to double the $x$ variable yields
    \begin{equation*}
      \partial_k\phi'(x+h_1)+\partial_k\phi'(x+h_2)-\partial_k\phi'(x+h_1+h_2)=0
    \end{equation*}
    for $\gg\delta^{O(1)}c_\delta^2|G|^4$ choices of
    $x,h_1,h_2,k\in G$. Making the change of variables
    $h_1\mapsto h_1-x$ and applying the Cauchy--Schwarz inequality
    again to double the $x$ variable reveals that $\phi'$ is
    $\Omega(\delta^{O(1)}c_\delta^4)$-approximately linear on $G$. The
    theorem now follows by applying Theorem~\ref{thm:approxlinear}. If
    $p^n\leq(\frac{2}{c_\delta})^{16}$, then the theorem follows
    trivially by taking $Q\equiv \phi(0)$.
  \end{proof}

  \bibliographystyle{plain} \bibliography{bib}

\end{document}